\documentclass[11pt]{amsart}
\usepackage{geometry}                
\usepackage{graphicx}
\usepackage{amssymb, eucal}
\usepackage{epstopdf}
\usepackage{color}

\theoremstyle{plain}
\newtheorem{theo}{Theorem}[section]
\newtheorem{lem}[theo]{Lemma}
\newtheorem{cor}[theo]{Corollary}
\newtheorem{prop}[theo]{Proposition}
\newtheorem{defi}[theo]{Definition}

\numberwithin{equation}{section}

\theoremstyle{definition}

\newtheorem{remark}[theo]{Remark}

\def\e{\varepsilon}

\def\vf{\varphi}
\def\a{\alpha}
\def\b{\beta}
\def\P{\mathfrak P}
\def\th{\theta}
\def\d{\delta}

\def\z{\zeta}
\def\D{\Delta}
\def\G{\Gamma}
\def\fg{d}
\def\g{\gamma}
\def\gg{\mathbf g}
\def\k{\kappa}
\def\DD{\mathcal D}
\def\LL{\mathcal L}
\def\pp{\partial}
\def\l{\lambda}

\def\rr{T}
\def\s{\sigma}
\def\Si{\Sigma}
\def\m{m}
\def\x{\times}
\def \R{\mathbb R}
\def \N{{\mathbb N}}
\def\E{\mathbb E}
\def \Z{\mathbb Z}

\def \H{\mathbb H}

\def \P{\mathbb P}
\def \X{\mathbb X}
\def\K{\mathcal K}
\def\RR{\mathcal R}
\def\cT{\mathcal T}
\def\B{\Omega}
\def\ls{\lesssim}
\def\ov{\overline}
\def\un{\underline}

\def\v{\varphi}

\def\om{\omega}
\def\Om{\Omega}

\def\A{\mathcal A}
\def\AA{\mathfrak m}

\def\aa{x_i}
\def\ab{R}
\def\bb{x_{i+1}}
\def\C{\mathbb C}
\def\CC{\mathcal C}

\def\W{\mathcal W}
\def\wh{\widehat}
\def\wt{\widetilde}
\def\({\biggl(}
\def\){\biggr)}

\def\<{\bold\langle}
\def\>{\bold\rangle}

\def\LL{{\mathcal L}}

\def\M{\widetilde {M}}

\def\bg{{\mathbf g}}

\def\xx{x}
\def\yy{y}
\def\zz{z}
\def\ww{w}

\def\ppi{{p}}
\def\Dppi{{p}\,}
\def\Pp{\wp}
\def\Vol{\mathrm{Vol}}
\def\Ma{\widetilde{M}_2}
\def\Mb{\widetilde{M}_3}
\def\Mc{\widetilde{M}_4}
\def\Md{\widetilde{M}_5}
\def\Me{\widetilde{M}_1}

\def\kk{\ell}

\def\pp{\partial }
\begin{document}

\title[Local limit theorem]{Local limit theorem in negative curvature}

\author{Fran\c cois Ledrappier}
 \address{Department of Mathematics,  255 Hurley Hall, University of Notre Dame, Notre Dame IN, USA}
 \address{
 Sorbonne Universit\'e, UMR 8001, LPSM, F-75252, Paris Cedex 05 France}
 \address{
 CNRS, UMR 8001, LPSM, F-75252, Paris Cedex 05, France} \email{ledrappier.1@nd.edu}
\author{Seonhee Lim}
\address{Department of Mathematical Sciences and Research Institute of Mathematics, Seoul National University, Seoul
151-747} \email{slim@snu.ac.kr}

\date{May 26, 2020}

\begin{abstract} Consider the heat kernel $\Pp(t,x,y)$ on the universal cover $\widetilde{M}$ of a closed Riemannian manifold of negative sectional curvature. We show the local limit theorem for $\Pp$ : $$\lim_{t \to \infty} t^{3/2}e^{\lambda_0 t} \Pp(t,x,y)=C(x,y),$$
where $\lambda_0$ is the bottom of the spectrum of the geometric Laplacian and $C(x,y)$ is a positive $\l_0$-harmonic function which depends on $x, y \in \widetilde{M}$.

We show that the $\lambda_0$-Martin boundary of $\widetilde{M}$ is equal to its topological boundary.
The Martin decomposition of $C(x,y)$ gives a family of measures $\{\mu^{\l_0}_x \}$ on $\pp \M$. We show that $\{\mu^{\l_0}_x \}$ is a family minimizing the energy or the Rayleigh quotient of Mohsen.

We use the uniform Harnack inequality on the boundary $\partial \widetilde{M}$ and the uniform three-mixing of the geodesic flow on the unit tangent bundle $SM$ for suitable Gibbs-Margulis measures. 

\end{abstract}

\maketitle

\section{Introduction}
Let $(M, \fg)$ be an $m$-dimensional closed connected Riemannian
manifold of negative sectional curvature, and $(\wt{M}, \wt{\fg})$ its universal cover
endowed with the lifted Riemannian metric. Let us denote by $d$ the distance on $M$, $\M$, as well as on their unit tangent bundles $SM$ and $S\M$ (see \cite{PPS} for various distances on $M$ and on $SM$ and the equivalences between them). Let us denote by $\pi : SM \to M$ and $\pi : S\M \to \M$ the projection of each vector to its base point and by $\Dppi$
the natural projection $(\wt{M}, \wt{\fg})\to (M, \fg)$ and its derivative.  The fundamental group
$\G=\pi_1(M)$ acts on $\wt{M}$ as isometries such that $M=\wt{M}/\G$. {Let $M_0$ be  a bounded fundamental domain for this action.} 

We consider the {\it{geometric Laplace}} operator $\Delta:=-\mbox{Div}\nabla$ or smooth
functions on $\wt{M}$ and the corresponding heat kernel function
$\Pp(t, x, y), t\in \mathbb R_{+}, x, y\in \wt{M}$, which is the probability density defined as the
fundamental solution of the heat equation, {i.e. the function which satisfies $\frac{\partial
\Pp}{\partial t}+\Delta_y \Pp =0 $ and $\underset{t \to 0}{\lim} \; \Pp(t,x,y) =\delta(x-y)$}. The function $\Pp$
 is clearly $\G$-invariant and symmetric in $x$ and $y$. See Section~\ref{Appendix2} for background on general potential theory and properties of the heat kernel. 
 
Denote by $\l _0 $ the bottom of the  spectrum 
of the operator  $\D$ on $L^2( \wt M, \mbox{Vol} )$, where $d \mbox{Vol}(z)$ is the Riemannian
volume form on $\wt{M}$ (see Definition~\ref{lambda0}). 
Since $\G$ is not amenable, $\l _0 $ is positive \cite{Br}. For all 
$x, y  \in \wt{M}$, we have
\begin{equation}\label{bottom} 
\l_0 \; = \; \lim _{t\to \infty } -\frac{1}{t} \log \Pp(t, x,y)
\end{equation} by {the} spectral theorem 
(See \cite{CK} and  \cite{Si}).
Our main result is a local limit theorem which refines (\ref{bottom}).
\begin{theo}[Local Limit Theorem]\label{locallimit} There exists a positive function $C$ on $\wt{M} \x \wt{M}$ such that {for all $x,y \in M$,}
\begin{equation}\label{locallimitequation} \lim _{t \to \infty } t^{3/2}e^{\l_0 t} \Pp(t, x, y) \; = \; C(x,y) . \end{equation}
\end{theo}

When $\wt{M} $ is the hyperbolic space $\H^3$, there is an explicit expression for $\Pp(t, x, y) $ (\cite{DGM}) and Theorem \ref{locallimit} is clear, with \[ C(x,y ) =  (4\pi )^{-3/2} \frac{d(x,y)} {\sinh d(x,y)} .\]  

In the case of symmetric spaces of non-compact type, i.e. when $\wt{M}=G/K$ for a semi-simple Lie group $G$ and a maximal compact subgroup $K$ of $G$, Bougerol proved an analog of Theorem \ref{locallimit}  with $t^{k/2}$ instead of $t^{3/2}$, where the integer $k$ is given by the rank plus twice the number of positive indivisible roots. In particular, $k =3$ for all rank one symmetric spaces and this explains why one might expect $t^{3/2}$ for negatively curved manifolds. Bougerol proved the theorem for all random walks on $G$ with a distribution that is left and right $K$-invariant which implies the same result for Brownian motions on $\M$. 

The limit function $C(x,y) $ is symmetric by Theorem \ref{locallimit} and it is a positive harmonic function in $y$ for the operator $(\Delta-\l_0)$: 
$$(\D- \l_0) C(x,y) =0.$$
From now on, we will call such a harmonic function for $(\Delta-\l_0)$ \emph{a $\l_0$-harmonic function}.
We further give a formula in Theorem \ref{Phi_zero} below. We remark that
it was already known that  if the limit  \begin{equation}\label{Davies} \lim_{t\to \infty} \frac{\Pp(t,x,y)}{\Pp(t,x,x)} \; = \; \frac{C(x,y)}{C(x,x)} \end{equation} exists on a Riemannian manifold, then $C(x,y)$ is a  $\l_0 $-harmonic function in $y$  \cite{ABJ} (Theorem 1.2). It is indeed a conjecture by Davies (\cite{Da}) that the limit  (\ref{Davies}) always exists (see  \cite{Ko} for a recent counterexample for the analogous question on graphs). Our result can be stated as:
\begin{cor} The universal cover of a compact Riemannian manifold with negative sectional curvature satisfies Davies conjecture. \end{cor}
See \cite{ABJ} for further discussion and applications of Davies conjecture.

 A local limit theorem similar to Theorem~\ref{locallimit} was first observed by Gerl \cite{Ge} and  Woess \cite {GW} for random walks on a free group which are supported on a finite set of generators of the group. It was then proven by Lalley for random walks with finite support  on a {finitely generated} free group \cite{La}. This was extended by Gou\"ezel and Lalley  to symmetric random walks with finite support on {cocompact} Fuchsian groups \cite{GL} and finally by Gou\"ezel to symmetric random walks with finite support on hyperbolic groups \cite{G1}. Our proof follows the strategy  and ideas of \cite{GL} and \cite{G1}.  By \cite{G2}, this general strategy works for measures of infinite support and with superexponential moments.
 
Two main new ingredients of the proof of Theorem~\ref{locallimit} are the uniform rapid-mixing of the geodesic flow generalizing Dolgopyat theorem and the generalised Patterson-Sullivan conformal family whose Radon-Nikodym derivative is the Martin kernel $k^2_{\l_0}(x,y,\xi)$, which is defined in Theorem~\ref{thm:1.4} below and which is a family realizing the minimum of Mohsen's Rayleigh quotient (see Corollary~\ref{Mohsen}).
 
 As in \cite{G1}, we obtain several subsequent results which have their own interest. 
 Let us introduce more notation to describe these results. For any real $\l < \l_0$, we define the {\it $\l$-Green function} $G_\l $:  for all 
$x \not = y  \in \wt{M}$, 
\[ G_\l (x, y) \; := \; \int_0^\infty e^{\l t} \Pp(t,x,y) dt .\] 

The integral on the right hand side is finite: it converges at $\infty $ thanks to the spectral theorem (\ref{bottom}) and it converges at 0 since as $t \to 0 $, $\Pp(t, x, y) \sim C/ t^{m/2} e^{- \frac{d^2(x,y) }{4t}},$ which can be deduced from the fact that as $t \to 0$, the ambient space can be approximated by Euclidean space. The function $G_\l(x,\cdot)$ is positive and $\l$-harmonic for all $y \neq x$.

We first observe in Lemma~\ref{lem:1} that
for all
$x \not = y  \in \wt{M}$,  the integral
\[ G_{\l_0} (x, y) \; := \int_0^\infty e^{\l_0 t} \Pp(t,x,y) dt \] 
is finite.
{In Section~\ref{section3}, we show (see Proposition~\ref{prop:expdecay}, where we relate  $\tau$ with other dynamical properties)
\begin{theo}\label{expdecay} There are positive constants $\tau$ and $C$ such that, for $x,y \in \M$ with $ d(x,y) \geq 1,$
\[ G_{\l _0} (x,y) \; \leq \;  C e^{-\tau d(x,y)} .\] \end{theo}	}

Two geodesic rays  in $\M$ are said to be {\it equivalent} if they remain a
bounded distance apart. The {\it geometric boundary} $\pp \M$ is defined as the space of equivalence  classes of unit
speed geodesic rays. A sequence $\{ y_n \}_{n \in \N} $ in $\M$
converges to a point  in $\pp \M$ if, and only if, for some (hence, for all) $x \in \M$, 
\[ d(x, y_n) + d(x, y_m ) - d(y_n, y_m ) \to \infty \quad {\textrm {as }} \; n,m \to \infty .\]
We now describe the Martin boundary of the operator $\D - \l_0 $. 
 The Martin boundary of $\D - \l_0 $ is the closure of the embedding $y \to{ k_{\l_0 }} (\cdot, y) = \frac{{G_{\l_0}}(\cdot, y)}{G_{\lambda_0}(o, y)}$ in the space of functions with the topology of pointwise convergence.  It is crucial for us to identify the Martin boundary of $\Delta -\lambda_0$ with the geometric boundary when we use thermodynamics formalism for the measures on the Martin boundary to obtain the Local Limit Theorem \ref{locallimit}.

\begin{theo}\label{thm:1.4}[$\l_0$-Martin boundary]
Fix $x \in \M$ and assume that the sequence  $\{ y_n \}_{n \in \N} $ converges to a point $\xi \in \pp \M$. Then, there exist a positive $\l _0$-harmonic function $k_{\l_0} (x, y, \xi) $ of the Laplacian, which we call the Martin kernel, such that 
\[ \lim\limits_{n \to \infty } \frac{G_{\l_0 } (y, y_n)}{G_{\l_0 } (x, y_n)} \; = \; k_{\l_0} (x, y, \xi). \] 
 Moreover, the Martin boundary of $\D - \l_0$ coincides with the geometric boundary. In particular,
for any positive $\l_0$-harmonic function $F$ and any $x \in \M$, there is a finite measure $\nu _{x,F}$ on $\pp \M$ such that 
 \[ F(y) \; = \int _{\pp \M} k_{\l_0} (x, y, \xi) d\nu _{x,F}(\xi ) .\] 
 \end{theo}

  See Section~\ref{section3} for the proof and more properties of the Martin kernel $k_{\l_0} (x, y, \xi)$.  
  The Martin kernel squared $k_{\l_0} ^2(x,y,\xi)$ plays the role of a conformal density for a family of measures on the boundary $\pp \M$.
 
 \begin{theo}\label{PattersonSullivan}
There is a family $\{\mu _x^{\l_0} \}_{x \in \M}$ of  finite measures on $\pp \M$ such that 
\begin{itemize} \item[1)] the family $x \mapsto \mu _x^{\l_0} $ is  $\G$-equivariant: $ \mu _{\g x}^{\l_0} = \g_\ast (\mu _x^{\l_0} )$ for $\g \in \G$ and
\item[2)] for  $\mu _x^{\l_0} $-a.e. $\xi \in \pp \M$, all $y \in \M$, we have $$ \frac{d\mu _y^{\l_0} }{d\mu _x^{\l_0} }(\xi ) \; = \; k_{\l_0}^{2} (x, y, \xi ).$$
\end{itemize}
The family is unique if we normalize by $\int _{M_0} \mu _x^{\l_0} (\pp \M )  d\mbox{Vol} (x)= 1.$
\end{theo}

Consider 
a $\G$-equivariant family $\nu = \{\nu _x\}_{x \in \M} $ of measures on $\pp \M$ with cocycle $\kk(x,y, \xi) := \frac{d\nu_y}{d\nu_x}(\xi) $ and normalized by $\int _{M_0} \nu _x^{\l_0} (\pp \M )  d\mbox{Vol} (x)=1$. Assume that for $\nu$-a.e. $\xi$, the function $ y \mapsto \log {\kk(x,y, \xi )} $ is a Lipschitz continuous function on $\M$ so that the value $\| \nabla_{y} \log {\kk(x,y,\xi)} \|$, which is independent of $x$, is defined for almost every $(x,y,\xi)$
\footnote{The value of $\| \nabla_{y} \log {l(x,y,\xi)} \|$ is defined for a.e. $(x,y,\xi)$.
Indeed, $ \log {l(x,y,\xi)} $ is defined for $\nu$ a.e. $\xi$ and, if we assume the function to be Lipschitz continuous, then its gradient exists for Lebesgue a.e. $y$, by Rademacher theorem.
The value $\| \nabla_{y} \log {l(x,y,\xi)} \|$ is constant in $x$ when defined. Therefore, the set of $(x,y,\xi) $ where $\| \nabla_{y} \log {l(x,y,\xi)} \|$ is not defined is  negligible for $\Vol \x \Vol \x \nu$ and does not depend on $x$. It follows that  $ \| \nabla_{y=x} \log {l(x,y,\xi)} \|^2 $ makes sense for $\Vol \x \nu$-a.e. $(x,\xi).$}. For such a family $\nu$, we define the \emph{energy of} $\nu$ as follows:
\[ \mathcal{E} (\nu ) : = \int _{M_0} \left(\int_{\pp \M} \| \nabla_{y=x} \log {\kk(x,y,\xi)} \|^2 d\nu _{x}(\xi)  \right) d\mbox{Vol} (x),\]
We define the energy to be infinite otherwise. Since for any fixed $x_0,$ 
\begin{equation}\label{rayleigh} \| \nabla |_{y=x} \log \ell (x_0,y, \xi) \|^2 = \frac{\| \nabla |_{y=x}\ell (x_0,y,\xi) \|^2}{\ell ^2(x_0,x,\xi)}=
 4\| \nabla |_{y=x}\sqrt {\ell (x_0,y,\xi)}\|^2 \frac{d\nu _{x_0}}{d\nu _x},\end{equation}
the energy is equal to 4 times the Rayleigh quotient
\[ \RR (\nu ) : = \int _{M_0} \left(\int_{\pp \M} \| \nabla_x \sqrt {\kk(x_0,x,\xi)} \|^2 d\nu _{x_0}(\xi)  \right) d\mbox{Vol} (x)\]
defined by O. Mohsen in \cite{Mo}.  
Mohsen showed that $\l_0= \inf_{\nu} \RR(\nu)$ and asked whether the minimum is achieved. We have 
\begin{cor}\label{Mohsen}
The family $\mu ^{\l_0}_x$  achieves the minimum Rayleigh quotient.\end{cor}
See Section \ref{sec:5.2.3} for a proof.
Mohsen proved the uniqueness for the manifolds with constant negative curvature.

The family $\mu ^{\l_0}_x$ is a fourth natural $\G$-equivariant family $\nu = \nu _x $ of measures on $\pp \M$ with regular cocycles, alongside with the Lebesgue visual measures, the Margulis-Patterson-Sullivan measures and the harmonic measures. Observe that the energy of the Margulis-Patterson-Sullivan measure is the volume entropy squared, and the energy of the harmonic measure is the Kaimanovich entropy \cite{H2}, \cite{K1}, \cite{L3}. For rank one symmetric spaces, all of these families are the same up to normalization.

The last result we would like to emphasize is a formula of the function $C(x,y)$ in Theorem~\ref{locallimit}.

\begin{theo}\label{Phi_zero} Fix $x \in \M$. There is a constant $\Upsilon=\Upsilon_{\l_0}$ such that the positive  $\l _0$-harmonic function $C(x,y)$ satisfies
\begin{equation*} C(x,y) \; =  \frac{\sqrt \Upsilon}{2 \sqrt \pi}  \int_{\partial \wt M} k_{\l_0} (x,y, \xi )  d\mu^{\l_0}_x (\xi) \; = \; \frac{\sqrt \Upsilon}{2 \sqrt \pi}  \int_{\partial \wt M} \sqrt {  d\mu^{\l_0}_x (\xi)}  \sqrt {  d\mu^{\l_0}_y (\xi)} . \end{equation*}
 \end{theo}
Note that the formula for the constant $\Upsilon$ is given by (\ref{Upsilon}).

Here, $ \displaystyle \int_{\partial \wt M} \sqrt {  d\mu^{\l_0}_x (\xi)}  \sqrt {  d\mu^{\l_0}_y (\xi)} := \int _{\pp \M} \sqrt {\frac{d\mu _y^{\l_0} }{d\mu _x^{\l_0} }(\xi )}  d\mu _x^{\l_0} (\xi ) $ {as used in} unitary representation of $\G$ associated to its action on $(\pp \M, \mu ^{\l_0} ).$
In case of symmetric spaces, the function $C(x,y)$ is the positive  $\l_0$-harmonic function invariant under the stabilizer $K_x$ of the point $x$, a.k.a. the \emph{Harish-Chandra function}, or the ground state, centered at $x$.

The article is organized along the path of the proof of Theorem~\ref{locallimit}.

In Section~{\ref{sec:2}}, we {recall the consequences of Ancona's boundary Harnack inequality for $\l < \l_0$ (\cite{A}), in conjunction with the thermodynamic formalism for the geodesic flow (following \cite{K1}, \cite{H3} and \cite{L2}).  Using mixing  properties of the geodesic flow on the  unit tangent bundle $SM$ for suitable $\G$-invariant Gibbs measures, we show that there is a function $P(\l ) $ of $\l$ and a positive  function $D(x, \l) $ such that, {for $\l < \l_0$, as $R \to \infty $}} 
\begin{equation}\label{eq:renewal}  e^{-P(\l ) R } \int _{S(x, R)} G_{\l}^2 (x, z)  dz \to D(x, \l), \end{equation} 
where $P(\l)<0$ for $\l < \l_0$ and $S(x,R)$ is the sphere of radius $R$ centered at $x$
(see Proposition~\ref{renewal}).

{We also recall from \cite{H3} Corollary 5.5.1 that $\int _{S(x, R)} G_{\l_0}^2 (x, z)  dz $ is bounded independently of $R$ (Proposition \ref{Lemma2.5}).}

In Section 3, we use {this bound} to establish the uniform Harnack inequality at the boundary, i.e. the Ancona-Gou\"ezel inequality (Theorem \ref{Alano}). Theorem~\ref{thm:1.4} follows and the other applications of thermodynamic formalism hold equally at $\l = \l_0$. 

In Section 4, we discuss limits of measures on large spheres using uniform mixing of the geodesic flow.
One consequence of our results is that the measures $\mu_{x,R}$ on the spheres $S(x,R)$ with density $e^{-R P(\l_0)} G^2_{\l_0}(x,y)$  converge to some measure $\mu^{\l_0}_x$ as $R\to \infty$ (Corollary~\ref{renewal5}). The measures $\mu^{\l_0}_x$ turn out to be a $\G$-equivariant family with regular cocycle $e^{P(\l_0)\b(x,y,\xi)} k^2_{\l_0}(x,y,\xi)$, where $\b(x,y,\xi)$ is the Busemann function (see the equation (\ref{eqn:busemann})). On the other hand, for $\l \in [0, \l_0], x\in \M$ and $R>2$, we define the measure $\m _{x,\l ,R}$ on $SM$ by:
\begin{center}\label{def:mxR}
 lifting the measure $ e^{-P({\l} ) R } G_{\l}^2 (x, z)  dz $ on $S(x,R) $ to 
 the set \\ of unit vectors pointing towards $x$, then projecting to $SM$ by $\Dppi$.  \quad \quad ($*$)
 \end{center}
 Another consequence is that there exists a probability measure $\ov m $ over $SM$ such that the measures $\m _{x, \l, R}$ converge towards $\mu^{\l_0}_x (\pp \M) \ov m $ on $SM$ as $R \to \infty $ and $\l \to \l_0$ {(see Corollary~\ref{renewal4})}.

Once we prove that $P(\l _0 ) = 0 $ in Section~\ref{sec:5}, the family of measures $\mu _y^{\l_0}$ satisfies the statements of Theorem \ref{PattersonSullivan}. We also obtain that
for $x,y \in \M$, $\displaystyle  \lim\limits_{\l \to \l_0} -P({\l} ) \frac{\pp }{\pp \l } G_\l (x,y) $ is proportional to $C(x,y)$.

  By a precise study of the second derivative $\displaystyle \frac{\pp ^2 }{\pp \l^2 } G_\l (x,y) $ in {Section~\ref{sec:7}}, we obtain that both
\[ \frac{P({\l} ) }{\sqrt {\l_0 - \l } }  {\textrm { and }} \sqrt {\l_0 - \l } \frac{\pp }{\pp \l } G_\l (x,y) \] 
converge towards positive numbers as $\l \to \l_0 .$ In {Section~\ref{sec:8}}, we conclude the proof of Theorem \ref{locallimit} from Theorem~\ref{Theorem3.1}  thanks to a Tauberian Theorem as in \cite{GL}.  Theorem \ref{Phi_zero}  follows as well.

In {Section~\ref{Appendix},} we prove a uniform version of Dolgopyat's rapid-mixing for hyperbolic flows {which is an important tool for the proofs in the previous sections. As its proof is independent of the rest of the sections and the result is of independent interest as well, we made an Appendix for it.}
In Section~\ref{Appendix2}, for completeness, we prove the precise balayage estimates in the form that is used in the article.

\begin{remark}In this text, $C$ stands for a number depending only on the geometry of $M$ and $\G$. However, its actual value may change from one formula to another. For the sake of clarity, we specify $C_0, \cdots, C_{11}, C_\e, C(T)$ when the same number is used in another computation. Note that $C_1, C_6, C_7$ in Section~\ref{Appendix} have the same role as in \cite{Me}.
Likewise, we consider spaces of $\a$-H\"older continuous functions for some $\a$ of which the actual value may vary. 
Let us also remark that when the constant changes from one line to another, we used the symbols $\simeq$ and $\lesssim$ to indicate that the constant  has changed.
\end{remark}
Acknowledgement : We would like to thank M. Pollicott for generously sharing his insights and ideas \cite{P1}, \cite{P2}, P. Bougerol for his interest and the \cite{ABJ} reference and S. Gou\"ezel for helpful comments. We are very grateful to  several referees for their many precise and thoughtful remarks. The work was supported by University of Notre Dame, Seoul National University and MSRI during our visits. The second author was supported by  NRF-2013R1A1A2011942, SSTF-BA1601-03 and Korea Institute for Advanced Study (KIAS).

\section{Potential theory and thermodynamic formalism}\label{sec:2}

We recall in this section the results obtained by applying classical potential theory to the Laplacian on $\wt M$ and thermodynamic formalism to the geodesic flow.
See Section~\ref{Appendix2} for general potential theory. We have $G_{\l_0} (x,y) = \int_0^\infty e^{\l_0 t} \Pp(t,x,y) \, dt,$ where $\l_0$ is defined in Definition~\ref{lambda0}.
\begin{lem}\label{lem:1} 
For any $x \neq y$,
\begin{equation}\label{Gfinite}
G_{\l_0} (x,y) < \infty.
\end{equation}
For any $x$ and any compact set $K \subset \M$ with non-empty interior, we have
\begin{eqnarray}\label{finiteintegral}
\label{eqn2.3}
 \int _K  G_{\l_0} (x, y) \, d\Vol (y) &< & \infty .\end{eqnarray}
\end{lem}

\begin{proof}
The following argument is inspired by an idea of Guivarc'h in case of Lie groups.
Let $\phi$ be a positive $\l_0$-harmonic function of the Laplacian, i.e. $\Delta \phi = \l_0 \phi$, which exists by Lemma~\ref{lem2.2} (1). 
Then 
$q(t,x,y)$ defined in \eqref{eqn2.1} defines a Markov process $D$ with its Green function
$G_D(x,y) = G_{\l_0}(x,y) \frac{\phi(y)}{\phi(x)}.$

Suppose on the contrary to \eqref{eqn2.3} that there is a compact set $K$ with non-empty interior such that $\int _K  G_{\l_0} (x, y) \, d\Vol (y) = \infty $. It implies that 
$\int _K  G_{D} (x, y) \, d\Vol (y) = \infty $. By the proof of Theorem 4.2.1.(ii) of \cite{Pi},
$G_D(x,y) = \infty$, which implies $G_{\l_0}(x,y) = \infty$, for all $y$.
By Lemma~\ref{lem2.2} (2), there is a unique $\l_0$-harmonic function $\phi$ up to multiplicative constant. It follows that $\phi(y)/\phi(x)$ is $\G$-invariant, 
thus $G_D$ is $\G$-invariant. By discretization (see the proof of the main theorem of \cite{BL}) there is a recurrent random walk $\mu_D$ on $\Gamma$ with Green function $G_D$, which implies that $\G$ is virtually $\mathbb{Z}, \mathbb{Z}^2$ or trivial \cite{V}, which is a contradiction. Thus $G_{\l_0} (x,y) < \infty$ for some $y\neq x$.

Equation \eqref{Gfinite} follows from Equation \eqref{eqn2.3} 
since if $G_{\l_0} (x,y) < \infty $ at some points $y \not = x,$ then $G_{\l_0} (x,y) < \infty $ at all  points $y \not = x$ (see \cite{Da}, Theorem 13). \end{proof} 

\begin{prop}\label{prop:derivative} We have, for $\l \in [0, \l_0) $, for any two points $x \not = y  \in \M$: 
\begin{equation} \label{derivative} 
 \frac {\pp ^k }{\pp \l^k} G_\l (x,y) \; =\; k! \int _{\M ^k} G_\l (x, x_1) G_\l (x_1, x_2) \cdots G_\l (x_k, y ) \, d{\textrm{Vol}}^k (x_1, \cdots, x_k ) . 
\end{equation}
 \end{prop}

\begin{proof} It follows from computation (see \cite{GL} Proposition 1.9). For example, for $k=1$,
\begin{align*}
 \int_{\wt M} G_\l (x,z) G_\l (z,y) dz &=  \int_0 ^\infty  \int_0 ^\infty \int_{\widetilde{M}} e^{\l (t+ u)} \Pp(t,x,z) \Pp(u,z, y) dz dt du  \\ 
 & \stackrel{(8.1)}{=} \int_0^\infty \int_0^\infty e^{\l (t+ u)} \Pp(t+u, x,y) dt du  \\
 &= \int_0 ^\infty \int_0 ^s e^{\l s} \Pp(s, x,y) dt ds = \int_0 ^\infty s e^{\l s} \Pp(s,x,y) ds  =  \frac{\pp }{\pp \l } G_\l (x,y).
 \end{align*}
 \end{proof}

Since the Green function is positive, by (\ref{derivative}) for $k=1$ and $2$,  the map $ \l \mapsto G_\l (x,y)$ is a convex increasing function. Since
$G_\l (x, y) $ is analytic outside the spectrum as a resolvent, its derivative is finite as well, i.e.
\begin{equation}\label{2.5} 
{\textrm {for all }} \l < \l_0, \; {\textrm { all }} x \not = y \in \M, \; \int _{\M} G_\l (x, z) G_\l (z, y) d\mbox{Vol} (z) < +\infty . \end{equation}

For each
$x \in \M$ and $v \in S_x \M$, let $\sigma_x(v)$ be the equivalence class of the geodesic $\g_v$ with the initial vector $v$. 
The mapping $\sigma_x$ is a
homeomorphism from the unit tangent
sphere $S_{x}\wt{M}$ of $\wt{M}$ at $x$ to $\partial \wt{M}$. 
Thus we
will identify the unit tangent bundle $S\wt{M}$ with
$\M\times
\partial\M$. 

For each $x \in \M$, $\pp \M$ is endowed with the Gromov metric $$d_x (\xi, \eta ) \; = \; e^{-a (\xi |\eta )_x} , $$
where $0< a \leq 1$ is such that the sectional curvature $\k$ satisfies $ \k \leq - a^2$ on $\M$ and  $(\xi |\eta )_x $ is the {\it Gromov product}
\begin{equation}\label{eqn:4.1}
(\xi|\eta)_x=\lim\limits_{y\to \xi, z\to \eta}\frac{1}{2}\left(d(x,
y)+d(x, z)-d(y, z)\right).
\end{equation}
The following properties follow from pinched negative curvature: 
\begin{prop}[\cite{A}] \label{Martinboundary}For all $\l \in [0,\l_0)$, every $\xi \in
\partial\M$ there exist a positive $\l$-harmonic function $k_{\l} (x, y, \xi) $ in $y$ such that for each $x,y \in \M,$
\begin{equation}\label{Martin} \lim\limits_{z \to \xi } \frac{G_{\l } (y, z)}{G_{\l } (x, z)} \; = \; k_{\l} (x, y, \xi). \end{equation}
For any positive $\l$-harmonic function $F$, any $x \in \M$, there is a measure $\nu _{x,F}$ on $\pp \M$ such that  \[ F(y) \; = \int _{\pp \M} k_{\l} (x, y, \xi) d\nu _{x,F}(\xi ) .\] \end{prop} 
\begin{prop}[\cite{H1}] \label{HolderNonu} Moreover, for all  $\l \in [0,\l_0)$, there are constants $\a (\l) >0, C(\l ) >0 $  such that \[ \frac{ \| \nabla_y \log k_\l (x,y, \xi ) - \nabla_y \log k_\l (x,y, \eta ) \|}{(d_x(\xi, \eta ))^{\a(\l)} } \; \leq \; C(\l) .\]
 \end{prop}

\begin{prop}[\cite{K1}]  \label{Naim}For three distinct points $x,y,z \in \M $, consider the function
\begin{equation}\label{eq:Naim}
\th _x^\l (y,z) \; := \; \frac{G_\l (y,z)}{G_\l (y,x) G_\l (x,z) } .\end{equation} 

There is a positive function $\th _x^\l  (\xi , \eta) $ on $\pp \M \times \pp \M \backslash \mathrm{Diag} := \{ (\xi, \eta) \in \pp \M \times \pp \M: \xi \neq \eta\}$ such that  
\[ \th_x^\l  (\xi  , \eta)  \; = \; \lim\limits_{y\to \xi, z\to \eta} \th _x^\l(y,z).\]
\end{prop}
The function $\th_x ^\l (\xi , \eta)$, when it is finite {as it is here},  is called the {\it Na\"im kernel} in potential theory \cite{N}. Compare with the definition of the Gromov product {(\ref{eqn:4.1})}. 

Consider $v\in SM.$ For a lift $\wt v$ in $S\M$, consider the geodesic  $\g _{\wt v} (t)$ with initial tangent vector $\dot {\g_{\wt v} }(0) = \tilde{v}.$ We will denote $\wt v^- = \g_{\wt v}(-\infty)$ and $\wt v^+ = \g_{\wt v}(+\infty)$. Set, for $v \in SM$, 
\begin{equation}\label{thetalambda} \th _\l (v) \; := \; \th _{\g_{\wt v}(0)}^\l (\wt v^+, \wt v^- ), \end{equation}
where $\wt v $ is any lift of $v$. 
Observe that, by definition,  $\th _\l(v) = \th _\l(-v) .$

Fix $x \in \M$. For $\xi \in \pp \M , y \in \M $, the Busemann function $\b (x, y, \xi)$ is defined by
\begin{equation}\label{eqn:busemann}  \b(x, y, \xi ) \; = \; \lim\limits_{y_n \to \xi } \left(d(x, y_n ) - d(y, y_n)\right) .\end{equation}

Since  $\M$ is the universal cover of a closed manifold of negative curvature, we also use the thermodynamic formalism of the geodesic flow as in \cite{K1}, \cite{H1}, \cite{L2}.

The geodesic flow $\gg = \{ \gg_t \} _{t \in \R} $ is defined on the unit tangent bundles $SM $ and $S\M$. On $SM$, the geodesic flow is an Anosov flow. For a  $\gg$-invariant probability measure $m$ on $SM$, denote by $h_m (\gg) $ the measure-theoretic entropy of the time-1 map $\gg_1$ with respect to $m$  (see e.g. \cite{W}) . For any continuous function $\vf $, define the \emph{topological pressure $P(\v )$ of $\v$} by
\begin{equation} \label{pressure}  P(\v) \; := \; \sup_m \left( h_m (\gg) + \int_{SM} \v dm \right), \end{equation} where the supremum is taken over all $\gg$-invariant probability measures on $SM$.

For all $\l \in [0,\l_0)$,
the \emph{potential function associated to $\l$} is the function on $SM$ defined as
\[ \v _\l (v) : = - 2 \frac{d}{dt} \log k_\l (\g _{\wt {v}} (0),  \g _{\wt {v}} (t), \wt v^+  )\Big|_{t= 0 } .\]

We set $P(\l ) := P(\vf _\l)$ for $0 \leq \l  < \l_0.$ 
\begin{defi}
Define $m_\l$ to be the unique equilibrium {probability}\footnote{The uniqueness follows from H\"older continuity of $\v_\l$ (Proposition~\ref{HolderNonu}).} measure of $\vf_\l$, which attains the supremum in (\ref{pressure}).
\end{defi} The measure $m_\l$ is mixing  for the geodesic flow $\gg$ of $M$.  The \emph{generalized  family of Patterson-Sullivan measures associated to the potential function $\vf_\l$}, characterized by the following proposition, can be used to describe $m_\l$ as in \eqref{eqn:Hopf}.

\begin{prop}[\cite{L2}]\label{measureNonu}  Fix $\l \in [0,\l_0)$. There is a family of finite measures $\{ \mu _y^\l \}_{y \in \M} $ on $\pp \M$ all in the same measure class such that 
\begin{itemize} \item[1)] the family $y \mapsto \mu _y^\l $ is $\G $-equivariant: $ \mu _{\g y}^\l = \g_\ast (\mu _y^\l )$ for $\g \in \G $  and
\item[2)]given any $x,y \in \M$, for $\mu _x^\l $-a.e. $\xi \in \pp \M$, $$ \frac{d\mu _y^\l }{d\mu _x^\l }(\xi ) \; = \; k_\l^{2} (x, y, \xi ) e^{P(\l ) \b(x, y, \xi ) }.$$
\end{itemize}
The family is unique if we normalize by setting $\int _{M_0} \mu _y^\l (\pp \M ) d\mbox{Vol}(y)= 1.$ 
\end{prop}

\begin{cor}\label{boundonmeasures} There exists a constant $C >0 $, such that for all $\l \in [0, \l_0) $, all $x \in \M$, 
$$ C^{-1} \leq \mu _x^\l ( \pp \M ) \leq C. $$ \end{cor} 

\begin{proof} 
By Proposition~\ref{Harnack} applied to $k_\l(x,y,\xi)$, for $x,y \in M_0$, $|\log k_\l^2(x,y,\xi)|$ are bounded. By Proposition \ref{Harnack}  again, the function $\vf_\l$ is bounded by $2\log C_0$. It follows that the pressure $P(\l) $ is bounded. Thus, the Radon-Nikodym derivatives $\frac{d\mu^\l_x}{d\mu^\l_y}$ are bounded for $x,y \in M_0$ uniformly in $\l$. Since the total measure is 1, the corollary follows.
\end{proof}

Fix $x_0 \in \M$. By the \emph{Hopf parametrization}, i.e. by associating $(v^-,v^+, \beta (x_0, \g_v(0), v^+))$ to $v$, we identify $S \wt M$ with $(\partial \wt M \times \partial \wt M  \backslash {\textrm{Diag}}(\pp \M)) \times \R$, where ${\textrm{Diag}}(\pp \M)$ is the diagonal embedding. 
Since $(\th^\l_x (\xi ,\eta))^2 e^{2P(\l) (\xi |\eta)_x}d\mu _x (\xi) d\mu_x(\eta) $ is independent of $x$, we define a $\G$-invariant, $\gg_t$-invariant measure $\wt m_\l$ by
\begin{equation}\label{eqn:Hopf} d\wt m_\l (\xi, \eta, t) =
 \Om _\l (\th^\l_x (\xi ,\eta))^2 e^{2P(\l) (\xi |\eta)_x} d\mu^\l _x (\xi) \x d\mu^\l_x(\eta)\x dt
 \end{equation}
on $S \wt M$, which does not depend on $x.$ 
Here, $\Om_\l$ is the normalizing constant chosen so that the measure $\wt m_\l$ is equal to the $\G$-invariant  lift of the \emph{probability} measure $m_\l$ to $S \wt{M}$.
\begin{remark} 
Note that we have a symmetric measure thanks to the fact that our potential function $\varphi_\l$ is cohomologous to $\varphi_\l \circ \iota$ where $\iota$ is the flip map $v \mapsto -v$ (compare with asymmetric measure in \cite{PPS} Section 3.7). Indeed, we can write, for $v \in SM, t>0,$
\begin{eqnarray*}
\int_0^t (\vf _\l- \vf_\l \circ \iota) ( \gg_s v) \, ds & = &\int_0^t \vf _\l (\gg_s v) \; ds -  \int_0^t \vf _\l (-\gg_{s} v) \; ds \\ 
&= &\log k_\l^{-2} (\g_v(0), \g_v (t) , \g_v(+\infty ) ) - \log k_\l^{-2} (\g_v(t), \g_v (0) , \g_v(-\infty ) )\\
&= &-2 \lim\limits_{s, s' \to \infty } \log \frac{G_\l (\g_v(t), \g_v (s) ) G_\l (\g_v(t), \g_v (-s') ) }{G_\l (\g_v(0), \g_v (s) ) G_\l (\g_v(0), \g_v (-s') ) }\\
&=&  \log \th _\l^2(\g _v(t )) -   \log \th _\l^2(\g _v(0 )) . \end{eqnarray*}
Note the role of $\log \theta^2_\l$ and its occurrence in the formula \eqref{eqn:Hopf}.
\end{remark}
We can also identify the orthogonal two frame bundle $S^2\M$ with the triples of pairwise distinct points in $\pp \M \times \pp \M \times \pp \M$ by associating $(v, w \in v^\perp)$ to $(v^+, v^-, w^+)$. The measure 
\begin{equation}\label{eq:tau}
d \wt \tau^\l_x (\xi,\eta, \zeta) : = \Upsilon_\l \th^\l_x(\xi, \eta) \th^\l_x(\eta, \zeta) \th^\l_x(\zeta ,\xi) e^{ P(\l) \left( (\xi|\eta)_x + (\eta|\zeta)_x + (\zeta |\xi )_x \right) } \, d\mu^\l_x(\xi) d\mu^\l_x(\eta) d\mu^\l_x(\zeta) 
\end{equation}
does not depend on $x$ and is $\G$-invariant. Here $\Upsilon_\l$ is the normalizing constant chosen so that the measure $\wt \tau^\l=\wt \tau^\l_x$ is equal to the $\G$-invariant  lift of the \emph{probability} measure $\tau^\l$ to $S^2 \wt{M}:$ for any fundamental domain $M_0$ for $\G$, \begin{equation}\label{Upsilon}
\wt \tau^\l(S^2M_0)=1.
\end{equation}

Let us recall dynamical foliations of $S\M$ in order to define measures associated to $\mu^\l_x$.
For every $v \in S\M$, define the \emph{strong stable manifold, strong unstable manifold, weak (or central) stable manifold and weak (or central) unstable manifold} of $v$ as follows:
\begin{eqnarray*} 
W^{ss}(v) & = & \{ w\in S\M: \underset{t \to +\infty}{\lim} d(\gg_t v, \gg_t w) =0 \},  \\
W^{uu}(v) & = & \{ w\in S\M: \underset{t \to -\infty}{\lim} d(\gg_t v, \gg_t w) =0 \} , \\
W^{cs}(v) & = & \{ w\in S\M:  \exists s, \underset{t \to +\infty}{\lim} d(\gg_{t+s} v, \gg_t w) =0 \},  \\
W^{cu}(v) & = & \{ w\in S\M: \exists s, \underset{t \to -\infty}{\lim} d(\gg_{t+s} v, \gg_t w) =0 \}.
\end{eqnarray*}
Recall that the homeomorphism $\sigma_x: S_x \M \to \partial \M$ sends $v$ to $v^+$. More generally, on  any manifold $T$ transversal to the foliation into $\wt W^{cs}$, the mapping $v \mapsto \s _{\pi v}v$ defines a local homeomorphism $\s : T \to \pp\M.$ For any family of measures $\{\nu _x\}_{x\in \pp \M}$ with continuous densities $\kk(x,y, \xi) := \frac{d\nu_y}{d\nu_x}(\xi) $, the measure on $T$ with density $\kk(x_0, \pi v, \s (v)) $ with respect to  $(\s ^{-1})_\ast \nu_{x_0}$ does not depend on $x_0$ (see \cite{PPS} Section 3.9 for example). 
Using the generalized Patterson-Sullivan measures $\mu^\l_x$ obtained in Proposition~\ref{measureNonu},  we can therefore define measures $\mu_\l^{uu}$ on any transversal $T$ by
$$ d\mu^{uu}_\l (w) := k_\l^2(x_0, \pi (w), w^+) e^{P(\l) \b (x_0, \pi(w), w^+)}d(\s ^{-1})_\ast \mu^\l_{x_0}(w),$$ for $w \in T$. They have the property that 
for two transversals through $\s _x^{-1} (\xi)$ and $\s_y^{-1}(\xi)$, respectively,
the Radon-Nikodym derivative $\rho_\l(\s _x^{-1}(\xi), \s^{-1}_y (\xi)) $ of the holonomy from $\s _x^{-1} (\xi)$ to $\s_y^{-1}(\xi)$ along the leaf $\wt M \times \{ \xi\}$ is given by
\begin{equation}\label{holonomy}   \rho_\l(\s _x^{-1}(\xi), \s^{-1}_y (\xi))   = k^2_\l (x,y, \xi) e^{P(\l)\beta(x,y,\xi)}.\end{equation}
Observe that moreover, the family $\mu^{uu}_\l$ is $\G$-equivariant and therefore defines a family of measures on transversals to the foliation into $W^{cs} $ in $SM.$ Similarly, using the mapping $v \mapsto \s_{\pi v} (-v)$, one associates to $\mu ^\l _x, x \in \pp\M$ an equivariant family of measures on the transversals to the foliation into $W^{cu}$:
$$ d\mu^{ss}_\l (w) = k_\l^2(x_0, \pi w, w^-) e^{P(\l) \b (x_0, \pi(w), w^-)}d(-\circ \s ^{-1})_\ast \mu^\l_{x_0}(w)$$
that satisfy  the same holonomy equation
\begin{equation}\label{holonomy2}   \rho_\l( -\s _x^{-1}(\eta), -\s^{-1}_y (\eta))   = k^2_\l (x,y, \eta) e^{P(\l)\beta(x,y,\eta)}.\end{equation}

Observe that $\mu^{uu}_\l$ on $S_x\M$ is $(\s^{-1}_x )_* \mu^\l_x$;
note that   
\begin{equation}\label{flowholonomy}
 \frac{d \mu^{uu}_\l }{d(\gg_{-t})_* \mu^{uu}_\l  } (v)  = e^{-t P(\l )} k_{\l }^2 (\g_{ \wt v} (t) ,  \g _{\wt v} (0) , \g_{\wt v} (\infty ) ), 
 \end{equation}
and  for any continuous functions $f$ and $h$ on $SM$,
\begin{eqnarray}\label{eqn:2.9}  \int_{S_{\ppi x} M }  f(v)\, d\mu _\l^{uu} (v) &= & \int _{\pp \M }  f( \ppi \circ \s ^{-1}_x \xi) d\mu _x^{\l_0 }(\xi ), \\ \label{eqn:2.10}  \int _{S_{\ppi y} M } h(-u) \, d\mu _\l^{ss} (u) &= & \int_{ \pp \M}  h (\Dppi \circ \s^{-1}_y \xi )   d\mu_y^{\l_0} (\xi),\end{eqnarray}
By a direct generalization of Margulis argument \cite{Ma} to Gibbs measures, one obtains the following proposition (see Section~\ref{sec:4} for details).
\begin{prop}\label{renewal} There exists a positive continuous function $D : (\M \x [0, \l_0 )) \to \R_+ $ such that
\[ \lim\limits_{R\to \infty } e^{-R P(\l)} \int _{S(x, R)} G_\l^2 (x, z) dz \; = \; D(x, \l )  .\]
\end{prop}
Clearly, $x \mapsto D(x, \l) $ is $\G$-invariant and depends only on $\ppi (x)  \in M.$ {The function $D(x, \l)$ will be described in Corollary~\ref{formula}.}

\begin{cor}\label{Pnegative} For all $\l \in [0, \l_0) $, we have $ P(\l ) <0.$  \end{cor}
\begin{proof}  Indeed, otherwise, we have by Proposition~\ref{Harnack} and Proposition~\ref{renewal},$$\int _{\M} G_\l (x, z) G_\l (z, y) d\mbox{Vol} (z) \gtrsim \int _{1 + d(x,y)}^{+ \infty } \left(\int _{S(x, R)} G_{\l}^2 (x, z)  dz \right) dR \gtrsim D(x,\l) \int _{1+d(x,y)}^{+ \infty }  dR.$$
The integral diverges, which is in contradiction with  (\ref{2.5}) for any $x \neq y$. 
\end{proof}

The rest of this section is devoted to the proof of Proposition \ref{Lemma2.5}, originally due to Hamenst\"adt, and of Corollary \ref{powerT}. Firstly we observe that the  easy side of the Ancona inequality is uniform in $\l \in [0,\l_0]$. For later use, we state this relation for the {\it{relative Green function }} $G_\l (x,y:\DD)$ associated to an open set $\DD$ (see equation (\ref{stoppingtime}) for definition). If $\DD = \M,$ then $G_\l(x,y:\M) = G_\l (x,y).$
\begin{prop}  There is a constant $C'_0$ such that for any open set $\DD$, 
 any $0 \leq \l \leq \l_0 $ and any $x, y,z \in \DD$ such that $d(x,z), d(x,y), d(x,\pp\DD), d(y,\pp\DD), d(z,\pp\DD) $ are all at least 1, we have
\begin{equation}\label{easyAncona} G_\l (x,z:\DD) G_\l (x,y:\DD) \leq C'_0 G_\l (z,y:\DD). \end{equation}
\end{prop}
\begin{proof}
By Corollary~\ref{coro:Harnack}  for $0 \leq \l \leq \l_0 $ and $x,y,z$ such that $d(x,z), d(x,y), d(x, \pp \DD),$ $d(y,\pp\DD), d(z,\pp\DD)$ are all at least 1, we have
\[ G_\l (x,z:\DD) G_\l (x,y:\DD) \leq C_0 \max \{ G_\l (x,y:\DD);  d(x,y )\geq 1\}   G_\l (z,y:\DD) .\]
For a fixed $\l < \l_0$, $G_\l (x,y:\DD) \leq G_\l(x,y)$ goes to 0 as $d(x,y ) \to \infty $ (see\cite{A}, Remark 2.1 page 505). By the maximum principle, \[\max \{ G_\l (x,y);  d(x,y )\geq 1\} = \max \{ G_\l (x,y);  d(x,y )=  1\}.\] Moreover, $  \max \{ G_\l (x,y);  d(x,y )=  1\} \leq \max \{ G_{\l_0} (x,y);  d(x,y ) = 1\}  .$
Set $$C'_0 :=C_0 \max \{ G_{\l_0} (x,y);  d(x,y )= 1\} $$ which is finite by compactness. Relation (\ref{easyAncona}) holds for all $\l < \l_0$, thus for $\l_0 $ as well.
\end{proof}

\begin{cor} For $0 \leq \l < \l_0 $,  $x,z$ such that $d(x,z)\geq 1 $ and $\xi \in \pp \M$, we have
\begin{equation}\label{easyAnconaII} G_{\l} (x,z) \leq C'_0 k_\l (x,z, \xi) .\end{equation}
\end{cor} \begin{proof} Divide the relation (\ref{easyAncona}) by $G_\l (x,y) $ and let $y \to \xi $.\end {proof}

Two submanifolds $A,B$ of $S\M$ are said to be $\e$-transversal at an intersection point $x$ if the angle between the spaces $T_xA $ and $T_xB$ is greater than $\e$, and transversal if the angle is positive. If $\W$ is a lamination of $S\M$ with smooth leaves $W(x), x \in S\M$, $A$ is said to be $\e$-transversal to $\W$ if at each $x \in A$, $A$ and $W(x)$ are $\e$-transversal. For example, by the Anosov property, the unit sphere $S_x\M$ at $x$ and its images by the geodesic flow $\gg_t$ for $t \geq 0$,  are all  $\e _0$-transversal to the central stable foliation $\W^{cs}$, for some $\e_0$.

\begin{prop}\label{covering} Assume $A$ is $(m-1)$-dimensional and $\e$-transversal to $\W^{cs}$ and let $\d >0$. There exists $R = R(\e,\d) $ such that for any ball  $B_A (x, \d) \subset A, $ \[ p \left(\cup_{x \in B_A(x, \d) } B^{cs} (z,R) \right)  \; = SM.\] \end{prop}
\begin{proof}
It suffices to prove it for spheres. Consider the open set $$V_R = \{ (x,z) \in SM \times SM : B^{cs}(z,R) \cap B^{S}(x, \d) \neq \emptyset \},$$
where $S=S_{p(x)}(M)$. 
By minimality of $\W^{cs}$ and the transversality of $S$ to $\W^{cs}$, we have $\underset{R >0}{\cup} V_R  = SM \times SM.$ Therefore,  $V_{R_0} = SM \times SM$ for some $R_0=R(\d).$
It follows that for any $(x,z),$ there exists $y \in B^{cs}(z,R_0) \cap B^{S}(x, \d),$ i.e.
$ z \in B^{cs}(y,R_0)$ for some $y \in B^{S}(x, \d).$ 
\end{proof}

If $A_1, A_2$ are two $(m-1)$-dimensional submanifolds both transversal to $\W^{cs}$ and $x_1 \in A_1, x_2 \in A_2$ belong to the same leaf $W^{cs}$ of $\W^{cs}$, then the holonomy from a neighborhood $B_{A_1}(x_1)$ of $x_1$ in $A_1$, to a neighborhood $B_{A_2}(x_2)$ of $x_2$ in $A_2$ is defined  by continuously extending the intersection mapping which sends $x_1$ to $x_2$. 

We defined above for $0 \leq \l < \l_0$  a  family of measures $\mu_\l^{uu}$ on $m-1$ dimensional transversals to $W^{cs}$ that are quasi invariant under the holonomy with Radon-Nykodym derivative 
$$   \rho_\l(\s _x^{-1}(\xi), \s^{-1}_y (\xi))   = k^2_\l (x,y, \xi) e^{P(\l)\beta(x,y,\xi)}$$
and that coincide with $(\s^{-1}_x )_* \mu^\l_x$ on $S_x \M.$

\begin{cor}\label{measureofballs} Let $A$ be a  $(m-1)$-dimensional submanifold of $S\M$, $\e$-transversal to $\W^{s}$ and a ball  $B_A (w, \d) \subset A. $ There is a constant $C= C (\e,\d)  $ such that,  for $0 \leq \l < \l_0,$ 
\[ \mu _\l^{uu } (B (w, \d)) \; \geq \; C^{-1} .\] 
\end{cor}
\begin{proof} By Lemma~ \ref{covering}, there is $R= R(\e,\d)$ such that $$p \left(\cup_{x \in B(w, \d) } B^{cs} (x,R) \right)  \; = SM.$$ In particular any sphere $S_yM$ is covered by $K$ holonomy images of $B(w,\d),$ with $K$ bounded by some $K_0(\e, \d)$. There is $ C_0 (\e, \d) $ such that the Radon-Nykodym derivative of the measure $\mu_\l^{uu} $ under these holonomies are bounded by $C_0(\e,\d)$. Therefore,  for all $y \in M$, $\mu _\l^{uu} (S_yM) \leq K_0(\e, \d) C_0(\e,\d)  \mu _\l^{uu } (B (w, \d)).$ By our choice of normalisation, $\int _M \mu _\l^{uu} (S_yM) \,  d\Vol (y) \; = \; 1.$ Corollary~\ref{measureofballs} follows with $C = K_0(\e, \d) C_0(\e,\d) \Vol (M).$ \end{proof}

The following proposition corresponds to \cite{G1}, Lemma 2.5.
\begin{prop} [\cite {H3}, Corollary 5.5.1)] \label{Lemma2.5}  There is a constant $C>0 $ such that for all $x \in \M$ and all $R \geq 1$, \[ \int _{S(x,R)} G_{\l_0}^2 (x, z)  dz \leq C.\]
 \end{prop} 
 \begin{proof} We first lift $S(x,R) \subset \M$ to $\gg_R S_x\M \subset S\M.$ Let $w \in \gg_R S_x\M $ and consider the ball  $B(w,1)$ of radius 1 in $\gg_R S_x\M $.  The $(m-1)$-dimensional volume of $B(w,1)$ is bounded from above, uniformly in $R \geq 1$ and $w$, whereas by Corollary \ref{measureofballs}, $\mu_\l^{uu}(B(w,1))$ is  bounded from below, uniformly in $\l, 0 \leq \l <\l_0.$ Finally, by Proposition \ref{Harnack}, the function $G_{\l}^2 (x, z) $ has a bounded oscillation on that set, uniformly in $\l, 0 \leq \l \leq \l_0.$ It follows that there is a constant $C$ such that  for any $R \geq 1$, $0 \leq \l < \l_0$ and a ball $B(w,1)$ of radius 1 in $\gg_R S_x\M $,
\[ \int_{B(w,1)} G^2_\l (x,\pi v) e^{-P(\l)R} dv \; \leq \; C  \int_{B(w,1)} G^2_\l (x,\pi v) e^{-P(\l)R} \, d\mu_\l^{uu} (v) .\]
By (\ref{easyAnconaII}) and \eqref{flowholonomy}, \[ G^2_\l (x,\pi v) e^{-P(\l)R}  \leq C'_0  k_\l^2 ( \pi v,x, \g_v (+\infty) )e^{P(\l) \b ( \pi v, x, \g_v(+\infty))}  = C'_0 \frac{d\gg_R \mu^{uu}_\l}{d \mu^{uu}_\l} (v).\]
Altogether, we get, for any ball of radius 1 in $\gg_R S_x\M $, for $0 \leq \l < \l_0,$ 
\[  \int_{B(w,1)} G^2_\l (x,\pi v) e^{-P(\l)R} dv \leq C C'_0 \int_{B(w,1)} \frac{d\gg_R \mu^{uu}_\l}{d \mu^{uu}_\l} ( v) \, d\mu_\l^{uu} (v) = CC'_0 \mu_\l^{uu} (\gg_{-R}( B(w,1)) ).\]
The sets $\gg_R S_x\M , R\geq 1$ are locally uniformly Lipschitz homeomorphic to open subsets of Euclidean $\R^{n-1}.$ Therefore we obtain a Besicovitch cover, i.e. there  is an integer $N$, independent of $R$, and  covers of $\gg_R S_x\M $ by balls of radius 1 such that any point can belong to at most $N$ distinct balls. The images of the balls in that cover by $\gg_{-R}$ form a cover of $S_x\M $ such that any point can belong to at most $N$ such images. Thus, 
 \[ \int _{\gg_R S_x\M} G^2_\l (x,\pi v) e^{-P(\l)R} dv\; \leq \; N C C'_0 \mu _\l^{uu} (S_x\M) .\] 
Since $\mu _\l^{uu} (S_x\M) = \mu _x^\l (\pp \M) $  is bounded by Corollary \ref{boundonmeasures}, we  found a constant $C$ such that for all $\l<\l_0$ and for $R\geq 1,$ 
\begin{equation}\label{2.9} \int_{S(x,R)} G^2_\l (x,z) e^{-P(\l)R} dz \leq C.\end{equation}
Here, we used the fact that the measures $\pi _\ast dv$, the projection of the Lebesgue measure for the restriction of the Sasaki metric to $\gg_R S_x\M$, and $dz$, the Lebesgue measure on $S(x,R), $ are  equivalent with bounded density. 

Since $P(\l ) <0 $ for all $\l < \l_0 $ by Corollary \ref{Pnegative}, there is a constant $C>0$ such that  for all $\l \in [0, \l_0) $, all $x \in \M$, all $R \geq 1$, \[  \int _{S(x, R)} G_{\l}^2 (x, z) dz \leq C.\]
Proposition \ref{Lemma2.5} follows by letting $\l $ go to $\l_0$. \end{proof}

\begin{cor}\label{powerT} For $T >0 $,  let $P_T (\l) $ be the pressure of the function $\frac{T}{2} \vf _\l $. Then there exists a constant $C(T)$  such that  for all $\l \in [0,\l_0), R \geq 1 , x \in \M$,
\[ e^{-R P_T(\l)} \int _{S(x, R)} G_{\l}^{T} (x, z) dz \; \leq C(T) .\] \end{cor}

\begin{proof} We have as above
 $$G^T_\l (x,z) e^{-P_T(\l)d(x,z)} \leq {C_0'}^T k^T_\l (x, z, \xi) e^{-P_T(\l)d(x,z)}.$$
 We can also apply Proposition \ref{measureNonu} to the H\"older continuous function  $\frac{T}{2} \vf _\l $ instead of  $ \vf _\l $. We obtain a family of measures $\mu _x^{\l,T} $ on $\pp \M$ such that for all $\l \in [0,\l_0)$, $\mu _x^{\l ,T}$-a.e. $\xi \in \pp \M$, $$ \frac{d\mu _y^{\l,T} }{d\mu _x^{\l,T} }(\xi ) \; = \; k_\l^{T} (x, y, \xi ) e^{P_T(\l ) \b(x, y, \xi ) } $$  and  $\int _{M_0} \mu _y^{\l,T} (\pp \M ) d\Vol(y)= 1.$
 We can therefore associate measures $\mu ^{uu}_{\l,T} $ on transversals to the central stable manifolds such that  the holonomy from $\s_x^{-1} (\xi)$ to $\s_y^{-1}(\xi)$ along the leaf $\wt M \times \{ \xi\}$ is given by
$$   \rho_\l(\s _x^{-1}(\xi), \s^{-1}_y (\xi))   = k^T_\l (x,y, \xi) e^{P_T(\l)\beta(x,y,\xi)}.$$
The same computation yields the analog of (\ref{2.9}). \end{proof}   

\section{Ancona-Gou\"ezel inequality}\label{section3}

\begin{defi} Let $v \in S\M$. The {\it {cone $\CC(v) $ based on $v$}}  is defined by:
\[ \CC(v) \; := \; \{ y; y \in \M, \angle_x (v, y ) \leq \pi / 2 \}, \] 
where $\angle_x (v,y) $ denotes the angle between $v \in T_x \M$ and the geodesic going from $x $ to $y$. \end{defi} 
We denote $\pp \CC(v) : =  \{ y; y \in \M, \angle_x (v, y ) = \pi / 2 \}. $ Observe that $\M = \CC(v) \cup \CC (-v) $ and $\pp \CC(v) = \CC (v) \cap \CC (-v) .$

\subsection{Ancona-Gou\"ezel inequality}

The key property of the $\l$-Green functions for $ 0 \leq \l \leq \l_0$ is the following \emph{uniform} Ancona inequality, which we call \emph{Ancona-Gou\"ezel inequality}. Recall the definition (\ref{stoppingtime}) of the relative Green function $G_\l(x,y:\DD)$, where $\DD$ is an open subset of $\M$ and $x\neq y \in \DD.$
\begin{theo}\label{Alano} There are constants $C_4,R_0$ such that for all $\l \in [0, \l_0 ]$, all points $(x,y,z) $ such that $y$ is on the geodesic segment $[xz]$ from $x$ to $z$ and $d(x,y) \geq R_0, d(y,z) \geq R_0$,
\begin{equation}\label{Ancona} C_4^{-1} G_\l (x,y:\DD) G_\l (y,z:\DD) \; \leq \; G_\l (x, z:\DD) \; \leq \;  C_4 G_\l (x,y:\DD) G_\l (y,z:\DD) \end{equation} 
for all  open sets $\DD$  containing $\CC(\gg_{-1}v) \cap \CC( -\gg_{d(x,z) +1}v),$
where $v \in S_x\M$  is the initial vector of the geodesic $[xz]$. \end{theo}
Theorem \ref{Alano} was proven by A. Ancona for $\l < \l_0 $ (\cite {A}). The first inequality in (\ref{Ancona}) is uniform for $\l \in [0, \l_0]$ (see (\ref{easyAncona})). The new fact here is that the second inequality (\ref{Ancona}) holds when $\l = \l_0$ as well, with the same constant $C_4$, so that  the consequences of Theorem \ref{Alano} are now uniform in $\l \in [0, \l_0 ]$. The Ancona inequality follows from the pre-Ancona inequality in the following Proposition.

\begin{prop}\label{preAlano} Let $x,y,z$ be points on a geodesic $\g$ in this order, $v$ the tangent vector to $\g$ at $x$.   Then,  there exists $\e>0, R_2>1$ such that if $r\geq R_2$ and $d(x,y) > r +1, d(y,z) > r +1,$ we have
$$G_{\l_0} (x,z: B(y,r)^c \cap \CC(g_{-1}v)\cap \CC( -\gg_{d(x,z) +1}v)) \leq 2^{-e^{\e r}}.$$
\end{prop}

\begin{proof}
As in \cite{G1}, we will construct $N=e^{\e r}$ barriers, for a positive constant $\e$ which we will specify as follows. 
\begin{figure}[!htb]
\centering
\includegraphics[scale=0.9]{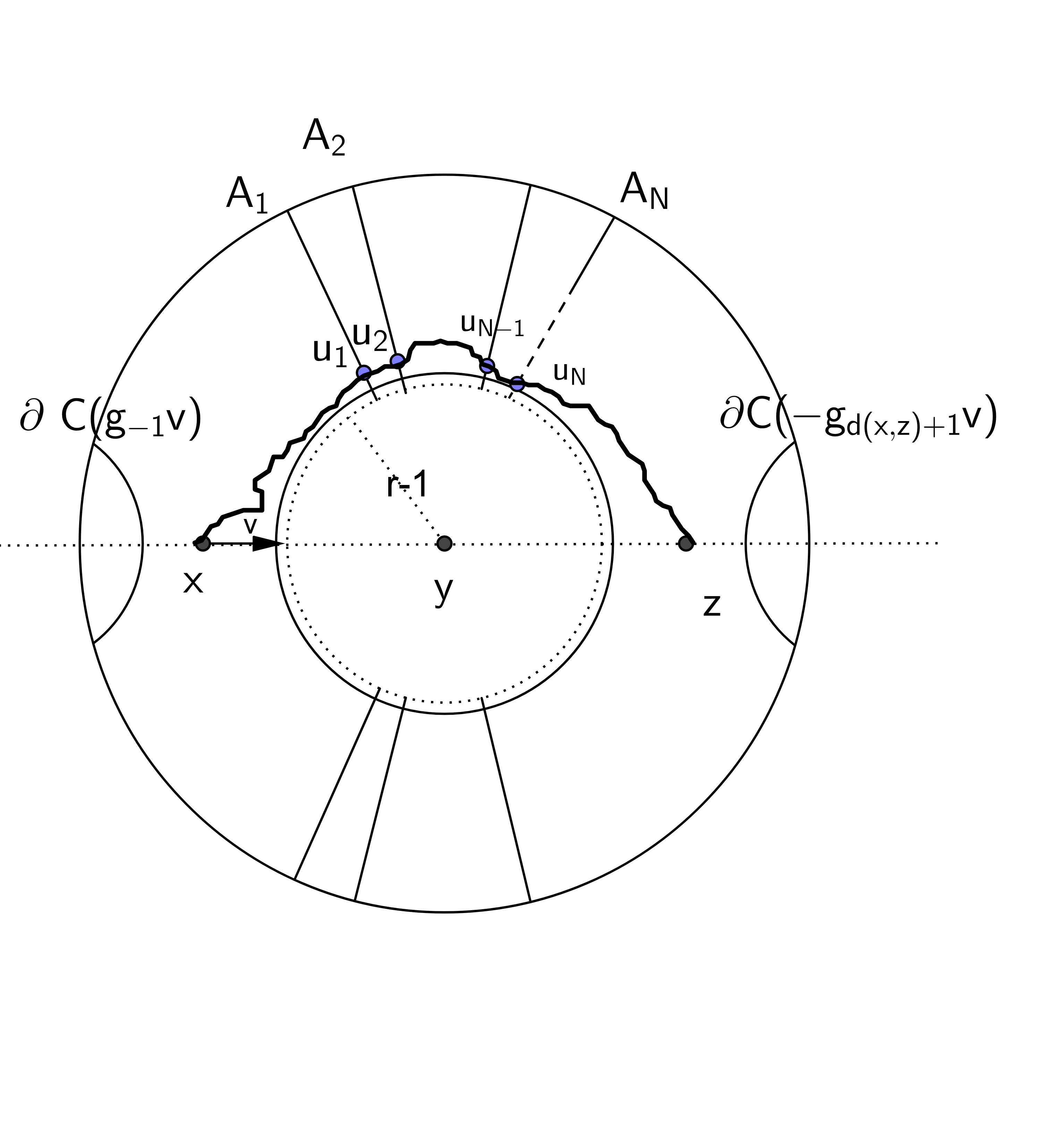}
\caption{Ancona-Gou\"ezel inequality}
\label{fig5}
\end{figure}

For $i =1 ,\cdots ,N$, let $X_i = ((N +2i-1)\pi/4N , (N+2i)\pi /4N ) \subset [\pi /4, 3\pi /4]$. Choose $\theta_i$ from $X_i$, for $i =1 ,\cdots ,N$.

By negative curvature, the intersections $\{A_i \}$'s of $B(y,r-1)^c$ and the cones $\{ w : \angle_y (x,w) = \theta_i\}$ of angle $\theta_i$ at $y$, are of distance between them bounded below by $1$ for all $r$ large enough. Set $\DD : = B(y,r)^c\cap \CC(\gg_{-1} v) \cap \CC( -\gg_{d(x,z) +1}v)$. Each set $A_i \cap \DD$ separate $\DD $ into two disjoint open sets. Let $\CC _i $ be the one containing $x$. Then $z \not \in \CC_i $. Moreover, the sets $A_i \cap \DD $ have  bounded geometry and do not intersect $\pp \CC(\gg_{-1}v) \cup \CC( -\gg_{d(x,z) +1}v)$ (see Figure~\ref{fig5}).

By  (\ref{Poissonbarrier}), we may write:
\begin{eqnarray*} G_{\l_0} (x,z :\DD ) &=& \int_{A_1\cap \DD}   G_{\l_0}(u_1, z:\DD) \, d\varpi _x^{\l_0} (u_1) \\
& =& \int _{A_1\cap \DD} \int _{A_2\cap \DD}    G_{\l_0}(u_2, z:\DD) \, d\varpi _x^{\l_0}  (u_1)  d\varpi _{u_1}^{\l_0} (u_2) \\
& =& \int _{A_1\cap \DD} \cdots \int _{A_N\cap \DD}  G_{\l_0}(u_N, z:\DD) \, d\varpi _x^{\l_0}  (u_1) \cdots  d\varpi _{u_{N-1}}^{\l_0} (u_N) \\
&\leq & \int _{A_1\cap \DD} \cdots \int _{A_N\cap \DD}  G_{\l_0}(u_N, z) \, d\varpi  _x^{\l_0}  (u_1) \cdots  d\varpi  _{u_{N-1}}^{\l_0}  (u_N)
\end{eqnarray*}
 where  $\varpi  _u^j $ is the distribution on $A_j \cap \DD$ given by (\ref{Poisson}). (Observe that $G_\l(u_j, z : \DD \setminus A_j)=0$ since $u_j,z$ are separated by $A_j$.) 
Observe that, by Proposition~\ref{Harnack}, for  all $u_N \in A_N , ||\nabla _{u_N} \log G_{\l _0} (u_N , z) || \leq \log C_0$. By construction, $d(A_N \cap \DD , B(y, r-1)) \geq 1$ and for all $ u_{N-1} \in A_{N-1}, d(u_{N-1}, B(y,r-1)) \geq 1$.  So,  we may apply   Proposition \ref{regularity} and obtain a constant $C_5= C_3 C_0^2$  such that 
\[ \int _{A_N\cap \DD}  G_{\l_0}(u_N, z) \,   d\varpi  '_{u_{N-1}} (u_N) \; \leq \;C_5  \int _{A_N} G_{\l_0}(u_{N-1}, u_N)  G_{\l_0}(u_N, z) \, du_N, \]
where $\varpi ' _z$ is the distribution on $\ov {A_N \cap \DD}$ associated with (\ref{Poisson}) for the domain 
$B(y,r-1)^c\cap \CC(\gg_{-1} v) \cap \CC( -\gg_{d(x,z) +1}v).$ Since $\DD \subset B(y,r-1)^c\cap \CC(\gg_{-1} v) \cap \CC( -\gg_{d(x,z) +1}v),$ we have $\varpi _z^{\l_0}  \leq \varpi '_z$ on $A_N \cap \DD$ and therefore
\[ \int _{A_N\cap \DD}  G_{\l_0}(u_N, z) \,   d\varpi  _{u_{N-1}}^{\l_0}  (u_N) \; \leq \;C_5  \int _{A_N} G_{\l_0}(u_{N-1}, u_N)  G_{\l_0}(u_N, z) \, du_N. \]
The right hand side satisfies  for  all $u_{N-1} \in A_{N-1}  ,$\[ ||\nabla _{u_{N-1}} \int _{A_N} G_{\l_0}(u_{N-1}, u_N)  G_{\l_0}(u_N, z) \, du_N|| \leq  C_0 \int _{A_N} G_{\l_0}(u_{N-1}, u_N)  G_{\l_0}(u_N, z) \, du_N\] because it is an integral in the variable $u_N$ of the functions $G_{\l_0}(u_{N-1}, u_N) $ with that property. We can iterate the application of Proposition \ref{regularity} and obtain
\begin{eqnarray*} 
G_{\l_0} (x,z : \DD) & \leq & C_5^N\int_{A_1} \cdots \int_{A_N}  G_{\l_0} (x,u_1)G_{\l_0}(u_1, u_2) \cdots G_{\l_0}(u_N,z) \, du_1 \cdots du_N  \\
&=& C_5^N \int G_{\l_0} (x, u_1) \left(L_1 \cdots L_{N-1} G_{\l_0}(u_N, z) \right) (u_1) du_1\\
& = &  C_5^N || G_{\l_0}(x, u_1)||_{L^2(A_1)} \cdot || L_1 \cdots L_{N-1} G_{\l_0} (u_N, z) ||_{L^2(A_1)} \\ &\leq  & C_5^N || G_{\l_0}(x, u_1)||_{L^2(A_1)}  \prod_{i=1}^{N-1} ||L_i|| \cdot || G_{\l_0} (u_N, z) ||_{L^2(A_N)}, 
  \end{eqnarray*}
where $L_i : L^2(A_{i+1}) \to L^2(A_i)$ is defined by $L_i h(u_i) = \int G_{\l_0} (u_i, u_{i+1}) h(u_{i+1}) du_{i+1}$, $||\cdot||_{L^2(A_i)}$ is the $L^2$-norm on $A_i$  and $||L_i||$ is the operator norm.
Set $$ f_0 :=  || G_{\l_0}(x, u_1)||_{L^2(A_1)} , \; f_i =  ||L_i|| {\textrm{ for }} i = 1, \cdots, N-1,$$
$$ {\textrm{ and }} f_N := || G_{\l_0} (u_N, z) ||_{L^2(A_N)} . $$
Thus, to prove Proposition \ref{preAlano},  it suffices to show that there exist $\theta_1, \cdots, \theta_N$ such that for all $i = 0, \cdots, N$, $f_i(\theta_1, \cdots, \theta_N)  < \frac{1}{4C_5}.$

Now choose $\theta_i$ uniformly from $X_i$.
We claim that, for all $i$, the expectation of $f_i^2=f_i^2(\theta_i, \theta_{i+1})$ with respect to normalized measures $\frac{16}{\pi^2} N^2 d\theta_i d\theta_{i+1}$ satisfies
$$\E(f_i^2) \leq \frac{e^{-\e r}}{20C_5^2} ,$$ 
if $\e $ is small enough. It will imply that $\E(\sum f_i^2) \leq \frac{(N+1)  e^{-\e r}}{20C_5^2} < \frac{1}{16C_5^2}$, which will  in turn imply that $\sum f_i^2(\theta_1, \cdots, \theta_N) < \frac{1}{16C_5^2}$ for some $\{\theta_1, \cdots, \theta_N\}$, thus $f_i(\theta_1, \cdots, \theta_N) < \frac{1}{4C_5}$ for all $i$ for that choice of $\{\theta_1, \cdots, \theta_N\}$ and Proposition~\ref{preAlano} will follow.

Now it remains to prove the claim. 
Fix a set $S$ of generators for $\G$, an order on $S$ and its induced lexicographical order on $\G$.
For $\aa\in A_i, \bb \in A_{i+1}$, let $\g_0$ and $\g_1$ be the first elements of $\G$ in the lexicographical order such that 
$$d(\g_0 y, \aa) < \mathrm{diam} \; M \; \mathrm{and} \; d(\g_1 y, \bb) < \mathrm{diam} \; M . $$
Set $\Phi(\aa,\bb, \theta_i, \theta_{i+1}) =\g_0^{-1} \g_1  \in \G.$ 
 
Denote by $d\mu(\aa,\bb,\theta_i, \theta_{i+1})$ the {product} of the Lebesgue measures on $A_i , A_{i+1}$ and of  $\frac{16}{\pi ^2} N^2 d\theta_i d \theta_{i+1}$
and  define $$\AA(z) = \mu (\{ (\aa,\bb, \theta_i, \theta_{i+1}) : z \in \Phi(\aa,\bb,\theta_i, \theta_{i+1}) M_0\} ) /\mathrm{vol}(M).$$ Here, for convenience, we choose $M_0$ to be a fundamental domain containing $y$.
We have \[  G_{\l_0} (\aa,\bb) =  G_{\l_0} (\g_0^{-1} \aa, \g_0^{-1} \bb) \leq C_0^{2\mathrm{ diam }  M} G_{\l_0} (y, \g_0^{-1} \g_1 y),\] 
where $C_0^{2\mathrm{ diam }  M}$ comes from Proposition~\ref{Harnack}.
Thus,
\begin{eqnarray*}  
\E (f_i^2) & = & \int G_{\l_0}^2 (\aa,\bb) d\mu(\aa,\bb, \theta_i, \theta_{i+1}) \\
&\leq& C_0^{2  \mathrm{ diam }  M} \sum_{\g \in \Gamma} G^2_{\l_0} (y, \g y) \mu (\{ (\aa,\bb, \theta_i, \theta_{i+1}) : \Phi(\aa,\bb,\theta_i, \theta_{i+1} ) = \g \})\\
&\leq & C_0^{4\mathrm{ diam }  M} \int_{\wt{M}} G^2_{\l_0} (y,w) \AA(w) d\mathrm{\Vol}(w),
  \end{eqnarray*}  
  Let us estimate $\AA(w)$ for a fixed $w \in \M$. First $w$ determines $\gamma$ such that $w \in \g M_0$. For arbitrary $\g_0$, set
$$\AA(w, \gamma_0)  := \mu \{ (\aa,\bb, \theta_i, \theta_{i+1}) : \aa \in \g_0 M_0, \bb \in \g_0 \g M_0 \}.  $$
For such $(\aa,\bb,\theta_i, \theta_{i+1})$, $\theta_i, \theta_{i+1}$ vary in intervals of size $e^{-a_0 d(y,\aa)}$, $e^{-a_0 d(y,\bb)},$ respectively, for some constant $a_0$ depending on the upper bound of the sectional curvature. 
Therefore,
 $$\AA(w, \g_0)  \leq \frac{16}{\pi^2} N^2 e^{-a_0 (d(y,\aa) + d(y,\bb))}\leq \frac{16}{\pi^2} N^2 e^{-a_0 d(\aa,\bb)}.$$
Now let us bound the number of possible $\g_0$ . 
Observe that the angles $\angle _y (\g_0y, \aa)  , \angle _y (\g_1y, \bb) $  are at most $\mathrm{diam} M \cdot  e^{-a_0 r}$.
 If $\e $ is chosen small enough, this implies that  $\angle _y (\g_0y, \g_1 y)  \geq e^{-\e r}/2 $.
It follows that the distance from $y$ to the geodesic $[\g_0y, \g_1 y] $ is at most $a_1 \e r$,  for some constant $a_1$ depending on the upper bound of the sectional curvature. 
The number of possible choices for $\g_0$ is proportional to the volume of an $a_1 \e r$-neighborhood of  the geodesic $[\g_0y, \g_1 y] $. The distance $d (\g_0y, \g_1 y)$ is $d (y, (\g_0)^{-1}\g_1 y) \leq d(y,w) + 2 {\textrm { diam}}M_0. $ We also have $d(\aa,\bb) \leq d(y,w) + 2 {\textrm { diam}}M_0.$ Thus,   
$$\AA(w)\; \ls \; d(y,w) e^{a_1a_2 \e r} e^{2\e r} e^{-a_0 d(y,w)},$$
where $a_2$ is a constant coming from Bishop comparison theorem (thus depends on the lower bound of the sectional curvature).
It follows that there exists $R_2$ such that if $\e$ is chosen small enough and $r \geq R_2$,
\begin{eqnarray*}
\E(f_i^2) & \ls &  e^{(2+ a_1a_2 ) \e r} \int_r^{\infty} R e^{-a_0  R} \int_{S(y,R)} G^2_{\l_0} (y,z) dR \\
& \ls &  e^{(2+ a_1a_2 )\e r} \int_r^\infty Re^{-a_0  R} dR \; \ls \;e^{((3+ a_1 a_2)\e-a_0) r}  < \frac{ e^{-\e r}}{20C_5^2},
\end{eqnarray*}
where we used Proposition~\ref{Lemma2.5} for the second inequality.

The proof that one can choose $\e $ and $R_2$ so that $\E f_0^2$ and $\E f_N^2$ are less than $ e^{-\e r}/{20C_5^2}$ as well is similar. For instance, let us estimate
\[ \E f_0^2 \; = \; \frac{4 N}{\pi} \int _{A_1\x X_1} G_{\l_0}^2 (x,u_1) \, du_1 d\theta _1 \; \ls e^{\e r} \sum_{\g, d(y,\g x) \geq r} G_{\l_0}^2 e^{-a_0 d(y, \g x)}.\]
There is a constant $a_3$ depending only on the upper bound of the curvature such that $ 0\leq d(x,y) + d(y, \g x) - d(x, \g x) \leq a_3.$ It follows that 
\[ \E f_0^2 \;\ls \; e^{\e r} e^{a_0 d(x,y)} \int _{r + d(x,y) -a_3}^\infty  e^{-a_0 s} \, ds \; \ls e^{-(a_0 -\e)r} ,\]
where we used Proposition~\ref{Lemma2.5} for the first inequality.
\end{proof}

\noindent \textit{Proof of Ancona-Gou\"ezel inequality.}
Theorem \ref{Alano}  follows from Proposition \ref{preAlano} by an inductive  argument (see also \cite{G1}, \cite {GL}).
Indeed, let $x,y,z, \DD $ be as in  Theorem \ref{Alano}, $\l \in [0,\l_0]$. We want to estimate from above \[ \frac{G_\l(x, z: \DD)}{G_\l(x, y: \DD)G_\l(y, z: \DD)} .\]
Set $\Psi (r,r') $ the highest possible value of this ratio for  $x,y,z, \DD $ as in  Theorem \ref{Alano}, with $d(x,y ) \leq r, d(y,z) \leq r'$, and $\l \in [0,\l_0]$. By Proposition \ref{Harnack}, this quantity is well defined. Moreover, by definition, the functions $r, r' \mapsto \Psi (r, r') $ are nondecreasing.
 Assume without loss of generality that $r \geq r'$. 
\begin{lem}\label{lem:3.6}
There is $\theta, 0< \theta < 1 $ and $R$  such that,  if $r \geq r'  \geq R$, \begin{equation}\label{nesting}  \Psi (r,r') \; \leq \;e^{ \theta ^r}\Psi (r/2, r') .\end{equation}\end{lem}
It follows that for all $(r,r')$, 
\[ \Psi (r, r') \; \leq \Pi_{k \in \N} e^{2 \theta ^{2^k R }} \Psi (R,R) .\]
This shows Theorem \ref{Alano} since the infinite product is converging and $ \Psi (R,R) $ is finite.

It remains to prove Lemma~\ref{lem:3.6}.
\begin{proof}
Consider  $(x,y,z, \DD) $ as in  Theorem \ref{Alano}, with $d(x,y ) \leq r, d(y,z) \leq r' ,$  and $\l \in [0,\l_0]$ such that 
\[ \frac{G_\l(x, z: \DD)}{G_\l(x, y: \DD)G_\l(y, z: \DD)}\; \geq \; e^{-\th^r /3} \Psi (r, r') .\]
for some $\theta , 0 < \theta  <1$ chosen later. There is nothing to prove  if $d(x,y) 
\leq  r/2$. Assume $d(x,y) >r/2$ and let  $x'$ be the point in the segment $[x,y]$ with $d(x',y ) = 0.3 \,r$. 
Using (\ref{Poisson}) with the sphere $S(x', 0.1 r)$ of points at distance $0.1 r$ from $x'$, we see that we can write
 \begin{equation}\label{induction} G_\l (x,z : \DD) \; = \; \int _{S (x', 0.1 r)} G_\l (w,z:\DD) \, d\varpi _{x}^\l (w)  + G_\l (x,z : \DD \cap B(x',0.1r)^c ). \end{equation}

By hypothesis, the domain $\DD$ contains $\CC(g_{-1}v)\cap \CC( -\gg_{d(x,z) +1}v) $. Recall  $R_2$ is the constant in Proposition~\ref{preAlano}. If $r > 10 R_2$, we can apply Proposition \ref{preAlano} to $x,x'$ and $z$ (we indeed have $d(x,x') \geq 0.2 r > 0.1 r +1$) and get,  for all $\l, 0 \leq \l \leq \l_0 $, \[  G_\l (x,z : \DD \cap B(x', 0.1\,r)^c )\leq  G_{\l_0} (x,z : \DD \cap B(x', 0.1\,r)^c ) \leq 2^{-e^{\e(0.1 \, r)}} .\]  

On the other hand, for $w\in S (x', 0.1 r),$ $d(w,z_{-1}) \leq 1.4 r$ and $d(w,x_1) \leq 0.8 r$, where $x_1 = \g_v (1), z_{-1}= \g_v(d(x,z) -1)$, so that, by Propositions \ref{regularity2} and \ref{Harnack} 
\begin{eqnarray*} \int _{S (x', 0.1 r)} G_\l (w,z:\DD)\, d\varpi _x^\l (w) &\geq & C_3^{-1} C_0^{-2} \int _{S (x', 0.1 r)} G_\l (w,z :\DD) G_\l (w, x: \DD) \, dw \\
&\geq & C_3^{-1} C_0^{-2-2.2 r} \kappa ^2 \int _{S (x', 0.1 r)} dw \\ &\geq & c^r \end{eqnarray*} 
for some $c >0$ if $r $ is large enough, where $\kappa >0$ is given by $\kappa := \inf_{x,z, \DD} \{G_{0} (x, x_1): \DD), G_{0}(z, z_{-1}:\DD )\} $. For all  $\th $ there is  $R$ such that for $r \geq R$, \[ 2^{-e^{\e(0.1 \, r)}}   \; \leq \; \left(e^{\th^r /3} -1\right) c^r, \quad {\textrm{so that }} \]
\[G_\l (x,z : \DD \cap B(x',0.1r)^c )\; \leq \; \left(e^{\th^r /3} -1\right) \int _{S (x', 0.1 r)} G_\l (w,z:\DD) \, d\varpi _x^\l (w) \quad {\textrm {and thus}}\] 
\begin{equation}\label{eqn:3.5}  G_\l (x,z : \DD) \leq  e^{\th^r /3}  \int _{S (x', 0.1 r)} G_\l (w,z:\DD) \, d\varpi _x^\l (w). \end{equation} 

Let $z_1$ be the point $z_1 := \g_v(d(x,z)+1) \in \DD.$ Consider on the geodesic segment $[w,z_1]$ the point  $y'$ such that $d(y',z_1) = d(y,z_1) $ and $z'$ the  point closest to $z $ with the  property that $ \CC_{-v_{z_1}^x} \subset \CC _{-\gg_{-1}v_{z'}^{w}} . $ With such a choice, each $(w, y',z', \DD)$ satisfies the hypotheses of theorem \ref{Alano} with $d(w,y') \leq r/2, d(y',z') \leq r'$ so that $G_\l (w,z' : \DD)  \leq  \Psi(r/2,r')  G_\l (w,y' : \DD) G_\l(y',z':\DD).$\\
Moreover, there are constants $a_0, a_1$, depending only on the curvature such that \footnote{Let $w'$ be the point in the segment $[x,z]$ that is closest to $w$. The estimate on $d(y,y') $ follows from the comparison of the geodesic triangle $wz_1w'$. Since $d(z_1,y) = d(z_1, y') = r' +1 \geq R+1$, the angle at $z_1$ in the geodesic triangle $wz_1w'$ is  at most $d(y,y')$ for $R$ large enough. Then $d(z,z') = d(z_1,z'_1)$, where $z'_1$ is the closest point to $z_1$ in the segment $[w,z_1]$ with the  property that $\CC _{\pm} (v_{z'_1}^w) $ does not intersect $\CC_{\pm}(v_{z_1}^x). $ There is an ideal triangle based on the segment $[z_1z'_1]$ with angle $\pi /2 $ at $z'_1$ and at least  $\pi /2 - d(y,y') $ at $z_1$. The estimate on   $d(z,z') = d(z_1,z'_1)$ follows by comparison.}  \[ d(y,y') \; \leq e^{-0.3 a_0 r} 0.1 r \;\;  {\textrm { and }}  d(z',z) \; \leq \; \frac {d(y,y')}{a_1} .\]
So, by Proposition \ref{Harnack}, we obtain, replacing $y'$ by $y$ and $z'$ by $z$, \begin{eqnarray*} G_\l (w,z : \DD)  & \leq & C_0^{\frac {d(y,y')}{a_1}}  G_\l (w,z' : \DD) \; \leq \; C_0^{\frac {d(y,y')}{a_1}} \Psi(r/2,r')  G_\l (w,y' : \DD) G_\l(y',z':\DD) \\ 
& \leq &  C_0^{(2  + 2/a_1) d(y,y') } \Psi(r/2,r')  G_\l (w,y: \DD) G_\l(y,z:\DD) .\end{eqnarray*}
We choose $\th  $ and $R$ such that (\ref{eqn:3.5}) holds and that for $r \geq R,$ \[ C_0^{(2  + 2/a_1)e^{-0.3 a_0r} 0.1 r } \; \leq \; e^{\th^r /3} \] (take for instance $ e^{-0.2 a_0}< \th <1  $ and $R$ large). We obtain 
 \[  G_\l (x,z : \DD) \;\leq \; e^{2\th^r /3} \Psi (r/2, r') G_\l (y,z:\DD)  \int _{S (x', 0.1 r)} G_\l (w,y:\DD) \, d\varpi _x^\l (w) .\]
By (\ref{Poisson}), the last integral is at most $G_\l (x,y:\DD) $ and  Lemma~\ref{lem:3.6} follows:
\[ \Psi (r,r') \; \leq \; e^{\th^r /3}\frac{G_\l(x, z: \DD)}{G_\l(x, y: \DD)G_\l(y, z: \DD)} \; \leq \; e^{ \theta ^r}\Psi (r/2, r') .\]
\end{proof}

We use the following notation throughout this article: $\sim^{a}$ means that the ratios between the two sides  are bounded by $a$.

\begin{cor} There are constants $C_8,R_1$ such that, for all $\l \in [0,\l_0]$, all $v \in S\M$,  all  $y,y' \not \in \CC(\gg_{-R_1} v) $ and all $z \in \CC(\gg_{R_1} v) $,
\begin{equation}\label{AndersonSchoen2} G_\l(y,z)\;\sim^{C_8}\;  G_\l(y, \pi(v)) G_\l( \pi(v),z), \quad \frac{ G_\l(y,z)}{ G_\l(y', z)} \sim^{C_8^2} \frac{ G_\l(y, \pi(v))}{ G_\l(y', \pi(v))} .\end{equation} \end{cor}
\begin{proof}
Let  $y \not \in \CC(\gg_{-R} v) ,z \in \CC(\gg_{R} v) $. If $R$ is large enough, on the geodesic $[yz],$  the closest point $w(y,z) $ to $\pi(v)$  satisfies $d (w(y,z) ,\pi(v)) \leq 1$. The first inequality in (\ref{AndersonSchoen2})
 follows directly from  (\ref{Ancona}) and Proposition~\ref {Harnack}, the second from the first applied to $y,y' \not \in \CC(\gg_{-R_1} v) $.
\end{proof}

\subsection{$\l_0$-Martin boundary}

 We now follow Section 6 of \cite{AnS} simultaneously for all $\l \in [0,\l_0]$ to obtain Propositions~\ref{Martinboundary}, ~\ref{HolderNonu}, ~\ref{Naim}  uniformly in $\l \leq \l_0$. For $x,y, z \in \M , \l \in [0,\l_0],$ set $$ k_\l (x,y,z) \; := \; \frac {G_\l (y,z) }{G_\l (x,z)}.$$ The function $k_\l (x,y,z) $ is clearly $\l$-harmonic in $y$ on $\M \setminus \{z\}.$
 
 \begin{lem}\label{unifstrong} There are constants $C >1, K<1$ such that for all geodesic $\g$ and all $x, y \notin \CC (\dot \g(-2R_1-T)),z, w\in \CC (\dot \g (2R_1)), \l \in [0,\l_0],  T>0$, 
\begin{equation*} \Big| \log  \frac {k_\l (x,y, z)}{k_\l (x,y, w )} \Big|  \; \leq \; CK^T . \end{equation*} \end{lem} 
\begin{proof} 
It suffices to prove the case $T = 2nR_1 $ for $n \in  \N$. For $v \in S\M$, denote $\CC _{\pm 1} (v) := \CC (\gg_{-1} v) \cap \CC (- \gg_1(v))$. Fix a geodesic $\g$ and points $z,w \in \CC (\dot \g (2R_1)).$  for $x,y \in \CC_{\pm 1} (\dot \g(-2n R_1)) ,$
denote \[ k_\l (\xx,\yy,\zz;n) = \frac {G_\l (\yy,\zz: \CC (\dot \g(-2nR_1 -2))) }{G_\l (\xx,\zz:\CC (\dot \g(-2nR_1 -2) ))}.\] 

The following numbers $ \ov \th (n), \un \th(n)  $ are well defined for $n \in \N$ since by \eqref{AndersonSchoen2}, they are between $\left(C_8^4\right)^{-1}$ and $C_8^4$, independently of $\l \in [0,\l_0],$ the geodesic  $\g$ and $z,w \in \CC (\dot \g (2R_1))$  :
	\[\ov \th(n) \; : = \; \underset{\xx,\yy \in  \CC_{\pm 1} (\dot \g ( -2nR_1))}{ \sup} \frac {k_\l (\xx,\yy, \zz ; n)}{k_\l (\xx,\yy, \ww ; n)} \quad \quad \un \th(n) \; : = \; \underset{\xx,\yy \in  \CC_{\pm 1} (\dot \g ( -2nR_1))}{ \inf}\frac {k_\l (\xx,\yy, \zz ; n)}{k_\l (\xx,\yy, \ww ; n)}.\]

Let $x,y \in \CC_{\pm 1} (-2(n+1)R_1)$. We apply Proposition~\ref{Markov} with $\DD = \M$ and the separating $A = \pp \CC (\dot \g (-2nR_1)).$ Denote $\varpi _\xx^\l , \varpi _\yy^\l $ the hitting distributions on $\pp \CC (\dot \g (-2nR_1)).$ Any continuous curve from $\xx$ or $\yy$ to $\zz$ or $\ww$ 
crosses $\pp \CC (\dot \g (-2nR_1))$, so that we have the following estimates. 
(For simplicity, we omit the domain $\CC (\dot \g(-2(n+1)R_1 -2))$ of the Green functions in the following paragraph.)
\begin{eqnarray*}
&   &  \frac {k_\l (\xx,\yy, \zz; n+1)} {k_\l (\xx,\yy, \ww; n+1 )} - \un \th (n) \; = \;  \frac{G_\l(\yy,\zz) G_\l(\xx,\ww) -\un \th(n) G_\l(\xx,\zz) G_\l(\yy,\ww) }{G_\l(\xx,\zz) G_\l(\yy,\ww) }\\ 
&=&\frac{\int_{a, b \in \pp \CC (\dot \g (-2nR_1))}  \left[ G_\l(a,\zz) G_\l(b,\ww) - \un \th (n) G_\l(b,\zz) G_\l(a,\ww) \right] \, d\varpi _\xx^\l (b) d \varpi _\yy^\l (a) }
{\int_{a, b \in \pp \CC (\dot \g (-2nR_1))}  G_\l(a,\ww) G_\l(b,\zz) \,  d\varpi _\xx^\l (b) d \varpi _\yy^\l (a)  }\\
 &\sim^{(C_3C_0)^4} & \frac{\int_{a, b \in \pp \CC (\dot \g (-2nR_1))}G_\l(\yy,a) G_\l(\xx,b)\left[ G_\l(a,\zz) G_\l(b,\ww) - \un \th (n) G_\l(b,\zz) G_\l(a,\ww) \right] \, da db }
{\int_{a, b \in \pp \CC (\dot \g (-2nR_1))}  G_\l(\yy,a) G_\l(\xx,b)G_\l(a,\ww) G_\l(b,\zz) \, da db },  \end{eqnarray*}
where we used Propositions~\ref{regularity}  and \ref{regularity2} to write the last line and $C_0$ comes from Proposition~\ref{Harnack}. This is possible since both functions \[ 
G_\l(a,\zz) G_\l(b,\ww) - \un \th (n) G_\l(b,\zz) G_\l(a,\ww)  \, {\textrm { and }} \, G_\l(a,\ww) G_\l(b,\zz) \] are positive harmonic in $a$ and in $b$ on a neighbourhood of size at least 1 of $\pp \CC (\dot \g (-2nR_1))$.
Using (\ref{AndersonSchoen2}) with the point $x_n : = \g(-(2n+1)R_1),$ we obtain 
\begin{eqnarray*} &  & \frac {k_\l (\xx,\yy, \zz;n+1)} {k_\l (\xx,\yy, \ww; n+1 )} - \un \th (n) \\ &\sim^{(C_8C_3C_0)^4} & \frac{\int_{a, b \in \pp \CC (\dot \g (-2nR_1))}G_\l(x_n,a) G_\l(x_n,b)\left[ G_\l(a,\zz) G_\l(b,\ww) - \un \th (n) G_\l(b,\zz) G_\l(a,\ww)  \right] \, da db }{\int_{a, b \in \pp \CC (\dot \g (-2nR_1))}G_\l(x_n,a) G_\l(x_n,b) G_\l(b,\zz) G_\l(a,\ww) \, dadb }.
 \end{eqnarray*}
  Since the last line above doesn't depend on $\xx$ and $\yy$, we have, setting $C' = (C_8C_3C_0)^8$,
\begin{eqnarray*} \ov  \th (n+1) - \un \th (n) & = & \sup \{ \frac {k_\l (\xx,\yy, \zz;(n+1))} {k_\l (\xx,\yy, \ww;(n+1) )} - \un \th (n) \} \\ &\leq & C'\inf  \{ \frac {k_\l (\xx,\yy, \zz;(n+1) ))} {k_\l (\xx,\yy, \ww;(n+1) )} - \un \th (n)\} \\ &=& C'\left( \un \th(n+1) - \un \th (n) \right). \end{eqnarray*}
Applying an analogous argument to the function   $\ov \th (n) - \frac {k_\l (\xx,\yy, \zz;(n+1))} {k_\l (\xx,\yy, \ww;(n+1) )}$, we get 
\[  \ov \th (n) - \underline \th (n+1) \leq C'\left(\ov \th (n) - \ov \th (n+1) \right). \]
Therefore,   by adding the two inequalities and multiplying the results, \[ \ov \th (n ) - \un \th (n ) \; \leq \left(\frac{C'-1}{C'+1} \right)^{n-1}   (\ov \th (1) - \un \th (1) )  \leq C_8^2\left(\frac{C'-1}{C'+1} \right)^{n-1}  .\] 
 Since both $k(\xx,\yy,\zz)$ and $k(\xx,\yy,\ww)$ are 1 for $\xx=\yy$, we have $ \un \th \leq 1\leq \ov \th $.  Since the difference $\ov \th (n)- \un \th(n)$ is small, they are both close to 1 and the ratio  is between 
$\log \un \th$ and $\log \ov \th$, which are of the same order as $\max\{ \ov \th -1, 1- \un \th \} \leq \ov \th - \un \th$.
Finally, we obtain constants $C   $ and $ K <1 $  such that, for all geodesic $\g$, all $\l \in [0,\l_0],$ all   $\xx,\yy \in  \CC_{\pm 1} (\dot \g (-2nR_1))$ and $ \zz,\ww \in \CC (\dot \g (2R_1))$ \begin{equation}\label{preunifstrong}
\Big| \log  \frac {k_\l (\xx,\yy, \zz; n)}{k_\l (\xx,\yy,\ww;n)} \Big|  \; \leq \; C K^n. \end{equation}
Consider now $\g, x,y, z,w, T$ in the statement  of Lemma~\ref{unifstrong}. Choose $N$ so that $2N R_1 \leq T < 2(N+1)R_1 .$ Setting $A = \pp \CC (\dot \g (-2NR_1))$ we can write, using (\ref{hitting})
\[  \frac {k_\l (x,y, z)}{k_\l (x,y, w )} = \frac{G_\l (y,z)G_\l (x,w)}{G_\l (x,z)G_\l (y,w)} = \frac{\int_{A\x A} G_\l (a,z)G_\l (b,w) \, d\varpi _y(a) d\varpi_x(b)}{\int_{A\x A} G_\l (b,z)G_\l (a,w) \, d\varpi _y(a) d\varpi_x(b)}.\]
Since $(a,b) \in A\x A \subset  \CC_{\pm 1} (\dot \g (-2NR_1))$ and $ \zz,\ww \in \CC (\dot \g (2R_1))$, Lemma~\ref{unifstrong} follows from (\ref{preunifstrong}).
\end{proof}

In the rest of this section, we use lemma \ref{unifstrong} to obtain the properties from Propositions  \ref{Naim}, \ref{HolderNonu},   \ref{Martinboundary} and \ref{measureNonu} at $\l_0$ and that the corresponding objects depend continuously  on $\l $ as $\l \to \l_0.$ 

\begin{prop}\label{demiMartin} (1) Let $\xi \in \pp \M, x,y \in \M$ and $ \l \leq \l_0 $. The following limit exists and defines a positive $\l$-harmonic function in $y$ $$ k_\l (x,y, \xi) \; = \; \lim _{z \to \xi} k_\l (x,y,z),$$
 which we call the $\l$-Martin kernel.

(2) Fix $x,y \in \M.$ There exist $\a$ and $C = C(\max \{d(x,y),1\} )>0$ such that for any $\l \in [0,\l_0] $,
\[  \Big| \log \frac{k_\l (x,y,\xi )}{k_\l (x,y , \eta)} \Big| \; \leq\; C (d_x(\xi , \eta))^\a ,\] where $d_x$ is the Gromov metric on $\pp \M$. Moreover, for $\a' < \a$, the function $\l \mapsto k_\l (x,y, \xi) $ is continuous from $[0,\l_0] $ into the space  of $\a'$-H\"older continuous functions on $\pp \M$. \end{prop} 
\begin{proof} (1) It suffices to show it for a fixed $x = x_0$ and a sequence $z_n  \to \xi .$
 Let $\g$ be the geodesic going from $x_0$ to $\xi $. There is $T$ such that $x_0, y \not \in \CC(\dot \g(T-2R_1))$. As $n \to \infty $,  $z_n  \in \CC(\dot \g(T_n +2R_1)),$ with $T_n \to \infty .$ By Lemma~\ref{unifstrong}, the sequence $k_\l (x_0, y, z_n) $ converges.
 
 (2) Let $\g $ be the geodesic such that $\g (0) =x, \g(+\infty ) = \xi$. There is $\d_0 $ depending only on the curvature bound such that if  the Gromov distance  $d_x(\xi, \eta )$ is smaller than $ \d_0, $ and  $T \leq - C \log d_x(\xi , \eta),$ then $\xi  , \eta $ lie in the closure of $ \CC (\dot \g (T)).$\footnote{By negative curvature, the function $\a:\R \to (0,\pi), \a (t) := \angle_{\g (t)} (\xi, \eta)$ is increasing. There is $T_0$ such that $\a(T_0) = \pi /2.$ By comparison with the space of constant curvature $-a^2$, \[T_0 \geq -a \log \tan \angle_x(\xi, \eta) \sim - \log d_x(\xi, \eta).\] } We choose $\d=\d(x,y) <\d_0$ small enough so that one can choose $T > \max \{d(x,y),1\} + 4R_1.$ Then, Lemma~\ref{unifstrong} applies to the limits $k_\l(x,y, \xi ) $ and $k_\l (x,y, \eta) $ so that for $\eta, \xi $ with $d_x(\eta, \xi)<\d$, \[ \Big| \log  \frac {k_\l (x,y, \xi)}{k_\l (x,y, \eta)} \Big|  \; \leq \; CK^{-d(x,y)}  K^{-C\log d_x(\xi, \eta)} \; = \; C(x,y) (d_x (\xi , \eta))^ \a ,  \] 
where $\a = -C \log K>0 $. 
For $\eta, \xi$ with $d(\eta, \xi) > \d$, the estimate follows from Harnack inequality \ref{Harnack}.

 As $\l$ varies, by Lemma~\ref{unifstrong}, the functions $k_\l (x,y,z) $ are uniformly $\a$-H\"older continuous on a neighborhood of $\xi $ in $\M \cup \pp \M$ and depend continuously on $\l \leq \l_0$. The $\a'$-H\"older continuity in $\l$  follows for any $\a' <\a$.\end{proof}

Recall from (\ref{eq:Naim}) that $\th_x^\l (y , z) \; :=  \frac{ G_\l (y,z) }{G_\l (y,x) G_\l (x, z) } $ for $x,y,z \in \M$, $\l \leq \l_0$.  

\begin{prop}\label{Naim2} Fix $x \in \M, \xi \not = \eta \in \pp \M$, $\l \in [0,\l_0].$ As $y \to \xi , z\to \eta$, the following limit exists and defines 
 the Na\"im kernel $\th_x^\l (\xi, \eta) $:
 \[ \th_x^\l (\xi , \eta) \; := \lim _{y \to \xi, z \to \eta} \th_x^\l (y , z) = \lim _{y \to \xi, z \to \eta}  \frac{ G_\l (y,z) }{G_\l (y,x) G_\l (x, z) }. \]
 The limit is uniform in $\l$ on the set of triples $(x,\xi, \eta)$  with $d_x(\xi, \eta) $ bounded away from 0. 
Set, for $v \in SM,$ $\th _{\l _0}(v)  : = \th _{\g_{\wt v} (0)}^{\l_0}  (\wt v^-, \wt v^+) $  as $\theta_\l$ in (\ref{thetalambda}). 
 Then there is $\a' $ such that the mapping $\l \mapsto \th _\l $ is continuous from $[0,\l_0]$ to the space of $\a' $-H\"older continuous functions on $SM.$ \end{prop} 
 \begin{proof} Let us give a proof which is uniform for $\l$ up to $\l_0$. Observe  that, by (\ref{AndersonSchoen2}),  for $d_x (y,z) :=e^{-a(d(x,z)+d(x,y) - d(x,z))}$ bounded away from 0, the functions $\th_x^\l (y , z) $ are uniformly bounded. As before, by (\ref{unifstrong}), the functions $y,z \mapsto  \th_x^\l (y , z) $  are uniformly $\a$-H\"older continuous in $y$ and in $z$ as long as $d_x(y,z) $ remains bounded away from 0  and $ \th_x^\l (y , z)  \to  \th_x^{\l_0} (y , z) $ as $\l \to \l_0.$ The convergence and the continuity follow. Observe also that the function $\th _{\g_{\wt v} (0)}^\l  (\wt v^- , \wt v ^+) $ is $\G $-invariant and so $\th _\l $ is indeed a function on $SM$. Since $d_{\g_{\wt v} (0)}  (\wt v^- , \wt v^+ ) =1$, the mapping $\l \mapsto \th _{\g_{\wt v} (0)}^\l  (\wt v^- , \wt v^+ ) $ is continuous from $[0,\l_0 ] $ to the space of $\a''$-H\"older continuous functions  on $S\M$ endowed with the metric coming from the identification with $\pp \M \x \pp \M \x \R$ for some $\a'' < \a$. This identification being itself H\"older continuous (\cite{AnS} Proposition 2.1), the last statement of Proposition~\ref{Naim2} follows. 
\end{proof}
 For $v \in SM$, $x \in \M$ , $\xi, \eta \in \pp \M$, we set 
 \begin{equation}\label{theta} \th (v) \; := \; \th _{\l_0} (v), \;\;\; \th_x (\xi, \eta) : = \th^{\l_0}_x(\xi, \eta) . \end{equation} 

Fix $x,z \in \M , d(x,z) \geq 1$ and $\xi \in \pp \M$. The functions $y \mapsto k_\l (x,y, z )$ and $ y \mapsto k_\l (x,y, \xi) $ are $\l$-harmonic in $y$ in a neighborhood of $x$. Let $v \in S_x\M$. The directional derivative $\pp _v k_\l (x,., z) $ exists. Since $k_\l(x,y,z)$ is a $\l$-harmonic function of $y$ away from $z$, by Proposition \ref{Harnack},  $|\pp _v \log k_\l (x, y, z)|_{y=x}| \leq \log C_0 $ where the constant $\log C_0$ does not depend on $\l \in [0,\l_0]$.  Following   \cite {H1} Lemma 3.2, we have:

\begin{prop}\label{derivativeHolder}  For fixed $x \in \M$ and $\wt v\in S_{x}\M$, the mapping $\xi  \mapsto \pp _{\wt v} k_\l (x, y , \xi ) |_{y=x}$ is  $\a $-H\"older continuous, uniformly in $\l \in [0,\l_0]$ and $\wt v \in S_x\M$.  
Let us define 
\[ \varphi _\l (v) := -2 \pp _{\wt v } \log k_\l (\g _{\wt v}(0), \cdot , \g_{\wt v} (+\infty)) \; =\;  -2 \lim _{\e \to 0 } \frac{1}{\e}\log k_\l (\g _{\wt v}(0), \g_{\wt v} (\e), \g_{\wt v} (+\infty)) ,\]
where ${\wt v }$ is a lift of $v \in SM.$
Then
there is $\a' >0$ such that the  function $\l \mapsto \varphi _\l  $ 
is continuous from $[0,\l_0] $ to the space of $\a'$-H\"older continuous functions on $SM.$  \end{prop}

\begin{proof} 
Let
$x \in \M, v \in S_x\M.$ For $\e >0$, set $x_\e := \g_v (\e)$. Then, for $\xi \in \M,$
\[  \pp _v k_\l (x,., \xi)  +2\log C_0 \; = \;  \lim _{\e \to 0 } \lim_{z \to \xi }  \frac{\e^{-1} (G_\l(x_\e,z ) - G_\l (x,z)) +  2(\log C_0) G_\l (x, z) }{ G_\l (x,z)}. \]
Let $\g$ be the geodesic with $\g(0) = x, \g(+ \infty ) = \xi $. For $T>3,$ a point $z \in \CC (\dot \g (T)) ,$ and  $\e <1$, we write,  using (\ref{Poissonbarrier}) and Proposition~\ref{barrier} for $S:= \pp B(x,2 ) $ and $B(x,2) \subset \M \setminus \CC (\dot \g (3)) ,$
\begin{eqnarray*} &  &\frac{G_\l(x_\e,z ) - G_\l (x,z)}{\e}  +  2 (\log C_0) G_\l (x, z)\\ & &\quad \quad \quad  = \int_{S} \left(\int _{\pp \CC (\dot \g (3) )} G_\l(a,z) \, d\varpi ^\l_s(a)\right)\left[\frac{\rho ^\l_{x_\e} (s) - \rho^\l_x(s)}{\e}  +  2 (\log C_0) \rho^\l_x(s)\right] \, \, ds,\end{eqnarray*}  
where $\rho_x^\l$ is the density of the hitting measure with respect to the Lebesgue measure (see Proposition \ref{prop:8.8})
By (\ref{Harnackdensity}), the expression $\displaystyle \frac{\rho ^\l_{x_\e}(s) - \rho^\l _x(s)}{\e}  + 2 (\log C_0) \rho^\l _x( s) $ is nonnegative and  at most $4 (\log C_0) \rho^\l _x( s)$ if $\e$ is small enough. Moreover, by (\ref{AndersonSchoen2}), if $z \in \CC(\dot \g(2R_1 +3)),$ $ k_\l (x,a, z) \leq C_8^2 k_\l (x, a, \g ( R_1 +3)).$ Consider  $\eta $ close to $\xi$ in $\pp \M$. In the formula  
 \begin{eqnarray*} &  & \pp _v k_\l (x,., \eta)  + 2 \log C_0 \\ & &\quad  = \lim _{\e \to 0 } \lim_{z \to \eta } \int_{S} \left(\int _{\pp \CC (\dot \g (3) )} k_\l(x,a,z) \, d\varpi ^\l_s(a)\right)\left[\frac{\rho ^\l_{x_\e} (s) - \rho^\l_x(s)}{\e}  + 2(\log  C_0) \rho^\l_x(s)\right] \, \, ds ,\end{eqnarray*}
 the integrand  is at most $4 (\log C_0) C_8^2 k_\l (x, a ,\g(R_1+3) ) \rho ^\l _x(s) $ for all $\e$ small and all $z \in \CC(\dot \g(2R_1 +3)). $ Since \[ \int_{S} \left(\int _{\pp \CC (\dot \g (3) )} k_\l(x,a,\g(R_3+1)) \, d\varpi ^\l_s(a)\right)\rho^\l_x(s) \, ds  \; = \;  k_\l(x,x,\g(R_3+1))  \; = \; 1, \]  we may exchange the limits and the integrals. Set 
 \[ F(x,v,s ) \; := \;  \lim _{\e \to 0} \frac{\rho ^\l_{x_\e}(s) - \rho^\l _x(s)}{\e}  + 2 (\log C_0) \rho^\l _x( s) \; = \; \pp _v \rho_{x'}^\l (s)|_{x'=x} +  2(\log C_0) \rho^\l_x (s). \] 
 There is $\theta _0 $ such that, if $d_x (\xi, \eta ) \leq \theta _0 ,$ then $\eta  \in \ov {\CC(\dot \g(4R_1 +3))} \cap \pp \M  $ and we can find $z_n \to \eta$ with all $z_n \in \CC(\dot \g(2R_1 +3)). $ This gives, for $d_x (\xi, \eta ) \leq \theta _0 ,$
\[ \pp _v k_\l (x,., \eta) +  2 \log C_0 \; = \; \int_S \left(\int _{\pp \CC (\dot \g (3) )}  k_\l (x,a, \eta)  \, d\varpi ^\l_s(a)\right) F(x,v,s) \, ds. \]
 It follows from Lemma~\ref{unifstrong} and (\ref{AndersonSchoen2}) that for $\xi , \eta \in  \pp \M , d_x (\xi, \eta ) \leq \theta _0 ,$
 \begin{equation}\label{almost3.12} \frac { \pp _v k_\l (x,., \eta ) + 2 \log C_0 } { \pp _v k_\l (x,., \xi ) +  2 \log C_0 } \; \leq \; e^{C d_x(\xi, \eta)^\a}. \end{equation}
Assume $\pp _v k_\l (x,., \xi )  \leq \pp _v k_\l (x,., \eta ) $ and recall that $| \pp _v k_\l (x,., .) | \leq \log C_0 $. For $d_x(\xi, \eta) $ small enough, it follows from (\ref{almost3.12}) that $\pp _v k_\l (x,., \eta )  -\pp _v k_\l (x,., \xi )  \leq 3C (\log C_0) (d_x (\xi, \eta))^\a.$ The Proposition follows.  
\end{proof}

\begin{cor}\label{Pnonpositive} The pressure $P(\l_0) := P(\varphi_{\l_0})$ of the function $\vf _{\l _0}$ is non-positive. \end{cor} 
Indeed we know by Corollary~\ref{Pnegative} that the pressure of the function $\vf _\l $ is negative, and by Proposition~\ref{derivativeHolder} that the mapping $\l \mapsto \vf _\l $ is continuous at $\l_0$. 

\begin{cor}\label{cor:3.10} The measures $\mu_\l$ and the normalising constants $\Om_\l, \Upsilon_\l$ are continuous functions of $\l$ as $\l \to \l_0$ in $[0, \l_0].$
\end{cor}
\begin{proof} Indeed, the measures $\mu_{\l_0}$ satisfy the conditions in Proposition~\ref{measureNonu} and $\Om_{\l_0}$ satisfies the expression (\ref{eqn:Hopf}). Since  the functions involved are continuous by  Proposition~\ref{demiMartin} and Proposition~\ref{derivativeHolder},  Corollary~\ref{cor:3.10} follows. The argument is the same for $\Upsilon_\l$.
\end{proof}

We can now prove Theorem~\ref{expdecay} giving the exponential decay of $G_{\l_0} (x,y)$ with the distance. More precisely, we have:
\begin{prop}\label{prop:expdecay} 
Let $\tau _0 := \sup \{ \int \varphi _{\l_0} \, dm \}$, where the supremum is taken over all  $\bg$-invariant probability measures. Then, $\tau_0 <0$ and
 \[ \lim _{R \to \infty } \frac{1}{R}\log  \max \{ G_{\l_0} (x,y): d(x,y) = R \} \; = \frac{ \tau_0}{2}. \] \end{prop}
 
\begin{proof} First we prove that $\tau _0 <0 .$ 
First note that $\sup \int \vf_{\l_0} dm$ is attained by compactness of $M$. Suppose that $m_1$ attains the supremum of $ \int \vf_{\l_0} dm$ and that $\int \vf_{\l_0}\, dm_1 \geq 0$. Then $h_{m_1} + \int \vf_{\l_0}\, dm_1 \geq 0$.
However, since $P(\vf_{\l_0} )\leq 0$ by Corollary~\ref{Pnonpositive}, it follows that $h_{m_1}=0$ and $\int \vf_{\l_0} dm_1 =0$, and therefore $m_1$ is the equilibrium state of $\vf _{\l_0}$. This is a contradiction since $h_{m_1} >0$ if $m_1$ is an equilibrium state of a H\"older continuous function.
This proves that $\tau_0=  \sup\{ \int \vf _{\l_0} \, dm \} <0$. 

 It follows from the definition (\ref{pressure}) of the pressure that 
\[ \lim _{t\to \infty } \frac {1}{t} P( t \varphi _{\l_0} ) \; = \;  \tau _0. \] For $\tau_0 < \tau' <0$, we can find $T$ large enough that $P_T(\l_0) = P(T\varphi _{\l_0} /2) < T\tau' /2$. By letting $\l \to \l_0$ in Corollary \ref{powerT}, there exists a constant $C(T)$ such that for all $R \geq 1, x \in \M$,
\[ e^{-(R P_T(\l_0))} \int _{S(x, R)} G_{\l_0}^{T} (x, z) dz \; \leq C(T) .\] 

Set $$\tau (R)   : = \frac{1}{R }\max \{ \log G_{\l_0} (x,z): d(x,z ) =R \} .$$ By compactness, there exist $x,y$ with $ d(x,y) = R $ and $G_{\l _0} (x,y) = e^{R \tau (R)} $. We have, for $z \in S(x,R), d(y,z) \leq 1 ,$ \[ G_{\l _0} (x,z) \geq C_0^{-1} e^{R \tau (R) } \quad  {\textrm { and thus }} \quad  G_{\l _0}^T (x,z) \geq C_0^{-T} e^{TR \tau (R) }.\]
Therefore, we have for all $R \geq 1$, 
\[ C(T) \geq  e^{-RT\tau'/2 } \int _{S(x, R)\cap B(y,1) } G_{\l_0}^{T} (x, z) dz \;\geq  C_0^{-T}  e^{R T (\tau(R)  -\frac{\tau '}{2})} \Vol_{m-1} (S(x, R)\cap B(y,1) ).\]  Since for $R \geq 1, \Vol (S(x, R)\cap B(y,1) ) $ is greater than a positive constant,  this is possible only if $ \limsup _R \tau (R) \leq \tau ' /2 .$  Since $\tau ' >\tau _0 $ was arbitrary, this proves that 
 \[ \limsup   _{R \to \infty } \frac{1}{R}\log  \max \{ G_{\l_0} (x,y): d(x,y) = R \} \; \leq \frac{ \tau_0}{2} .\]

 Conversely, recall that invariant probability measures supported by single closed geodesics are dense in the set of invariant probability measures (\cite{S}). Therefore, for all $\e>0 $, there exists a {closed geodesic}, say of length $\ell $, such that for $v $ tangent to that {geodesic}, 
 \[ \int _0^\ell \varphi _{\l _o} (\gg_sv) \, ds \;  \geq \; (\tau_0 - \e) \ell  .\]

 Let $\wt v $ be a lift of $v$. The geodesic $\g_{\wt v} $ is a periodic axis and for all $j \in \N$,
\[ k_{\l_0} (\g_{\wt v}(j\ell) , \g_{\wt v}((j+1) \ell), \g_{\wt v}(+ \infty ) ) \; \leq \; e^{ - (\tau_0 - \e) \ell /2} .\]
 By Lemma~\ref{unifstrong}, 
 we have
\begin{eqnarray*} \frac {G_{\l_0} (\g_{\wt v}(j\ell) , \g_{\wt v}(N \ell))}{G_{\l_0} (\g_{\wt v}((j+1) \ell) , \g_{\wt v}(N \ell))} &\geq &e^{- CK^{(N-j)\ell } }  k_{\l_0} (\g_{\wt v} ((j+1) \ell),  \g_{\wt v}( j\ell), \g_{\wt v} (+ \infty ))\\ &\geq &e^{- CK^{(N-j)\ell }}    e^{ (\tau_0 - \e) \ell /2 }.\end{eqnarray*}
Since the sum $\Si _0^\infty  CK^{j\ell }  $ converges, we have
 \[ \frac {G_{\l_0} (\g_{\wt v}(0) , \g_{\wt v}(N \ell))}{G_{\l_0} (\g_{\wt v}((N-1) \ell) , \g_{\wt v}( N\ell))} = \prod_{j = 0}^{N-2} \frac {G_{\l_0} (\g_{\wt v}(j\ell) , \g_{\wt v}(N \ell))}{G_{\l_0} (\g_{\wt v}((j+1) \ell) , \g_{\wt v}(N \ell))} \; \geq C  e^{ N (\tau_0 - \e) \ell /2}. \]
This shows  that, for all $\e >0$, 
\begin{eqnarray*} \liminf  _{R \to \infty } \frac{1}{R}\log  \max \{ G_{\l_0} (x,y): d(x,y) = R \} &\geq & \liminf _{N \to \infty } \frac{1}{N\ell} \log G_{\l_0} (\g_{\wt v}(0) , \g_{\wt v}(N \ell)) \\ &\geq & \;  \frac{\tau_0 - \e}{2}.\end{eqnarray*}
 \end{proof}

\begin{cor}\label{lem:3.15}There exists $C>0$ such that for any $\l \in [0, \l_0]$ and $x, \xi, \eta$, there exists $x_0 \in [\xi, \eta]$ such that if $y$ is in the geodesic ray from $x$ to $\xi$ and   $d(x,y) \geq d(x, [\eta , \xi ]) + 4R_0, $ then 
$$\frac{k_{\l}(x,y,\eta)}{k_{\l}(x,y,\xi)} \leq C G^2_\l(x_0,y).$$ 
\end{cor}
\begin{proof}
We first claim that by $\delta$-hyperbolicity, there exist points $x_0 \in [\xi, \eta], x_1 \in [x, \eta], x_2 \in [y, \eta],  x_3 \in [x, \xi]$ such that the distance between them is bounded above by $3\d$. Indeed, for $x'$ in the geodesic from $\eta$ to $\xi$, the distance function $x' \mapsto d(x', [x,\xi])$ is a decreasing function. Let $x'$ the first point where $d(x', [x,\xi])\leq \d$ and choose $x_0 \in [x', \eta]$ to be the point $\d$-apart from $x'$. By definition, $\d < d(x_0, [x, \xi]) \leq 2\d$, thus there exists $x_3 \in [x, \xi]$ of distance $2\d$-close to $x_0$. Choose $x_1 \in [x, \eta], x_2 \in [y, \eta]$ $\d$-close to $x_0$. The claim follows.\\
Let $[x, \xi] \ni w  \to \xi$ and $[x, \eta] \ni z \to \eta$. Let us write $G(x,y)=G_{\l}(x,y)$ for simplicity.

Choose $\theta_0$ such that if $\angle_x (\xi, \eta) \leq \theta_0$, then $x$ is $R_0$-apart from $x_0,\cdots, x_3$. For $x, \xi, \eta$ such that $\angle_x (\xi, \eta) \leq \theta_0, $ by Theorem~\ref{Alano} (which gives estimates up to $C_4$ since $d(x,y) >d(x,x_i)+R_0 +3\d , i = 2,3$) and Harnack inequality (which gives estimates up to $C_H$), we have
\begin{eqnarray*}
\frac{k_\l(x,y,z)}{k_\l(x,y,w)}  &=& \frac{G(y,z)G(x,w)}{G(x,z)G(y,w)} \\ &\sim^{(C_4 C_H)^4}& \frac{G(y,x_0)G(x_0, z)}{G(x,x_0)G(x_0,z)}\frac{G(x,x_0)G(x_0,y) G(y,w)}{G(y,w)}=G^2(x_0,y).
\end{eqnarray*}
For $x, \xi, \eta$ such that $\angle_x (\xi, \eta) > \theta_0$, $d(y,x)>3 R_0-3 \d$ and $d(y,x_2)>3R_0 -3\d$, so that we have
$$\frac{k_\l(x,y,z)}{k_\l(x,y,w)} = \frac{G(y,z)G(x,w)}{G(x,z)G(y,w)} \sim^{C_4^2} \frac{G(y,x_2)G(x_2, z)}{G(x,z)}\frac{G(x,y)G(y,w)}{G(y,w)}\sim^{C_H^3}G^2(x_0,y).$$
\end{proof}

\begin{proof}[Proof of Theorem~\ref{thm:1.4}]  Recall that Martin compactification of the operator $\D -\l_0$ is given by all possible limits of $k_{\l_0} (x,y,z) $ as $z \to \infty.$ Proposition~\ref{demiMartin} and its proof show that there is a continuous mapping from the geometric compactification of $\M$ onto the Martin compactification. So it suffices to show that this mapping is  one-to-one. 
    If $\eta \neq \xi$, by Corollary~\ref{lem:3.15}, $k_{\l_0}(x,y, \eta)/k_{\l_0}(x,y, \xi) \to 0 $ as $y \to \xi$ and thus  $k_{\l_0}(x,., \xi) $ does not  coincide with $k_{\l_0} (x, . ,\eta)$.
The decomposition of positive $\l_0$-harmonic functions follows then by general Martin theory.
\end{proof}
Since by Proposition~\ref{prop:expdecay},  $G_{\l_0} (x_0, \cdot ) $ goes to 0 at infinity uniformly, we get the following estimate for small $d(x,y)$:
\begin{cor}\label{Gequivalent}
For any compact neighborhood $K$ of x, there is a constant $C=C(m)$ such that, if $y \in K, 0\leq \l \leq \l_0,$
\begin{equation}\label{eq:Gequivalent}
C^{-1}  \leq  (d(x,y))^{m-2} G_{\l} (x,y )  \leq C {\textrm{ for }} m >2, \quad  C^{-1} \leq \frac{ G_{\l} (x,y )}{1+ |\log d(x,y)| }  \leq C {\textrm{ for }} m =2.\end{equation}\end{cor}
\begin{proof} Observe that, for $x \not = y$,
\[ G_{\l_0} (x,y) \; = \;  \int_0^1 e^{\l_0 t}\Pp(t,x,y) \, dt + \int _1^\infty e^{\l_0 t}\Pp(t,x,y) \, dt \] and that the last term is uniformly bounded for $y \in K.$ Indeed, let $A$ be the diameter of $K$. Then, 
\begin{align*}  &   \int _1^\infty e^{\l_0 t}\Pp(t,x,y) \, dt\; = \;   e^{\l_0} \int _{\M} \Pp (1,x,z) G_{\l_0 }(z,y) \, d\Vol (z) \\
 = & \; e^{\l_0} \int _{B(x,A+1)} \Pp (1,x,z) G_{\l_0 }(z,y) \, d\Vol (z) + e^{\l_0} \int _{\M\setminus B(x,A+1)} \Pp (1,x,z) G_{\l_0 }(z,y) \, d\Vol (z)\\
\leq & \; e^{\l_0} \max_{B(x,A+1)} \Pp (1,x,z) \int _{B(y,2A+1)}  G_{\l_0 }(z,y) \, d\Vol (z)  + e^{\l_0} \max _{d(z,y) \geq A} G_{\l_0} (z,y).\end{align*}
We used (\ref{eqn2.3}) to bound uniformly $ \int _{B(y,2A+1)}  G_{\l_0 }(z,y) \, d\Vol (z) $ and Proposition~\ref{prop:expdecay} to bound $\max _{d(z,y) \geq A} G_{\l_0} (z,y) < \infty .$ For $0\leq \l \leq \l_0,$ $G_\l \leq G_{\l_0}$ and it suffices to show the estimate (\ref{eq:Gequivalent}) on $\int_0^1 e^{\l t}\Pp(t,x,y) \, dt.$ \\
Since the curvature is bounded, it follows from \cite{Mv} that for $0<t \leq 1, 0 < d(x,y) \leq A$ \[ \Pp(t,x,y) (4\pi t)^{m/2} e^{-\frac{d(x,y)^2}{4t}} \; \sim ^C \; 1.\]
Corollary~\ref{Gequivalent} follows by integration in $t$. \end{proof}
\begin{cor}\label{convolution} For any $A>0$, any $m \geq 2,$ there is a constant $C$ such that, for $d(x,y) < A, 0 \leq \l \leq \l_0$,
\[ \int _{B(x, 2A)} G_{\l} (x, z )G_{\l } (z,y) \, d\Vol (z) \; \leq \; C G_{\l} (x, y ).\] \end{cor}
Indeed,  by Corollary~\ref{Gequivalent}, it suffices to show that there is a constant $C$ such that
\begin{eqnarray*} \int_{B(x,2A)} \frac{d\Vol(z)}{(d(x,z)d(y,z))^{m-2}} &\leq &\frac{C}{d(x,y)^{m-2} }\quad \textrm{ { for }} m>2,\\   \int_{B(x,2A)} |\log d(x,z) \log d(y,z)| d\Vol (z) &\leq & C | 1+ \log d(x,y) |\quad  \textrm{ { for }} m= 2.\end{eqnarray*}
The statement reduces to the Euclidean case, where it can be shown by direct computation.

\section{Renewal theory}\label{sec:4}
In this section, we use uniform mixing of the geodesic flow $\gg_t$ that will be established in Appendix I (Section~\ref{Appendix}) to control the convergence in Proposition~\ref{renewal} as $\l $ goes to $\l_0$. 
Throughout the section, let us denote $\chi (t) := 1$  for $|t| \leq 1/2$ and $0$ otherwise. Let $\chi _{\d'} (t) =\chi (t/\d') .$ Let $\psi(t) : = \max \{ 1- |t|, 0 \} $.

Thanks to Proposition~\ref{derivativeHolder}, for $\l $ close to $\l _0$, the functions $\vf _\l$ are close to $\varphi_{\l_0}$ in the space $\K _\a$ of $\a$-H\"older continuous functions, for some $\a=\a_0 >0 $  (see Section~\ref{sec:7.1} for definition of $\K_\a$).

\begin{prop}\label{4.1} There exist $\a>0$ and $\d_0>0$ with the following property. For every $\e>0$, $f,h \in \K_\a$ positive $\a$-H\"older continuous functions, there exists $t_0=t_0 ( f,h, \e)$,
 such that for $t \geq t_0$, for any $\l \in [\l_0 -\d_0, \l_0]$,
 $$ \int_{SM} f h\circ \gg_t \, dm_\l \sim^ {1+\e} \int_{SM} f \,dm_\l \int_{SM} h \,dm_\l.$$ 
Indeed, $t_0$ depends only on $\e, ||f||_\a, ||h||_\a, \inf_\l \int f \,dm_\l ,  \inf_\l \int h \,dm_\l $, in particular is independent of $\l\in [\l_0 - \d_0 , \l_0]$.
\end{prop}

\begin{prop}\label{4.2} There exist $\a>0$ and $\d_0'>0$ with the following property. For every $\e>0$, $f,u,h \in \K_{\a}$ positive $\a$-H\"older continuous functions, there exists $t_0'=t_0' ( f,u, h, \e)$,
 such that for $t \geq t_0'$, for any $\l \in [\l_0 -\d'_0, \l_0]$,
	 $${\frac{1}{t} \int_0^t \left[  \int f \cdot (u \circ \gg_s) \cdot (h\circ \gg_t)  \,dm_\l \right] ds } \sim^{1+\e} {\int f \,dm_\l\int u \,dm_\l  \int h \,dm_\l } .$$
Indeed, $t_0'$ depends only on $\e, ||f||_\a, ||h||_\a,$ $\|u\|_\a, \inf_\l \int f \,dm_\l ,  \inf_\l \int u \,dm_\l$ and $ \inf_\l \int h \,dm_\l$, in particular is independent of $\l\in [\l_0 - \d_0' , \l_0]$.
\end{prop} 
\begin{proof}
Noting that
$$ \big| \frac{ \int f h\circ \gg_t}{\int f \int h} -1 \big| \leq C\frac{||f||_\a ||h||_\a }{1+|t|^c} \frac{1}{\int f \int h},$$
we deduce Proposition~\ref{4.1} from Proposition~\ref{theo:uniformmixing} to the equilibrium measure $m_\l$ associated to $\vf_\l$. Proposition~\ref{4.2}  follows from Corollary~\ref{theo:uniformavmixing} in a similar way.
\end{proof}
\subsection{Integral on large spheres with respect to Green functions}

\label{sec:4.1}
  Let us introduce some more notations: for $x \not = z \in \M$, denote by $v_z^x$ the unit vector in $S_z\M $ pointing towards $x$ and $pv_z^x$ its projection on $SM$. The mapping $z \mapsto v_z^x$ identifies $\M \setminus \{x\} $ with a subset of $S\M$. 
 \begin{theo}\label{renewal2} 
Given $\e'>0$ and positive H\"older continuous functions $f,h$ on $SM$, there exist $R(f,h,\e')$ and $\d( f,h, \e')$ such that if $R >R( f,h, \e')$ and $\l \in [\l_0 - \d( f,h,\e' ), \l_0]$, for all $x \in \M$, 
 \begin{equation}\label{eqn:5.4} \hspace{-1.2 in}  e^{-R P({\l} )}  \int _{S(x, R)}  f(\Dppi v_x^y) h( \Dppi v_y^x)   G_\l^2 (x, y) dy \quad \quad \sim^{(1+\e')^3} \end{equation} 
 \begin{equation*}\Om _{\l}   \int _{\pp \M }  f( \Dppi \circ \s ^{-1}_x \xi) d\mu _x^{\l }(\xi )\int_{M_0}\left( \int_{ \pp \M}  h (\Dppi \circ \s^{-1}_y \xi )   d\mu_y^{\l} (\xi) \right)
  d\mathrm{Vol}(y).
 \end{equation*} 
Moreover, $R(f,h,\e')$ and $\d( f,h, \e')$ depends  only on $\e', ||f||_\a, ||h||_\a, \inf f $ and  $\inf h$.  \end{theo}
 
The rest of Section~\ref{sec:4.1} is devoted to the proof of Theorem~\ref{renewal2}. Let us first reduce Theorem~\ref{renewal2} to Proposition~\ref{LC} below.

Fix $f,h$ positive and H\"older continuous. We  choose  $\d'_0 >0 $ such that, if $R >1$ and $|R-R'| < \d'_0$, then, for all $x \in \M$ and $\l \in [\l_0 - \d(f,h,\e'),\l_0],$
\begin{equation}\label{4.3}  e^{-R P({\l} )}  \int _{S(x, R)}  f(\Dppi v_x^y) h( \Dppi v_y^x)   G_\l^2 (x, y) dy \sim^{1+\e'}  e^{-R' P({\l} )}  \int _{S(x, R')}  f(\Dppi v_x^y) h( \Dppi v_y^x)   G_\l^2 (x, y) dy. \end{equation}

Then, for $\d' \leq 2\d'_0,$ we claim that \eqref{eqn:5.4} satisfies 
 \begin{eqnarray*}  
\eqref{eqn:5.4} & \sim^{1+\e'}  \frac{1}{\d'} \int _\R  \chi _{\d'} (s-R) e^{-sP(\l)} \left(\int _{S(x, s)} f(\Dppi v_x^y) h(\Dppi v_y^x) G_\l^2 (x,y) dy\right) \, ds  \\
& =\frac{1}{\d'}  \int_{\M} \chi _{\d'} (d(x,y)-R) e^{-d(x,y) P(\l)} f(\Dppi v_x^y) h(\Dppi v_y^x) G_\l^2 (x,y) d\mbox{Vol}(y)\\
&  \sim^{(1+\e')^2} \frac{1}{\d'}\int_{M_0} \Si (x,y,R,\d') \,  d\Vol (y),  
\end{eqnarray*}
where 
\begin{equation}\label{eqn:Sigma} \Si (x,y,R, \d') :=  \underset{ \{ (v,\rr) : v \in S_{\ppi x}    M \cap \gg_{-\rr} S_{\ppi y } M\}}  {\sum} \chi _{\d'} (R-\rr) f (v )(\th _\l^{-2}h)(-\gg_\rr v)\frac{d \mu_\l ^{uu } }{d\gg_{-\rr}  \mu_\l ^{uu}} ( v).\end{equation}

The claim follows since we can replace $ e^{-\rr P(\l)} G_\l^2 ( \g _{\wt v }(0) ,  \g _{\wt v } (\rr)) $ by $\frac{1} { \theta _\l^2(-\gg_\rr v )} \frac{d \mu_\l ^{uu } }{d\gg_{-\rr}  \mu_\l ^{uu}} (  v) $. Indeed,  we have,   by equation (\ref{flowholonomy}), $$ \frac{d \mu^{uu}_\l }{d\gg_{-\rr} \mu^{uu}_\l  } (v)  = e^{-\rr P(\l )} k_{\l }^2 (\g_{ \wt v} (\rr) ,  \g _{\wt v} (0) , \g_{\wt v} (\infty ) ) . $$ 
Furthermore, by Proposition~\ref{Naim2}, for given $\e'$, if $R$ is large enough (depending on $\e'$) and $|\rr-R| \leq \d' \leq 1$, 
\begin{eqnarray*} \frac{d\mu^{uu}_\l  }{d \gg _{-\rr} \mu^{uu}_\l } ( v) &  
=& e^{-\rr P(\l ) } \lim_{z \to \wt v^+ } \frac{G^2_\l(\g_{\wt v}(0), z)}{G^2_\l(\g_{\wt v}(\rr), z) G^2_\l(\g_{\wt v}(0), \g_{\wt v}(\rr))} \,
G^2_\l(\g_{\wt v}(0), \g_{\wt v}(\rr))  \\
&\sim^{1+\e'}& e^{-\rr P(\l ) } \theta^2_\l(-\gg_\rr v)\, G^2_\l(\g_{\wt v}(0), \g_{\wt v}(\rr)),
\end{eqnarray*}
where the approximation is uniform in $\gg_\rr v$ and $\l$.
 It follows that for $\d' <2\d_0,$ given $\e' >0$,  for all $R$ large enough and all $\l$ close enough to $\l_0$,
\[  (\ref{eqn:5.4} ) \;\sim ^{(1+\e')^2}  \frac{1}{\d'} \int _{M_0}  \Si (x,y,R,\d') \,  d\Vol (y).\]
We are reduced to show:
\begin{prop}\label{LC}  Given $\e' >0$ and positive H\"older continuous functions $f ,h$ on $SM$,  there exist $R_0= R_0 (f,h, \e') $, $\d=\d( f, h, \e')>0$ and  $\d', 0 < \d' <2\d'_0 ,$  such that for $R\geq R_0$, all $x,y\in \M$ and all $\l \in [\l_0-\d, \l_0]$, 
\[
  \Si (x,y,R,\d')  \sim^{(1+\e')}  \Om _\l \d'  \left(\int_{S_{\ppi x} M }  f(v)\, d\mu _\l^{uu} (v) \right) \left( \int _{S_{\ppi y} M } h(-u) \, d\mu _\l^{ss} (u)\right)\\
\]
for $  \Si (x,y,R,\d')$ defined in \eqref{eqn:Sigma}.
Moreover, $R_0(f,h,\e')$ and $\d( f,h, \e')$ depend  only on $\e', ||f||_\a, ||h||_\a, \inf f $ and  $ \inf h.$
\end{prop}

The right hand side in Proposition~\ref{LC} is the same as 
$$ \d' \Om _{\l}   \int _{\pp \M }  f( \Dppi \circ \s ^{-1}_x \xi) d\mu _x^{\l }(\xi ) \int_{ \pp \M}  h (\Dppi \circ \s^{-1}_y \xi )   d\mu_y^{\l} (\xi)$$
 by (\ref{eqn:2.9}) and (\ref{eqn:2.10}).

Theorem~\ref{renewal2} follows from Proposition~\ref{LC} and the previous discussion by integrating the approximation in $y$ over a fundamental domain $M_0$.

\begin{proof}
We combine ideas of \cite{Ma} and Section III in \cite {L}. Choose $\e $ such that $(1+ \e)^{61} \leq 1+\e' .$  
Proposition~\ref{LC} follows from Proposition~\ref{4.1} applied to the non-negative H\"older continuous functions $F^\pm_\l, H^\pm_\l$ with the property that  there exist constants $C, \a, \g_0, \g'_0, \g$ such that for all $x,y\in \M$ and all $\l \in [0, \l_0]$, the following (1)-(5) holds.
\begin{enumerate}
\item $||F^\pm_\l||_\a < C, ||H^\pm_\l||_\a < C,$ 
\item $ \int F^\pm_\l dm_\l >C^{-1}, \int H^\pm_\l dm_\l >C^{-1}.$
\item\begin{eqnarray*} &  & \Om_\l \d' \g_0 (1+\e)^{-14}  \int _{S_{\ppi x} M} f( v) d\mu _\l^{uu} (v)\; \leq \; \int 
F^-_\l   d\m_\l \\ & \leq &\; \int F^+_\l  d\m_\l \; \leq \;  \Om_\l \d' \g_0 (1+\e)^{14} \int _{S_{\ppi x} M} f(v) d\mu _\l^{uu} (v). \end{eqnarray*}
\item \begin{eqnarray*} &  &\Om_\l \g \g'_0 (1+\e)^{-14} \int _{S_{\ppi y} M} h(-u) d\mu _\l^{ss} (u)\; \leq \; \int H^-_\l d\m_\l \\ & \leq &\; \int H^+_\l d\m_\l \; \leq \;  \Om_\l \g \g'_0  (1+\e)^{14}  \int _{S_{\ppi y} M}  h(-u) d\mu _\l^{ss} (u). \end{eqnarray*}
\item There is $R(\e)$ such that for $R \geq R(\e)$,
\begin{eqnarray*} 
 (1+\e)^{-30} \int F^-_\l H^-_\l\circ \gg_R \, dm_\l &\leq & \Om _\l \g \g_0 \g_0'  \Si (x,y,R,\d'(1+\e)),\\
 \Om _\l \g \g_0 \g_0'  \Si (x,y,R,\d') &\leq & (1+ \e)^{30} \int F^+_\l H^+_\l\circ \gg_R \, dm_\l. \end{eqnarray*}
\end{enumerate}

Let $a_4$ be the contraction rate of the stable submanifold: $d(\gg_t v, \gg_t v') \leq e^{-a_4 t}d(v,v')$ for $v,v'$ close enough on the same stable submanifold.
We choose $\d' <2\d'_0 $ with $e^{a_4 \d'} < 2$ and  such that, for all $\xi \in \pp \M,$ all $\l \in [0,\l_0]$, for $ d(v,v') <2 \d', d(x,x') <2\d' ,$ \[  \frac{f(v')}{f (v)}, \frac{h(v')}{h(v)} ,\frac{\th _\l^2(v')}{\th _\l^2 (v)}, k_\l (x,x', \xi), e^{P \b(x, x', \xi) }  \sim^{1+\e} 1,   \] 
where $P := \underset{ \l \in [0,\l_0]} \inf P(\l)<0.$ 
\begin{remark}\label{rmk:8}Dependency of $\d'$ on $\inf f, \inf h$ in Theorem~\ref{renewal2} comes from the choices in the paragraph above and the choice of $\d'_0$ at the beginning of the proof of Theorem~\ref{renewal2}.
\end{remark}
The functions $F^\pm_\l (v), H^\pm_\l(u)$ will approximate $\th_\l^{-2} f (v), \th_\l^{-2} h (-u)$ respectively, on the $\d'$-neighborhoods $N_{\d'} (S_{\ppi x}M)$, $N_{\d'} (S_{\ppi y}M)$ of $S_{\ppi x} M$, $S_{\ppi y}M$, respectively. 

For $w \in N_{\d'}(S_{\ppi x} M$), there exist a unique $v \in S_{\ppi x} M,$ and $v' \in W^{ss}_{loc}(v), t$ such that $v' = \gg_t w$. 
 Similarly, if $w \in N_{\d'} (S_{\ppi y} M)$, then there exists a unique triple $(u,u',s),  u \in S_{\ppi y}  M,  u' \in W^{uu}_{loc}(u)$  such that $u'= \gg_{s}(w).$

By the H\"older regularity of the strong stable and the  strong unstable foliations, the systems of coordinates  $(v,v',t)$ (respectively $(u,u',t)$) are H\"older continuous, uniformly in $x$ and $y$. 

{\it {Step 1. There exist   $\g_0, \g_0' >0 $ and non-negative H\"older continuous functions $a_\pm, b_\pm$ supported on $N_{\d'} S_{\ppi x} M$, $N_{\d'} S_{ \ppi y } M$, respectively, 
such that for all $v \in S_{\ppi x}M $ and $u \in S_{\ppi y}M,$}} \begin{equation}\label{positivemass} \int _{W^{ss}_{loc} (v)}  a_\pm (w)\, d\mu^{ss}_\l (w) = \g_0 (1+\e)^{\pm 1} , \quad   \int _{W^{uu}_{loc} (u)} b_\pm (w) \, d\mu^{uu}_\l (w) = \g_0' (1+\e)^{\pm 1}.\end{equation}
{\it{ Moreover, the H\"older exponent and the H\"older coefficient of $a_\pm, b_\pm $ are bounded uniformly in $x,y, \l$. }}

We denote  $d_{ss}$ (respectively $d_{uu} $) the induced metric on strong stable manifolds $W^{ss}$ (respectively on strong unstable manifolds $W^{uu}$).
\begin{lem}\label{hHolder}
Let $$h_{r,v,\l} = \int_{W^{ss}_{loc}(v)} \psi \left(\frac{d_{ss}(v, v')}{r}\right) d\mu^{ss}_\l (v').$$

The map $(r,v, \l) \mapsto h_{r,v,\l}$ is continuous in $r, v$ and $\l$. For a fixed $r$,  the function $v \mapsto h_{r,v,\l } $ is H\"older continuous, uniformly in $\l \in [0,\l_0]$. As $r $ varies from $0$ to $\d'$, the function $r \mapsto h_{r,v,\l }$ is  increasing and admits right and left derivatives that are bounded below by a positive constant  uniformly in $v,\l$ and $r$ away from zero. \end{lem}
\begin{proof} The continuity is as in Margulis's Lemma 7.1 in \cite{Ma2}(p.51). The proof also yields H\"older continuity in $v$. Indeed, $W^{ss}_{loc }(v) $ depends on $v$ in a H\"older continuous way and if $v_1, v_2 $ are close, the holonomy $H_1^2$ from $W^{ss}_{loc }(v_1) $ to $W^{ss}_{loc} (v_2) $ along $W^{cu}$  is H\"older continuous, and satisfies for $v_1', v_1'' \in W^{ss}_{loc} (v_1) ,$ 
\[ d (v_2, H_1^2 v_1) \leq C (d(v_1, v_2))^\a , \; {\textrm { and }} \; |d(H_1^2 v'_1, H_1^2 v''_1 ) -d(v_1', v_1'') | \leq C  (d(v_1', v_1'' ))^\a.\] Moreover  the logarithm of the Radon Nikodym derivatives of the measure $(H_1^2)_\ast \mu_\l  ^{ss} (v_2')  $ with respect to $ \mu_\l  ^{ss} (v_2')  $ is given by  $$
\log  \rho_\l(v_2, H_1^2 v_1)   =  \log k^2_\l (v_2, H_1^2 v_1, \xi) +P(\l)\beta(v_2, H_1^2 v_1,\xi)$$ (see (\ref{holonomy})) and thus it is at most proportional to $ d (v_2, H_1^2 v_1) $ (uniformly in $\l$). 
Since $d (v_2, H_1^2 v_1) \leq C (d(v_1, v_2))^\a$, we can report in the definition of $h_{r,v,\l}$ and see that, for $v_1, v_2 $ close, \[ |h_{r,v_1,\l} - h_{r,v_2 ,\l }| \; \leq \; C(r) (d(v_1,v_2) )^{\a} , \] 
where the constant $C(r) $ is uniform in $\l \in [0, \l _0] $ and goes to infinity as $r \to 0 $.

 Direct computation shows that, as $r $ varies from $0$ to $\d'$, the function $r \mapsto h_{r,v,\l }$ is 
  increasing and admits left and right derivatives    given by $$\frac{\pp }{\pp r} h_{r,v,\l}|_{r-}  =  \lim _{r' <r, r' \to r}  \int_{W^{ss}_{loc}(v)} \frac{1}{r'} d_{ss}(v, v')\chi _{d(v,\cdot) \leq r'} (v') \;d\mu^{ss}_\l (v') $$  and 
 $$\frac{\pp }{\pp r} h_{r,v,\l}|_{r+}  =   \int_{W^{ss}_{loc}(v)} \frac{1}{r} d_{ss}(v, v')\chi _{d(v,\cdot) \leq r} (v') \;d\mu^{ss}_\l (v') .$$ 
The left and right derivatives are bounded from below by a positive constant uniformly in $v,\l$ and $r$ away from $0$.
 \end{proof}
  
For given $\g_0>0$, choose $r^\pm_\l(v, \g_0)$ such that $h_{r^\pm_\l (v, \g_0), v, \l} = \g_0 (1+\e )^{\pm 1} $.
Now choose $\g _0$ so that $r^\pm_\l(v, \g _0) <  \e \d'/2$ for all $v$ and $\l$. Set $r^\pm_\l (v) := r^\pm_\l (v, \g_0 ).$ 
By the Implicit function theorem with H\"older coefficients, 
\footnote{We have $h_{r(v), v} =\g_0 =  h_{r(v'),v'} $ so that 
$ |h_{r(v),v } - h_{r(v'),v} | = |h_{r(v'),v} - h_{r(v'),v'}| \leq C(d(v,v'))^\a, $ 
with uniforms $C, \a$. But $ |h_{r(v),v } - h_{r(v'),v} | $ is greater than $|r(v) - r(v') | $ times the derivative at $r$ of $r \mapsto h_{r,v}$ and the derivative is bounded from below.}
the functions $r^\pm_\l (v) $ are H\"older continuous uniformly in $\l $ for $\l \in [\l_0 - \d(\e), \l_0]$ and $v$.

Now for $w = (v, v',t) \in N_\d(S_{\ppi x} M), \l \in [ \l_0- \d(\e), \l_0]$, define 
$$ a^\pm_\l (w) = \psi \left( \frac{d_{ss} (v, v')}{r^\pm_\l (v)} \right).$$

Properties similar to Lemma~\ref{hHolder} holds for
the function $$(r,u, \l) \mapsto h_{r,u,\l} = \int_{W^{uu}_{loc}(u)} \psi \left(\frac{d_{uu}(u, u')}{r}\right) d\mu^{uu}_\l (u'),$$thus we can define $r'^{\pm}_\l(u)$ analogously: $\g_0'$ is chosen so that $r'^{\pm}_\l(u,\g_0')<\e \d'/2$ and  $r'^{\pm}_\l(u)$ is such that $h_{r^\pm_\l (u), u, \l} = \g_0' (1+\e )^{\pm 1} $.
For $w=(u,u',s) \in  N_\g(S_{\ppi y} M)$, define
$$ b^\pm_\l (w) = \psi \left( \frac{d_{uu} (u, u')}{r'^\pm_\l (u)} \right).$$
The functions $a^\pm, b^\pm$ satisfy the properties of  Step 1. 
$\qed$ 

\begin{remark}\label{remark:lemma4.5} For $\zeta >0 $ small, set, for $t \in  \R,$  $\wt {\psi} _\zeta ^\pm (t) := \max \{1\pm \zeta -|t|, 0\}.$ For $v \in SM$, there are  unique $\zeta _\l^\pm (v)$ such that \[ \int_{W^{ss}_{loc}(v)} \wt{\psi }_{\zeta _\l^\pm (v)} ^\pm \left(\frac{d_{ss}(v, v')}{r_\l^\pm (v)}\right) d\mu^{ss}_\l (v') = \g_0 (1+\e)^{\pm 2} .\] We have an analogous property in coordinates $(u,u',s).$
By continuity,  we can choose $\zeta _0$, $ \zeta _0 := \inf \{ \zeta _\l^\pm (v), \zeta _\l^\pm (u) \}$ such that for all $u,v \in SM$, all $\l \in [0, \l_0],$
\begin{eqnarray*} \g_0 (1+\e)^{-2} &\leq &  \int_{W^{ss}_{loc}(v)} \wt{\psi }_{\zeta _0}^- \left(\frac{d_{ss}(v, v')}{r_\l^-(v)}\right) d\mu^{ss}_\l (v') \\ &\leq  &  \int_{W^{ss}_{loc}(v)} \wt{\psi }_{\zeta _0}^+ \left(\frac{d_{ss}(v, v')}{r_\l^+(v)}\right) d\mu^{ss}_\l (v')  \leq     \g_0 (1+\e)^{2} \\
 \g'_0 (1+\e)^{-2} &\leq &   \int_{W^{uu}_{loc}(u)} \wt{\psi }_{\zeta _0}^- \left(\frac{d_{ss}(u, u')}{r_\l'^-(v)}\right) d\mu^{ss}_\l (u')  \\ & \leq  & \int_{W^{uu}_{loc}(u)} \wt{\psi }_{\zeta _0}^+ \left(\frac{d_{ss}(u, u')}{r_\l'^+(v)}\right) d\mu^{ss}_\l (u')   \leq   \g'_0 (1+\e)^{2}.
 \end{eqnarray*}
 Observe that, given $(M, g)$, the value of $\zeta _0$ depends only on our choices of $\e, \g_0$ and $\g'_0$.
\end{remark}

{\it{Step 2. Definition of $F^\pm_\l, H^\pm_\l $ and Property (1)}}

Consider  Lipschitz continuous  $\chi _\pm (t) $ on $\R$ such that, for all $t \in \R$, $$\chi_{(1+\e)^{-2}}(t) \leq \chi _-(t)  \leq \chi_{(1+\e)^{-1}}(t) \leq  \chi (t) \leq  \chi_{(1+\e)}(t) \leq \chi _+(t)  \leq  \chi_{(1+\e)^2}(t). $$ Now for $w = (v,v', t)$, define
$$F^{\pm}_\l (w) = \chi _\pm (t/\d')a_{\pm}(v')( \th _\l^{-2} f)(v)$$
and for $ w = (u,u',s)$, $$H^\pm_\l(w) = \chi_\pm (s/\g) b_\pm(u')( \th _\l^{-2}h)(-u),$$
for some $\g< \d'\e/2$.

Recall that the systems of coordinates  $(v,v',t)$ and $(u,u',s)$ are H\"older continuous uniformly in $x$ and $y$. The functions $F^\pm_\l, H^\pm_\l $ in those coordinates are compositions of H\"older continuous functions ($\psi, f, h$) and of the functions $r_\pm, r'_\pm$ that depend on $v$ in a H\"older continuous way, uniformly in $\l \in [ 0, \l_0]$ by Step 1, which proves Property (1).

{\it{ Step 3. Properties (2), (3) and (4)}}

Recall that under Hopf parametrization introduced in Section~\ref{sec:2}, if we let $x_0=x$, the lift $\wt m_\l $ of $m_\l$ to $S\M$ is given by $$d\wt m_\l (\xi, \eta, t) =
 \Om _\l (\th_x^\l)^2 (\xi ,\eta) e^{2P(\l) (\xi | \eta)_x} [d\mu^\l _x (\xi) \x d\mu^\l_x(\eta)\x dt ] .$$

Consider $ \wt w =  \wt w(\xi, \eta, t)$ close to $S_x\M$ and write the coordinates $(v,v',t)$ of $ w = \Dppi \wt w$ as:
\[ v = \Dppi (\sigma _x^{-1} (\eta)), \quad v' = \Dppi (W^{ss} (\sigma _x^{-1} (\eta) )\cap \g_{[\xi, \eta]} ), \quad t = t.\]
In particular, $w $ is close to $v$ and 
$$ \th_x^\l (\xi , \eta) = \th _\l (w) k_\l (x, p(w), \xi )  k_\l (x, p(w), \eta) \sim^{(1+\e)^2} \th _\l (w) \sim^{(1+\e)^3} \th _\l (v) ,$$ and $$e^{-P(\l ) (\xi , \eta )_x } \sim ^{(1+\e)^2} 1 .$$

We see that the measure $\wt m_\l $ has a density $\sim^{(1+\e)^8} \Om _\l \th _\l^{2} ( v) $ 
with respect to the product measure $ d\mu _x^\l (\xi) \x d \mu _x^\l (\eta) \x dt.$
When we change coordinates from the Hopf parametrization $(\xi, \eta, t)$ to the coordinates $(v,v',t)$ in a neighborhood of $S_xM$, the  mapping $(\eta, t) \mapsto (v,t) $ sends the measure $d\mu _x^\l (\eta) \x dt $ to the measure $d\mu^{uu}_\l (v) \x dt$ (see equation \ref{eqn:2.9}), the mapping $\xi \mapsto v'$ sends the measure $d\mu _x^\l $ to a measure with density $\sim ^{(1+\e )^4} 1$ with respect to the measure $d\mu _\l ^{ss} (v') $.
This implies that   in the neighborhood of $S_{\ppi x} M$, the measure $ m_\l$ in the coordinates $(v,v', t) $ has a density  
$ \sim^{(1+\e)^{12}}$ with respect to the measure 
$$ \Om _\l \th^2_\l (v)  [d\mu_\l ^{uu} (v) \x d\mu_\l^{ss}(v')\x dt ].$$
 
Since $\d' (1+\e)^{-1} \leq  \int \chi_- \left(\frac{t}{\d'} \right) dt \leq  \int \chi_+ \left(\frac{t}{\d'} \right) dt \leq \d'(1+\e) $, it follows that 
\begin{eqnarray*}
 \int F^+_\l  (w,\l ) d\m_\l &\leq &(1+\e )^{12}\Om _\l \int \chi_\pm \left(\frac{t}{\d'} \right) dt  \int_{S_{\ppi x M}} \left( \int _{W^{ss}_{loc}(v)} a^\pm_\l (v') d\mu _\l^{ss} (v') \right) f( v) d\mu ^{uu}_\l (v) \\
& \leq & {(1+ \e)^{14}}  \Om_\l  \d' \g_0 \int _{S_{\ppi x} M} f( v)  d\mu _\l^{uu} (v),
\end{eqnarray*} and
$$
 \int F^-_\l  (w,\l ) d\m_\l \geq  {(1+ \e)^{-14}}  \Om_\l  \d' \g_0 \int _{S_{\ppi x} M} f( v)  d\mu _\l^{uu} (v).
 $$
 
Similarly, in the $\d'$-neighborhood of any lift of $S_{ \ppi y} M$, we have, in the $(u,u',s)$ coordinates, where $ u \in S_y \M, u' \in W^{uu}_{loc}(u), |s| \leq 2\d' $, 
$$dm_\l (u,u',s) \sim^{(1+\e)^{12}}
 \Om _\l \th^2_\l (u)  [d\mu_\l ^{uu} (u') \x d\mu_\l^{ss}(u)\x ds ].$$

The analog computation yields that
\begin{eqnarray*} &  & (1+\e)^{-14} \Om_\l \g \g'_0 \int _{S_{\ppi y} M} h(-u)  d\mu _\l^{ss} (u)\; \leq \; \int H^-_\l d\m_\l \\ & \leq &\; \int H^+_\l d\m_\l \; \leq \;  (1+\e)^{14} \Om_\l \g \g'_0   \int _{S_{\ppi y} M}  h(-u)  d\mu _\l^{ss} (u). \end{eqnarray*}
This shows Properties (3) and (4).
Property (2) follows as $ \int f d\mu_\l^{uu}  $ and $ \int h d\mu _\l^{ss} $ are bounded away from 0, uniformly in $ x, y $ and $ \l \in [0, \l_0]$ by Corollary~\ref{boundonmeasures}. 

{\it{ Step 4. Preparation for property (5)}}

We have to estimate \[ \Si (x,y,R,\d') = \underset{ \{ (v,\rr) : v \in S_{\ppi x} M \cap \gg_{-\rr} S_{\ppi y } M\}}{\sum} \chi _{\d'} (R-\rr)  f (v )(\th _\l^{-2}h)(-\gg_\rr v)\frac{d \mu_\l ^{uu } }{d\gg_{-\rr}  \mu_\l ^{uu}} ( v). \]
For the second inequality of property (5), for each $v_0 \in  S_{\ppi x} M \cap \gg_{-\rr} S_{\ppi y } M$ for some $\rr, |\rr-R| < \d'/2, $ let $$ B(v_0) \; := \; \{ w \in SM, d(\gg_\rr w, \gg_\rr v_0  ) \leq 2\d' \; {\textrm {for}} \; 0 \leq \rr \leq R \} .$$ If $\d'_0$ is small enough, the sets $B(v_0), B(v_0') $ associated to distinct $v_0 , v_0'$ are disjoint by expansivity of $\gg_t$. We will show in Step 5 that for each such $v_0$, 
\begin{equation}\label{F+H+} f (v_0 )(\th _\l^{-2}h)(-g_\rr v_0)\frac{d \mu_\l ^{uu } }{d\gg_{-\rr}  \mu_\l ^{uu}} ( v_0) \; \leq \; \frac{(1+\e)^{30}}{\Om _\l \g \g_0 \g_0' }  \int_{B(v_0)} F^+_\l H^+_\l\circ \gg_R dm_\l.\end{equation}
 The second inequality of Property (5) follows by summing over all possible $v_0$.

For the first inequality of property (5), assume $F^{-}_\l(w) H^-_\l(\gg_Rw) \neq 0$. Then, we claim that 
there is a unique $ v_0 \in S_{px}M$  and $\rr \in \R_+$ such that $g_\rr v_0 \in S_{py}M, w \in B(v_0),$ {$ |R-\rr| < (1+\e) \d'/2 $}. We will show in Step 5 that the  following equation holds
 \begin{equation}\label{F-H-} (1+\e)^{30} f (v_0 )(\th _\l^{-2}h)(-g_\rr v_0)\frac{d \mu_\l ^{uu } }{d\gg_{-\rr}  \mu_\l ^{uu}} ( v_0)\;  \geq \; \frac{1}{\Om _\l \g \g_0 \g_0' }  \int_{B(v_0)} F^-_\l H^-_\l\circ \gg_R dm_\l.\end{equation}
 The first inequality of Proposition (5) follows since the union of all $B(v_0) $ covers the set where $F^{-}_\l H^-_\l\circ \gg_R$ does not vanish.

To prove the claim, by negative curvature, it suffices to find a vector $v_0$ such that $d( w,  v_0  ) \leq 2\d' ,$ and $d(\gg_R w, \gg_R v_0  ) \leq 2\d' .$ The vector $v_0$ will be found at the intersection of  $S_{px }M $ with {$ \cup_{\tau, |\tau | \leq \d'/2} \gg_{-R+\tau }S_{py }M$}. Using the coordinates $(v,v',t)$ of $w$ and $(u,u',s) $ of $\gg_Rw$, observe that $d(\gg_{R} v , \gg_R v') < e^{-Ra_4}\d'$, $\gg_R v' = \gg_{s-t} u'$ and that $W^{uu}_{loc} (\gg_{s-t} u') $ intersects $\gg_{s-t} S_{py} M$  at $\gg_{s-t}u$ with $d_{uu} (\gg_{s-t}u, \gg_{s-t}u') < \d' \e / (1+\e)$. \footnote{ We have  $d_{uu} (u, u') < \frac{\d' \e } {2(1+\e)}$ and the $d_{uu} $ distance is expanded under $\gg_{s-t} $ by less than $e^{a_4 \d'} <2.$} 
For $R$ large enough,  the manifolds  $W^{uu}_{loc} (\gg_Rv'), W^{uu}_{loc} (\gg_Rv)  $ and $\gg_{R}S_{px} M$ are so  close that they all intersect {$ \cup_{\tau, |\tau | \leq 2\d'\e / (1+\e)} \gg_{\tau }\gg_{s-t}S_{py }M$} and the distances between the intersections is smaller than $\d' \e /16.$ We have found a point $v_0 \in S_{px}M $ and $\rr$ such that $\gg_\rr v_0 \in S_{py} M$. 
 The value of $\rr$ satisfies $ |\rr-R+t -s| \leq \d'\e /8$. Since $|s| \leq \g/2 < \d' \e /4$ {and $|t| < \d'/2$}, we have indeed  {$|R-\rr| < \d'/ 2  + 3\d' \e /8 < (1+\e)\d' /2.$}

The proof of Property (5) reduces to the proof of equations \ref{F+H+} and \ref{F-H-}.

{\it{ Step 5. Property (5): Proof of equations \ref{F+H+} and \ref{F-H-}}}

Fix $v_0 \in  S_{\ppi x} M \cap \gg_{-\rr} S_{\ppi y } M$ for some $\rr, |\rr-R| < \d'/2 .$
 Using the coordinates $(v,v',t)$ of $w$ and $(u,u',s) $ of $\gg_Rw$, we write  \begin{eqnarray*} &  \int _{B(v_0)} F^\pm_\l (w)  H^\pm_\l (\gg_R w ) \, dm_\l(w)  \; = \\  
& \int_{B(v_0)} \th _\l^{-2} f (v(w)) \th _\l^{-2} h (-u(\gg_Rw)) \chi _\pm (\frac{t (w)}{\d'}) \chi _\pm (\frac{s(\gg_Rw)}{ \g}) a_\pm (v'(w)) b_\pm (u'(\gg_Rw)) dm_\l(w).\end{eqnarray*}
  and we calculate this integral up to $(1+ \e)^{30}$.

Firstly, the functions $f(v(w)),h(-u(\gg_Rw)),\th_\l(v(w))$ and $\th_\l(-u(\gg_Rw))$ vary with ratio less than $(1+\e)$ on each $B(v_0) $.
Secondly, the measure $dm_\l (w) $ is the product of  the Lebesgue measure on the direction of the flow 
 and some measure on transversals, which we denote by $dm^\perp_\l (w) $. Furthermore, inside each geodesic intersected with $B(v_0)$, $t(w) - s(w) $ is constant. Recall that $\g < \d' \e/2.$ If there is $w \in B(v_0)$ with $t(w) \leq \d'/2$ such that $s(\gg_\rr w) \leq \g/2$ for some $\rr$ close to $R$, we still have $t(\gg_\tau w) \leq \d'/2$ and $s(\gg_{\rr +\tau} w) \leq \g/2$ for an interval of length $\g$ of values of $\tau$ unless $t(w) \geq \d'/2 - \g$ or $t(w) \leq -\d'/2 +\g.$ In all cases, we have 
$\int  \chi _- (t /\d') \chi _- (s/ \g)\, dt  \leq \int \chi_{+}(s/\g) ds \leq  (1+\e)^2\g , \int  \chi _+ (t /\d') \chi _+ (s/ \g)\, dt  \geq \int \chi_{-}(s/\g) ds \geq  (1+\e )^{-2} \g .$

It remains to estimate 
$ \int_{(B(v_0))^\perp }  a_\pm (v'(w)) b_\pm (u'(\gg_Rw)) dm^\perp_\l(w), $ where $\perp $ is a projection on some well chosen transversal to the flow direction in $v_0$. 
For $d(w, v_0) < 3\d'$, define $v''(w) = W^{ss} (v_0) \cap W^{cu} (w), u''    (w) = W^{uu} (v_0) \cap W^{cs} (w).$ For a transversal to the flow in $v_0$, the system $(v'',u'') $ form a system of coordinates in the neighborhood of $v_0$.

As before, the measure $\m_\l$ restricted to $B(v_0)$ satisfies
\begin{equation}\label{localproduct} dm_\l (u'',v'',t) \sim^{(1+\e)^{4}}
 \Om _\l \th^2_\l (v_0)  [d\mu_\l ^{uu} (u'') \x d\mu_\l^{ss}(v'')\x dt]. \end{equation}
We claim that if $R$ is large enough, then $ d(v'(w),v''(w)) \leq \zeta_1,$ where $\zeta _1 $ will be chosen later. Indeed, $v' (w)$ and $v'' (w)$ are on the same central unstable manifold.  There is $v''' \in W^{uu} (v'') $ and a time shift $\tau '$ such that $v' = \gg_{\tau'} v'''.$  We have $d(v'(w),v''(w)) \leq d(v''(w),v'''(w)) + \tau '.$ For $R $ large enough  $d(v''(w),v'''(w)) < \zeta _1 /3.$ To estimate $\tau'$, observe that this is the same time shift as the one between $\gg_t v'' $ and $\gg_t v''',$  i.e.  the intersections of $W^{ss} (\gg_tv)$ and $W^{ss} (\gg_t v_0)$ with the same central unstable manifold. The points $\gg_t v$ and $\gg_t v_0$ are $\d'$-close, since they are both $\d '$-close to $\gg_t w$. The  time shift as the one between $\gg_t v'' $ and $\gg_t v'''$ is of the order of the sum of $d(\gg_t v_0 , \gg_t v'') $ and the distance between $\gg_t v$ and $W^{uu} (v_0).$ Both distances can be made smaller than $\zeta _1 /3 $ by choosing $R$ large enough.

Since the functions $r^\pm $ are H\"older continuous, one may choose $\zeta _1$ in such a way that if $ d(v'(w),v''(w)) \leq \zeta_1,$ then, for all $\l \in [0, \l_0]$,  $|\frac{d_{ss}(v, v')}{r_\l^-(v)} - \frac{d_{ss}(v_0, v'')}{r_\l^-(v_0 )}| \leq \zeta _0, $ where $\zeta _0 $ is given by Remark \ref{remark:lemma4.5}. Then, 
\[  \wt{\psi }_{\zeta _0}^- \left(\frac{d_{ss}(v_0 , v''(w))}{r_\l^-(v_0)}\right)  \leq a_-(v'(w)) \leq a_+(v'(w)) \leq  \wt{\psi }_{\zeta _0}^+ \left(\frac{d_{ss}(v_0 , v''(w))}{r_\l^+(v_0)}\right) .\]

In the same way, reasoning around $\gg_Rv_0$,  we have, if $R$ is large enough,
\[  \wt{\psi }_{\zeta _0}^- \left(\frac{d_{uu}(\gg_Rv_0 , \gg_Ru''(w))}{r_\l'^-(\gg_Rv_0)}\right)  \leq b_-(u'(\gg_Rw)) \leq b_+(u'(\gg_Rw)) \leq  \wt{\psi }_{\zeta _0}^+ \left(\frac{d_{uu}(\gg_Rv_0 , \gg_Ru''(w))}{r_\l'^+(\gg_Rv_0)}\right) .\]
Using (\ref{localproduct}), we obtain that the integrals $ \int_{(B(v_0))^\perp }  a_\pm (v'(w)) b_\pm (u'(\gg_Rw)) dm^\perp_\l(w)$ are, up to $(1+\e)^4$, given by $ \Om _\l \th^2_\l (v_0)  $ times
\[ \int _{W^{ss} (v_0 ) \x W^{uu}(v_0) } \wt{\psi }^\pm_{\zeta _0} \left(\frac{d_{ss}(v_0 , v'')}{r_\l(v_0)}\right)  \wt{\psi }_{\zeta _0}^\pm \left(\frac{d_{uu}(\gg_Rv_0 , \gg_Ru'')}{r_\l'(\gg_Rv_0)}\right)   \, d\mu_\l ^{uu} (u'') \x d\mu_\l^{ss}(v'').\]This is the integral of a product over  a product measure. We have, by our choice of $\zeta _0$
\[  \int _{W^{ss} (v_0 ) } \wt{\psi }^\pm_{\zeta _0} \left(\frac{d_{ss}(v_0 , v'')}{r_\l(v_0)}\right)\,  d\mu_\l^{ss}(v'') \; \sim ^{(1+ \e)^2} \g_0 .\]
Recall that, on $W^{uu} (v_0) $, $ \frac{d(\gg_R)_\ast \mu ^{uu} } { d\mu^{uu} } (u'')  \sim ^{(1+\e)^{4}}  \frac{d\gg_R \mu _\l^{uu} }{d  \mu _\l^{uu} }(\gg_Rv_0) = \frac{d\mu _\l^{uu} }{d \gg_{-R} \mu _\l^{uu} }(v_0), $ so that 
\begin{eqnarray*}  &  & \int _{W^{uu}(v_0) }  \wt{\psi }_{\zeta _0}^\pm  \left(\frac{d_{uu}(\gg_Rv_0 , \gg_Ru'')}{r_\l'(\gg_Rv_0)}\right)   \, d\mu_\l ^{uu} (u'')\\ &  \sim ^{(1+\e)^4}  & \frac{d\mu _\l^{uu} }{d \gg_{-R} \mu _\l^{uu} }(v_0)  \int _{W^{uu}(\gg_Rv_0) }  \wt{\psi }_{\zeta _0}^\pm \left(\frac{d_{uu}(\gg_Rv_0 , u'')}{r_\l'(\gg_Rv_0)}\right)   \, d\mu_\l ^{uu} (u'') \\ & \sim ^{(1+\e)^6}  &\frac{d\mu _\l^{uu} }{d \gg_{-R} \mu _\l^{uu} }(v_0) \g'_0.
\end{eqnarray*}
Altogether, we see that 
\begin{eqnarray*} &  &  \int _{B(v_0)} F^\pm_\l (w)  H^\pm_\l (\gg_R w ) \, dm_\l(w)   \\ &   \sim ^{(1+\e)^{30}} & \th _\l^{-2}(v_0) f(v_0) \th _\l^{-2}(\gg_R v_0) h(-\gg_Rv_0) \x \g \x \Om _\l \th _\l^2 (v_0) \x \g_0 \x \frac{d\mu _\l^{uu} }{d \gg_{-R} \mu _\l^{uu} }(v_0) \g'_0. \end{eqnarray*}
This proves equations \ref{F+H+} and \ref{F-H-} and achieves the proof of property (5).

{\it{Step 6. End of the proof of Proposition~\ref{LC}}}

By Properties (1),  (2) we can apply Proposition~\ref{4.1} and find $R_0, \d_0 $ independent of $\l, x, y $ such that for $R >R_0, \l \in [\l_0 - \d_0 , \l_0 ]$, 
\begin{eqnarray*}  \int F^-_\l H_- \circ \gg_R \, dm _\l &\sim ^{(1+\e) }&  \int F^-_\l\,  d\m_\l  \int H_- \,dm_\l \\  \int F_+  H_+ \circ \gg_R \, dm _\l  &\sim ^{(1+\e)} & \int F_+\,  d\m_\l  \int H_+\,dm_\l .\end{eqnarray*}
We get 
\begin{eqnarray*} &  &\Om_\l  \g_0 \g \g_0'  \Si (x,y,R,\d') \\ && \quad \sim^{ (1+\e)^{60} }\; \Om_\l^2 \d ' \g_0 \g \g_0'  \left(\int_{S_{\ppi x} M }  f(v)\, d\mu _\l^{uu} (v) \right) \left( \int _{S_{\ppi y} M } h(-u) \, d\mu _\l^{ss} (u)\right),\end{eqnarray*}
which is the statement of Proposition~\ref{LC} after dividing both  terms  by $\Om _\l \g_0 \g \g_0' .$

The condition on $\d'$ before step 1 depends on functions $f,h$ (see Remark~\ref{rmk:8}). The conditions on $R$ and $\d$ have been geometric in Steps 1 to 5 and depend only on $\e.$ 
Now  $R_0$ and $\d_0$  are given by  Proposition~\ref{4.1}  and depend   on $\e, ||F_\l^\pm||_\a, ||H_\l^\pm||_\a, \inf_\l \int F_\l^\pm \,dm_\l $ and  $ \inf_\l \int H_\l^\pm \,dm_\l$. Finally,   $||F_\l^\pm||_\a, ||H_\l^\pm||_\a, \inf_\l \int F_\l^\pm \,dm_\l $ and  $ \inf_\l \int H_\l^\pm \,dm_\l$ themselves depend only on $\e, ||f||_\a, ||h||_\a, \inf f$ and  $ \inf h.$
\end{proof}

\subsection{Convergence of measures} 
We state in this subsection several consequences and variants of Theorem~\ref{renewal2} which will be used in the next sections. Set $\Om := \Om_{\l_0}$ and $\Upsilon:=\Upsilon_{\l_0}$.

First, observe that the expression (\ref{eqn:5.4}) is continuous in $\l$ as $\l \to \l_0$ by Corollary~\ref{cor:3.10}. By choosing $\d_1 =\d_1(f,h,\e) $ such that for $\l \in [\l_0 - \d_1, \l_0]$
\begin{eqnarray*} 
&  & \Om    \int _{\pp \M }  f( \ppi \circ \s ^{-1}_x \xi) d\mu _x^{\l_0 }(\xi )\int_{M_0}\left( \int_{ \pp \M}  h (\Dppi \circ \s^{-1}_y \xi )   d\mu_y^{\l_0} (\xi) \right) \mathrm{dVol}(y)\\
& \sim ^{(1+\e')}& \Om _{\l}   \int _{\pp \M }  f( \ppi \circ \s ^{-1}_x \xi) d\mu _x^{\l }(\xi )\int_{M_0}\left( \int_{ \pp \M}  h (\Dppi \circ \s^{-1}_y \xi )   d\mu_y^{\l} (\xi) \right)\mathrm{dVol}(y),
\end{eqnarray*}
we obtain a corollary of Theorem~\ref{renewal2} by taking $\d (f,h ,\e') < \d_1(f,h,\e')$
:
\begin{cor}\label{renewal6} 
Given $\e'>0$ and positive H\"older continuous functions $f,h$ on $SM$, there is $R(f,h,\e')$ and $\d( f,h, \e')$ such that if $R >R( f,h, \e')$ and $\l_0 - \l < \d( f,h,\e' )$, for all $x \in \M$, 
 \begin{equation*}  \hspace{-1.2 in}  e^{-R P({\l} )}  \int _{S(x, R)}  f(\Dppi v_x^y) h( \Dppi v_y^x)   G_\l^2 (x, y) dy \quad \quad \sim^{(1+\e')^4} \end{equation*} 
 \begin{equation*} \Om   \int _{\pp \M }  f( \Dppi \circ \s ^{-1}_x \xi) d\mu _x^{\l_0 }(\xi )\int_{M_0}\left( \int_{ \pp \M}   h (\Dppi \circ \s^{-1}_y \xi )   d\mu_y^{\l_0} (\xi) \right)
  d\mbox{Vol}(y),
 \end{equation*}
 where $R(f,h,\e')$ and $\d( f, h,\e')$ depends  only on $\e', ||f||_\a,$ $\inf f $ and  $\inf h.$ 
 
 \end{cor}

\begin{cor}\label{renewal5} Fix $x \in \M$. Given $\e'>0$ and a positive H\"older continuous function $f$ on $S_x \M$, there is $R(f,\e')$ and $\d( f,\e')$ such that if $R >R( f,\e')$ and $\l_0 - \l < \d( f,\e' )$, 
 \begin{equation}   e^{-R P({\l} )}  \int _{S(x, R)}  f( v_x^y)    G_\l^2 (x, y) dy \; \sim^{(1+\e')^4} 
\; \Om \int _{\pp \M }  f( \s ^{-1}_x \xi) d\mu _x^{\l_0 }(\xi ),
 \end{equation}  
 where $R(f,\e')$ and $\d( f,\e')$ depends  only on $\e', ||f||_\a,$ and $\inf f.$
In particular, for $\l= \l_0$, 
$$ \lim_{R \to \infty} e^{-R P({\l_0} )}  \int _{S(x, R)}  f(v_x^y)    G_{\l_0}^2 (x, y) dy = 
\; \Om  \int _{\pp \M }  f(  \s ^{-1}_x \xi) d\mu _x^{\l_0 }(\xi ).$$
\end{cor}
\begin{proof}
Extend $f$ to a $\G$-invariant H\"older continuous function on $S\M$ and consider the function induced on $SM$. The statement follows by letting $h = 1$ in Corollary ~\ref{renewal6}. \end{proof}
Letting $f = 1$ in Corollary~\ref{renewal6}, we obtain the convergence of measures announced in the introduction.
\begin{cor}\label{renewal4} Fix $x \in \M$. As $R \to \infty $ and $\l \to \l_0$, the measures $m_{x,\l, R}$ defined {in the introduction $(*)$} converge to the measure $ \Om \mu^{\l_0}_x (\pp \M)  \ov m $ on $SM,$ where $\ov m $ is given by, for any continuous function $h$ on $C(SM) ,$
\[ \int_{SM} h \, d\ov m  \; = \; \int_{M_0}\left( \int_{ \pp \M} h (\Dppi \circ \s^{-1}_y \xi )  d\mu_y^{\l_0} (\xi) \right)
  d\mbox{Vol}(y) .\] \end{cor}

  In  the  proof of Theorem~\ref{renewal2}, the choice of $\d(f,h, \e') $ is only made in Step 6, when we want to use the uniform mixing of Proposition~\ref{4.1}. For a fixed $\l$, we can use instead the regular mixing of $m_\l $ for H\"older continuous functions and obtain a proof of Proposition~\ref{renewal}. We  can write, taking $f=h=1,$
\begin{cor}\label{formula} In Proposition~\ref{renewal}, the limit $D(x, \l)$ is given by \[ D(x, \l ) \; = \; \Om _\l \, \mu _x^\l (\pp \M) \, \int _{M_0} \int_{\pp \M}  d\mu _y^\l (\xi) d\mbox{Vol} (y) = \Om _\l \, \mu _x^\l (\pp \M). \]
 \end{cor} 

As a Corollary of the proof of Theorem~\ref{renewal2} and Corollary~\ref{renewal5}, we state   a generalization which will be needed in Section~\ref{sec:7}.
\begin{prop}\label{renewal3}  Given $\e>0$ and positive H\"older continuous functions $f,u$ on $SM$, there is $R(f,u,\e)$ and $\d( f,u, \e)$ such that if $R >R( f,u, \e)$ and $\l_0 - \l < \d( f,u,\e )$,
 \begin{eqnarray*}  && e^{-R P(\l)}   \int _{S(x, R)} f(\Dppi v_x^y) \left( \frac{1}{R} \int _0^R u(\gg_s  \Dppi v_x^y) ds \right)  G_\l^2 (x, y) dy\; \;\;\;\;\; \\ 
 \quad \quad  && \quad \quad \sim^{1+\e} \; \Om  \int _{\pp \M } f( \ppi \s ^{-1}_x \xi) d\mu _x^{\l_0 } (\xi )  \int _{SM} u dm_{\l_0}, \end{eqnarray*}
where $R(f,u,\e)$ and $\d( f,u,\e)$ depends  only on $\e, ||f||_\a, ||u||_\a, \inf f$ and  $\inf u.$
 \end{prop}
This is the analog of Corollary \ref{renewal5}, with the  extra term $\frac{1}{R} \int _0^R u(\gg_s  \Dppi v_x^y) ds $, which should yield the term $  \int _{SM} u dm_{\l_0}$ in the limit. We introduce a H\"older continuous function $h$ and extend the proof of Theorem~\ref{renewal2} with an extra $u$-term.
So,  we replace $ \Si (x,y,R, \d') $ by
 \begin{eqnarray*} &    & \Si' (x,y,R, \d') := \\  & \;  & \underset{ \{ (v,\rr) : v \in S_{\ppi x}    M \cap \gg_{-\rr} S_{\ppi y } M\}}  {\sum} \chi _{\d'} (R-\rr) f (v )\left( \frac{1}{\rr} \int _0^\rr u(\gg_s  v) ds \right) (\th _\l^{-2}h)(-\gg_\rr v)\frac{d \mu_\l ^{uu } }{d\gg_{-\rr}  \mu_\l ^{uu}} ( v)\end{eqnarray*}
 and we similarly choose $\d' _1>0 $ such that, if $R $ is large enough and $\d' <2\d'_1$, then, for all $x \in \M$ and $\l \in [\l_0 - \d(f,h,\e'),\l_0],$
\begin{eqnarray*} &  &   e^{-R P({\l} )}  \int _{S(x, R)}  f(\Dppi v_x^y)\left( \frac{1}{R} \int _0^R u(\gg_s  v_x^y) ds \right)  h( \Dppi v_y^x)   G_\l^2 (x, y) dy  \\ 
&& \qquad \sim ^{(1+\e')}  \frac{1}{\d'} \int _{M_0}  \Si '(x,y,R,\d') \,  d\Vol (y).\end{eqnarray*}

 We are reduced to show the analog of Proposition~\ref{LC}, namely 
\begin{lem}\label{LlT6.4}  Given $\e' >0$ and positive H\"older continuous functions $f ,u$,  there exist $R_1= R_1 ( f,u,\e') $, $\d_1( f, u, \e')>0$ and  $\d', 0 < \d' <2\d_1' ,$  such that for $R\geq R_1$, all $x,y\in M$ and all $\l \in [\l_0-\d_1, \l_0]$, 
\begin{eqnarray*}
&   &   \Si' (x,y,R,\d') \\ && \quad \sim^{1+\e'} \Om _\l \d'  \left(\int_{S_{\ppi x} M }  f(v)\, d\mu _\l^{uu} (v) \right) \left( \int _{S_{\ppi y} M } h(-w) \, d\mu _\l^{ss} (w)\right) \left(\int u \, dm_\l \right) . \end{eqnarray*}
Moreover, $R_1(f,u,\e')$ and $\d_1( f,u, \e')$ depends  only on $\e', ||f||_\a, ||h||_\a, ||u||_\a, \inf f$, $\inf u$ and  $\inf h.$
\end{lem}
 
\begin{proof} We choose the same $\e $ such that $(1+ \e)^{61} \leq 1+\e' .$ We choose $\d' _1 < \d' $ small enough  that, for all $t>0$, if $v,w \in SM $ are such that $d(v,w) <\d' _1$ and $ d(\gg_t v, \gg_t w) < \d'_1$, then
\[ \int_0^t u(\gg_s v ) ds \; \sim^{1+\e} \; \int_0^t u(\gg_s w) ds .\] 
This is possible because $u$ is H\"older continuous, positive, and the two geodesics $\gg_s v, \gg_s w$ satisfy \[ d_{SM} (\gg_s v, \gg_s w )\;  \leq \;  C\d_1' \max \{ e^{-a_4 s}, e^{a_4 (s-t)} \} ,\] 
where $C$ is a positive geometric constant. 
We then construct $F^\pm_\l, H^\pm_\l $ in the same way, with this new $\d'_1$ (and accordingly possibly new $\g_0, \g_0', \g$). Properties (1) to (4) still hold. In the equations \ref{F+H+} and \ref{F-H-}, we consider the integrals
\[ \int _{B(v_0)} F^\pm_\l(w) \left( \frac{1}{R} \int _0^R u(\gg_s w) ds \right) H^\pm_\l (w) \, dm_\l (w). \] we loose one more $\sim^{(1+\e)} $ factor when we replace $\left( \frac{1}{R} \int _0^R u(\gg_s w) ds \right)$ by $\left( \frac{1}{R} \int _0^R u(\gg_s v_0) ds \right)$. 

The new Property (5) reads as:
there is $R(\e)$ such that for $R \geq R(\e)$,
\[ (1+\e)^{-31} \int F^-_\l \left( \frac{1}{R} \int _0^R u\circ \gg_s ds \right)H^-_\l\circ \gg_R \,dm_\l \;\leq \; \Om _\l \g \g_0 \g_0'  \Si '(x,y,R,\d'(1+\e)),\]
\[ \Om _\l \g \g_0 \g_0'  \Si '(x,y,R,\d') \;\leq \; (1+ \e)^{31} \int F^+_\l \left( \frac{1}{R} \int _0^R u\circ \gg_s ds \right)H^+_\l\circ \gg_R\, dm_\l. \]
We conclude as above, using  Proposition~\ref{4.2} instead of Proposition~\ref{4.1}. \end{proof}

\section{Topological pressure at $\l_0$}\label{sec:5}

In this section, we show that $P(\l _0 ) = 0 $ and show direct consequences. 
\subsection{Vanishing of $P(\l _0 ) $} We already know that $P(\l _0 ) \leq 0$ {by Corollary~\ref{Pnonpositive}}. We show below in Proposition~\ref{branch} that if $P(\l _0 ) <0$ and thus  $\int _{S(x, R)} G_{\l_0}^{2}  (x, y) dy $ decays exponentially with $R$ (by Theorem~\ref{renewal2}), then $G_{\l_0 + \e} (x, y)  $ is finite, contradicting the definition of $\l_0$.

\begin{prop}\label{branch} 
$P(\l _0 ) = 0 .$
\end{prop}

\begin{proof}
Assume that $P(\l _0 ) < 0 $. We claim that for all $x \not = x' $, there exists $\e >0$ such that the function $ \l \mapsto G_\l (x,x') $ admits a real analytic extension on an $\e $-neighborhood of $\l _0$. In particular, for $\l_0 < \l < \l_0 + \e $, the extension $G_\l (x, x') $ satisfies $G_\l (x,x' ) = \int _0^\infty  e^{\l t} \Pp(t,x,x') dt $, a contradiction with the definition of $\l_0 $.

Let us now prove our claim. Fix $x \neq x' \in \widetilde{M}$.
By Proposition~\ref{derivative},
\[ \frac {\pp ^k }{\pp \l^k} G_\l (x,x') \; = \; k! \int _{\M ^k} G_\l (x, x_1) G_\l (x_1, x_2) \cdots G_\l (x_k, x' ) \, d{\textrm{Vol}}^k (x_1, x_2, \cdots, x_k ) .\]
The claim follows with $\e = 1/\rho$, if we show that there are positive numbers $\d, C$ and $\rho $ such that: 
\begin{equation}\label{GL(64)} 
F_k  \; := \; \int _{\M ^k} G_{\l_0} (x, x_1) G_{\l_0} (x_1, x_2) \cdots G_{\l_0} (x_k, x')e^{\d d(x, x_k)} \, d{\textrm{Vol}}^k (x_1, x_2, \cdots, x_k ) \; \leq \; C \rho ^k .
\end{equation}

Since $P(\l _0 ) <0 $, by {Theorem~\ref{renewal2}}, there is  $C, \d >0 $ such that, for all $x  \in \M $, all $R >1$, 
\[ \int _{S(x,R)} G_{\l _0}^2 (x,z) \, dz\leq C e^{-\d R}  {\textrm { and thus }} \int _{\{ y \in \M; d(x,y) \geq 2\}} G_{\l_0}^2 (x,y) \, d{\textrm{Vol}}(y) < +\infty .\]
By possibly choosing a smaller $\d >0 $, we have
\begin{equation}\label{GL(65bis)}  \int _{\{ y \in \M; d(x,y) \geq 2\}} G_{\l_0}^2 (x,y)e^{\d d(x,y)}  \, d{\textrm{Vol}}(y) \;\leq \; B\end{equation}
for some constant $B$.   
For this choice of $\d $, we prove (\ref{GL(64)}) by induction on $k$. For $k=0 $, (\ref{GL(64)}) is trivial for a suitable choice of $C$. We are going to show that $ F_{k+1} / F_k $ is bounded independently of $k$ (compare \cite{GL} Proposition 4.7). We write: 
\[ F_{k+1} = \int_{\M} \int _{\M ^{k}} G_{\l_0} (x, x_1) \cdots G_{\l_0} (x_k, z )G_{\l _0} (z,x') e^{\d d(x, z)} \, d{\textrm{Vol}}^k (x_1, \cdots, x_k ) d{\textrm{Vol}}(z).\]
Relation (\ref{GL(64)}) follows from Lemma~\ref{lem:4.1} for $x=x' $ and $ y=x_k$ with $\rho := \rho ' e^{2 \d d(x,x')}.$
Indeed, this yields
\begin{eqnarray*}
\int _{\M} G_{\l_0} (x_k, z) G_{\l_0}(z,x') e^{\d d(x,z)}d\Vol(z) \leq &e^{\d d(x,x')} \int _{\M} G_{\l_0} (x', z) G_{\l_0}(z,x_k) e^{\d d(x',z)}d\Vol(z)\\
\leq  &e^{\d d(x,x')} \rho '  G_{\l_0} (x', x_k)  e^{\d d(x',x_k)}  \; {\textrm {by Lemma \ref{lem:4.1}}} \\
\leq  &e^{2\d d(x,x')} \rho '  G_{\l_0} (x_k, x')  e^{\d d(x,x_k)} .
\end{eqnarray*}
\end{proof}

\begin{lem}\label{lem:4.1} There is $\rho >0 $ such that, for all $x,y \in \M ,$ \[  \int _{\M} G_{\l_0} (x, z )G_{\l _0} (z,y) e^{\d d(x, z)} \,d{\textrm{Vol}}(z) \; \leq \; \rho' \;  G_{\l_0} (x, y )e^{\d d(x, y)} .\] \end{lem}
\begin{proof} 
{ 
Assume first that $ d(x,y) \leq 2\ab$, for some $\ab>R_0$ to be fixed later. By Corollary~\ref{convolution}, if $d(x,y) \leq 2\ab$ then \[ \int _{B(x,4\ab)} G_{\l_0} (x, z )G_{\l _0} (z,y) e^{\d d(x, z)} \,d{\textrm{Vol}}(z) \; \leq \; C'_0 \;  G_{\l_0} (x, y ) \; \leq \; C'_0 \;  G_{\l_0} (x, y )e^{\d d(x, y)}\] for some $C'_0 =C'_0(\ab)$. Moreover,  $G_{\l_0}(x,y)$ is bounded from below and therefore it suffices to show that 
\[  \int _{\M\setminus B(x,4\ab)} G_{\l_0} (x, z )G_{\l _0} (z,y) e^{\d d(x, z)} \,d{\textrm{Vol}}(z) \; \leq \;C''_0 \]
for some $C''_0$. On the set  $ z \in \M, d(z,y )\geq 4\ab$, we can write $G_{\l_0} (x, z )G_{\l _0} (z,y) \leq C_0^{2\ab} (G_{\l_0} (x, z ))^2$  by Proposition~\ref{Harnack}. 
By (\ref{GL(65bis)}), this part of the integral has a contribution at most {$C_0^{2\ab} B.$} 
Thus, there is a constant $\rho _0  $ such that, if $d(x,y) \leq 2\ab,$ then
\[  \int _{\M} G_{\l_0} (x, z )G_{\l _0} (z,y) e^{\d d(x, z)} \,d{\textrm{Vol}}(z) \; \leq \; \rho _0\;  G_{\l_0} (x, y )e^{\d d(x, y)} .\] }
Consider now the case $d(x,y) \geq 2\ab$ and
 let $L$ be the geodesic segment going from $y $ to $x$. We write $\M = \M_1 \cup \M_2 \cup \M_3 \cup \M_4\cup \M_5 \cup \M_6 $ and consider the six integrals $\int _{\M_i} G_{\l_0} (x, z )G_{\l _0} (z,y) e^{\d d(x, z)} \,d{\textrm{Vol}}(z)$. Let $pr(z)$ be the point of $L$ realizing $d(z,pr(z)) = d(z, L).$ We define, for  $ \ab' >\ab$ to be chosen later,
 \begin{eqnarray*} 
 \M_1 &:= & \{ z \in \M , d(pr(z), x) \geq \ab,  d(pr(z), y) \geq \ab, d(z,L )\geq \ab \}\\
   \M_2 &:= & \{ z \in \M , d(pr(z), y) \leq \ab, d(z, y )\geq \ab' \}\\
  \M_3 &:= & \{ z \in \M , d(pr(z), y) \leq \ab,  d(z, y) \leq \ab' \}\\
  \M_4 &:= & \{ z \in \M ,  d(pr(z), x) \leq \ab, d(z,x )\geq \ab' \}\\
   \M_5 &:= & \{ z \in \M , d(pr(z), x) \leq \ab, d(z, x) \leq \ab'   \}\\
   \M_6 &:= & \{ z \in \M , d(pr(z), x) \geq \ab,  d(pr(z), y) \geq \ab, d(z,L )\leq \ab \}.
\end{eqnarray*}
\begin{figure}[!htb]
\centering
\includegraphics[scale=0.4]{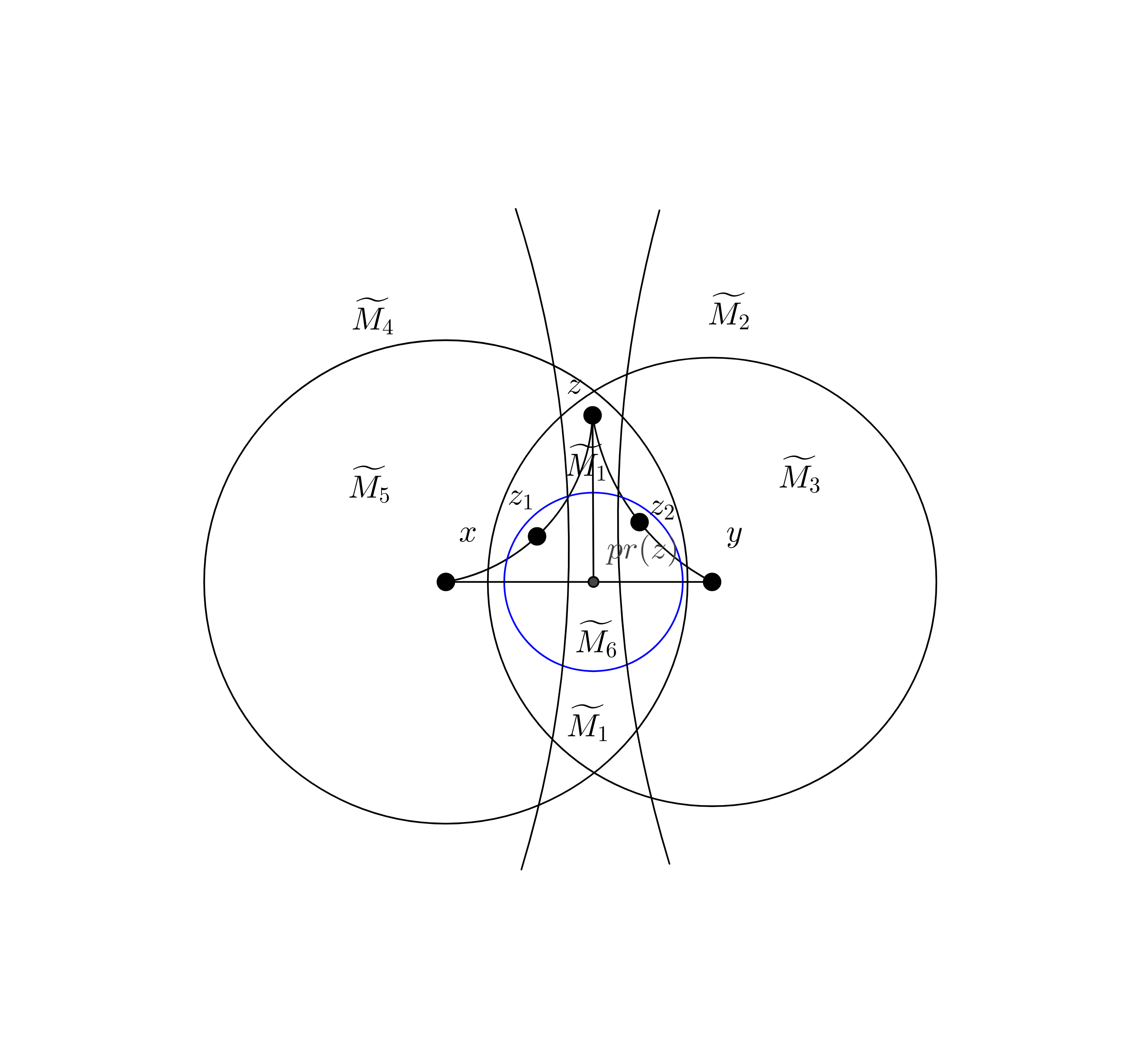}
\caption{$\widetilde M_i$}
\label{fig6}
\end{figure}

On $\M_1$, consider the thin geodesic right triangles $ (y,pr(z),z) $ and $ (x ,pr(z),z)$. The distances $d(pr(z), [z,y]), d(pr(z), [z,x])$ from $pr(z)$ to both geodesics $[z,y]$ and $[z,x] $ are bounded above by a hyperbolicity constant $a_5$.
Let $z_1, z_2$ be the points realizing these distances : $d(pr(z), [z,x])=d(pr(z), z_1), d(pr(z), [z,y])=d(pr(z), z_2).$

We choose $\ab \geq R_0$ such that $d(z, z_1), d(z, z_2), d(x,z_1)$ and $d(y,z_2)$ are equal or greater than $R_0$, where $R_0$ is the constant in Ancona-Gou\"ezel inequality (Theorem~\ref{Alano}). 
Using  Harnack inequality and the hard side of the Ancona-Gou\"ezel  inequality, we get 
\begin{eqnarray*} 
G_{\l _0} (x, z ) \; &\leq &C_4C_0 ^{2 a_5} G_{\l _0} (x,pr(z))G_{\l _0} (pr(z),z ) \\
G_{\l _0} (z,y) \; &\leq &C_4C_0^{2 a_5}  G_{\l _0} (z,pr(z))G_{\l _0} (pr(z),y).
 \end{eqnarray*} 
Therefore, we have
\begin{eqnarray*} & & \int _{\M_1 } G_{\l_0} (x, z )G_{\l _0} (z,y) e^{\d d(x, z)} \,d{\textrm{Vol}}(z) \; \\
&\ls &  \int _{\M_1}  G_{\l_0} (x , pr(z) )G_{\l _0} (pr(z), y ) e^{\d d(x, pr(z))}  G_{\l_0}^2 (z,pr(z))e^{\d d(z,pr(z))} \,d{\textrm{Vol}}(z)\\
&\ls &  G_{\l _0} (x ,y)  \int _{\M_1}  e^{\d d(x, pr(z))}  G_{\l_0}^2 (z,pr(z))e^{\d d(z,pr(z))} \,d{\textrm{Vol}}(z), \end{eqnarray*}
by the easy side of Ancona-Gou\"ezel inequality.

We use the function $\psi : \R \to \R, \psi(t) = \max (1-|t|, 0)$. 
Since $\int_{\ab-1}^{d(x, y)-\ab+1} \psi(t-s) dt = 1$ for all $s$ between $\ab$ and $d(x,y)-\ab$, we obtain 
\begin{eqnarray*}& &  \int _{\M_1}  e^{\d d(x, pr(z))}  G_{\l_0}^2 (z,pr(z))e^{\d d(z,pr(z))} \,d{\textrm{Vol}}(z) \\
 & = & \int _{\M_1}  \int_{\ab-1}^{d(x, y)-\ab+1} \psi(t - d(x, pr(z)))    e^{\d d(x, pr(z))} G_{\l_0}^2 (z,pr(z))e^{\d d(z,pr(z))} \, dt \; d{\textrm{Vol}}(z)
\end{eqnarray*}
Let $w_s $ be the point on the geodesic $[x,y]$ of distance $s$ from $x$,  for $\ab-1 \leq s \leq d(x, y)-\ab+1$. We disintegrate the integral with respect to $d\mathrm{Vol}(z)$ as $d\mu^{\mathrm{Vol}}_s(.)ds$, where $d\mu^{\mathrm{Vol}}_s(.)$  is a measure  on the points  $z$ with $d(x, pr(z))=s$. By Fubini theorem, the right hand side of the previous equality is equal to
\begin{eqnarray*}
&& \int_{\ab-1}^{d(x, y)-\ab+1} \int_R^{d(x,y)-R}  \int _{\{ z\in \M_1: d(x, pr(z))=s\}} \psi(t - s)    e^{\d s} G_{\l_0}^2 (z,w_s)e^{\d d(z,w_s)} \,d\mu^{\mathrm{Vol}}_s(z) \, ds \, dt\\
&& \ls   \int_{\ab-1}^{d(x, y)-\ab+1} \int_{R}^{d(x,y)-R}  \int _{\{ z \in \M_1: d(x, pr(z))=s\}}    e^{\d t} G_{\l_0}^2 (z,w_t)e^{\d d(z,w_t)} \,d\mu^{\mathrm{Vol}}_s(z)ds \, dt\\
&& \leq   \int_{\ab-1}^{d(x, y)-\ab+1}  e^{\d t}  \int _{\{ z \in \M; d(z, L) \geq 2 \}}   G_{\l_0}^2 (z,w_t)e^{\d d(z,w_t)} \,d\mathrm{Vol}(z) \, dt\\
&& \ls  \int_{\ab-1}^{d(x, y)-\ab+1}    e^{\d t} B \, dt \; \ls \;  e^{\d d(x, y)} ,\end{eqnarray*}
where the first inequality uses Harnack inequality for replacing $w_s$ by $w_t$ as $d(w_s,w_t)<1$, and the third inequality uses (\ref{GL(65bis)}). We conclude that there is a $C_1'$ such that \[ \int _{\M_1 } G_{\l_0} (x, z )G_{\l _0} (z,y) e^{\d d(x, z)} \,d{\textrm{Vol}}(z) \;\leq C_1' G_{\l_0} (x, y )e^{\d d(x, y)} .\] 

It remains to prove that the integrals on $\M_i$ for $i=2, \dots, 6$ have similar bounds.
Choose $\ab' >> \ab$ large enough so that there exists $a_6 =a_6(R, R')$ with the following properties:
\begin{enumerate}
\item for  $ z \in \Ma,$  there is a point $z_1 \in [z,x]$ with $d(z_1,x)>R_0$, $d(z_1,z)>R_0$ and $d(z_1,y) < a_6(R,R')$,
\item for  $ z \in \Mc,$  there is a point $z_1 \in [z,y]$ with $d(z_1,y)>R_0$, $d(z_1,z)>R_0$ and $d(z_1,x) < a_6(R,R').$
\end{enumerate}
 The choice of $R'$ can be made independent of the position of $x,y$ as soon as  $d(x,y) \geq 2R$.
Apply Harnack inequality (Proposition~\ref{Harnack}) and Ancona-Gou\"ezel inequality (Theorem~\ref{Alano}) to  get, if $z \in \M_2,$
\[ G_{\l_0} (x, z )G_{\l _0} (z,y) e^{\d d(x, z)}  \; \ls G_{\l _0} (x,y) e^{\d d(x, y)} (G_{\l_0} (y, z ))^2  e^{\d d(y, z)} .\]
By (\ref{GL(65bis)}), we obtain a constant $C'_2 $ such that \[\int _{\M_2 } G_{\l_0} (x, z )G_{\l _0} (z,y) e^{\d d(x, z)} \,d{\textrm{Vol}}(z) \;\leq C'_2  G_{\l_0} (x, y )e^{\d d(x, y)} .\]
The proof is similar for $\M_4$ and we obtain a constant $C'_4$.

For $z \in \M_3,$ we have, by Proposition~\ref{Harnack}, \[G_{\l_0} (x, z )G_{\l _0} (z,y) e^{\d d(x, z)}\ls G_{\l_0} (x, y )e^{\d d(x, y)}G_{\l_0} (y, z ).\]
Using (\ref{finiteintegral}), we obtain a constant $C'_3 $ such that \[\int _{\M_3 } G_{\l_0} (x, z )G_{\l _0} (z,y) e^{\d d(x, z)} \,d{\textrm{Vol}}(z) \;\leq C'_3  G_{\l_0} (x, y )e^{\d d(x, y)} .\]
The proof is similar for $\M_5$ and we obtain a constant $C'_5$.

For $z \in \M_6$,  $pr(z) $ is at distance at least $R_0$ from $x$ and from $y$. We then have $G_{\l_0} (x, z )G_{\l _0} (z,y) \ls G_{\l_0} (x, y )$ by Harnack inequality and the easy side of the Ancona inequality (\ref{Ancona}).The integral $\int _{ \M_6} e^{\d d(x, z)}  d \textrm{Vol}(z) $ can be estimated  as \[ Ce^{2\d R'} \int _{R_0}^{d(x,y )-R_0} e^{\d t} \, dt  \ls e^{\d d(x,y)},\] 
as for $\M_1.$ Altogether, we obtain a constant $C'_6$ such that $\int _{\M_6}  \;\leq C'_6  G_{\l_0} (x, y )e^{\d d(x, y)} .$
The constant in Lemma ~\ref{lem:4.1}  is {$\rho ' = \max\{ \rho_0,  \sum_{i=1}^6 C'_i \}$. } \end{proof}
\subsection{Applications of Proposition~\ref{branch} }
\subsubsection{Behavior of $\frac{\pp }{\pp \l } G_\l (x,y) $ at $\l_0$}
\begin{prop}\label{limit2}{ For $x \not= y \in \M$}, $$ \lim\limits_{\l \to \l_0} -P({\l} ) \frac{\pp }{\pp \l } G_\l (x,y) \; = \;  \Om c(x,y), $$
where  
$c(x,y) $ is given   by \begin{equation}\label{Harish-Chandra} c(x,y) =  \int k_{\l_0} (x,y, \xi )  d\mu^{\l_0}_x (\xi).\end{equation} 
 Moreover, for any compact neighborhood $K$ of $x$ in $\M$, there is $\l'< \l_0$ such that $y \mapsto \sup_{\l, \l' \leq \l \leq \l_0 } \left(-P({\l} ) \frac{\pp }{\pp \l } G_\l (x,y)\right) $ is integrable on $K$.
\end{prop} 
\begin{proof} We have:
\begin{eqnarray*}
 -P({\l} ) \frac{\pp }{\pp \l } G_\l (x,y) &= & -P({\l} ) \int_{\M} G_\l (x,z) G_\l (y,z)  d\mbox{Vol} (z) \\ 
 &=& -P({\l} ) \int _0^\infty e^{P({\l} ) R} \left( \int_{S(x,R)}e^{-P({\l} )R} k_\l (x,y,z) G_\l^2 (x, z)   dz\right) dR.
\end{eqnarray*}
 Let $A$ be the diameter of $K$. We are going to cut the integral $\int _0^\infty  = \int_0^{A+1} + \int _{A+1}^{R'} +\int _{R'}^\infty ,$ for some $R'$ chosen later, and show the (dominated on $K$) convergence of each integral separately.

By Corollary~\ref{convolution}, for $y \in K,$ $$ \int_{B(x,A+1)} G_\l (x,z) G_\l (y,z)  d\mbox{Vol} (z) \leq \int_{B(x,A+1)} G_{\l_0} (x,z) G_{\l_0} (y,z)  d\mbox{Vol} (z) \leq C G_{\l_0} (x,y) .$$  The function $y \mapsto G_{\l_0} (x,y)$ is integrable on $B(x,A+1) $ by  (\ref{eqn2.3}). Since $P(\l) $ goes to 0, this part  converges to 0. The convergence is dominated since $\sup_{\l, 0\leq \l \leq  \l_0 } | P(\l) | < \infty .$ 

In the same way, using Propositions~\ref{Harnack} and \ref{Lemma2.5}, we can write, for all $y, 0 <d(x,y )\leq A,$
\begin{eqnarray*}   \int _{A+1}^{R'} \left( \int_{S(x,R)} G_\l (x,z) G_\l (y,z) \, dz \right)  dR &\leq & C_0^A  \int _{A+1}^{R'}  \left( \int_{S(x,R)} G_{\l_0} (y,z) G_{\l_0} (y,z)  \, dz \right) dR  \\ &\leq  & C_0^A C (R' -A).\end{eqnarray*}
Thus $(-P(\l)) \int _{A+1}^{R'}  \left( \int_{S(x,R)} G_\l (x,z) G_\l (y,z)  \right) dR \to 0 $ as $\l \to \l_0$.

 On the other hand, as $R\to \infty $, the function $k_\l(x,y,z)$ is close  to $k_\l (x, y, (v_x^z)^+)$ uniformly in $\l $ (Theorem \ref{thm:1.4}),  thus it can be considered as a H\"older continuous function on $S_x M$. Observe that the  constant $C(\max \{d(x,y),1\} ) $ in Proposition~\ref{demiMartin} is uniform for $y \in K$ so that the H\"older norm of $k_\l (x, y, (v_x^z)^+)$ is uniformly bounded for $\l \in [0, \l_0]$  and $y \in K$.  \footnote{Here we use the fact that the interval $[\l_0 - \d, \l_0]$ in the conclusions of Section 4 depend only on $||f||_\a, \inf f,$ etc.} By Corollary~\ref{renewal5}, given $\e>0$, for $R'$ large enough and $\l$ close enough to $\l_0$, uniformly for $y \in K$,
 \begin{equation}\label{aformula}  \int_{S(x,R)}e^{-P({\l} )R} k_\l(x,y,z) G_\l^2 (x, z)   dz\sim^{1+\e}\Om \int k_{\l_0} (x, y, \xi )  d\mu^{\l_0}_x (\xi) = \Om \, c(x,y). \end{equation} As $\l \to \l_0$, $P(\l ) \to P({\l_0} ) = 0 $, it follows that
 $$ \underset{ \l \to \l_0}{\lim} -P(\l ) \frac{\pp }{\pp \l } G_\l (x,y) = \underset{\l \to \l_0, R \to \infty}{\lim} \int_{S(x,R)} e^{-P({\l} )R} k_\l(x,y,z)  G_\l^2 (x, z)   dz = \Om\, c(x,y). $$
\end{proof}

In particular, since $\Om$ and $c(x,y) $ are positive numbers, $ \frac{\pp }{\pp \l } G_\l (x,y) $ goes to infinity as $\l \to \l_0.$

\begin{remark}\label{rmk:5.4} It follows from the proof above that $$\lim_{\l \to \l_0} -P({\l} ) \int_{\M \setminus B(x,1)} G_\l (x,z) G_\l (x,z)  d\mbox{Vol} (z) = \Om \int k_{\l_0} (x, x, \xi )  d\mu^{\l_0}_x (\xi)  = \Om \mu^{\l_0}_x(\pp \M).$$
\end{remark}

\subsubsection{ Global limits.} Using corollary~\ref{renewal5} ($f =1$ for the first limit and $f=k_{\l_0}(x,y,z)$ for the second limit), we obtain
\begin{prop}\label{prop:5.5} For $x,y \in \M$, as $R \to \infty $, we have, with the above notations
\[ \int_{S(x,R)} G_{\l_0}^2 (x,z) \, dz \; \to \; \B \mu_x^{\l_0} (\pp \M), \quad  \int_{S(x,R)} G_{\l_0} (x,z) G_{\l_0} (y,z) \, dz \; \to \; \B c(x,y),\] and, for  any $\a$-Holder continuous function $h$ on $S_x \M$, there exists $R(h,\e)$ and $\d=\d(R,\e)$ such that for $R>R(h, \e)$ and $\l \in [ \l_0-\d, \l_0]$
$$ e^{-P(\lambda)R}\int_{S(x,R)} h(v_x^z)G_{\l}^2 (x,z) \, dz \; \to \; \B \int_{\pp \M} h(\Dppi \s _x^{-1}(\xi)) \mu_x^{\l_0}(\xi).$$
$$\int_{S(x,R)} h(v_x^z)G_{\l_0}^2 (x,z) \, dz \; \to \; \B \int_{\pp \M} h(\Dppi \s _x^{-1}(\xi)) \mu_x^{\l_0}(\xi)$$

\end{prop}
 
\begin{remark} 
Observe that the last limit can serve as another definition of the $\mu^{\l_0}_x$. Observe also that the bounds $R(h, \e), \d(h,\e)$ depend on the H\"older norm of $h$ and not anymore on $\inf h$ since the convergence holds for constant functions.\\
\end{remark}
\subsubsection{Proof of Theorem \ref{PattersonSullivan} and Corollary~\ref{Mohsen}}\label{sec:5.2.3}

\noindent {\it Proof of Theorem \ref{PattersonSullivan}.}
Since the function $\vf_{\l_0}$ is H\"older continuous (Corollary~\ref{derivativeHolder}), Proposition~\ref{measureNonu} applies to $\vf_{\l_0} $ as well. 
Theorem \ref{PattersonSullivan} follows since $P(\l_0 ) = 0 $. 
\begin{proof}[Proof of Corollary~\ref{Mohsen}]
We have to show that the  energy $\mathcal E(\mu ^{\l_0})$ of the family $\mu _x^{\l_0}$ is $4\l_0.$ By the relation (\ref{rayleigh}), 
\[ \mathcal E(\mu ^{\l_0}) = 4 \int _{M_0} \left(\int_{\pp \M} \| \nabla_x k_{\l_0}(x_0,x,\xi) \|^2 d\mu _{x_0}^{\l_0}(\xi)  \right) d\mbox{Vol} (x).\]
By using a partition of unity,  any $C^1$ vector field $Z$ on $M_0$ can be decomposed as a sum of $C^1$ vector fields with compact support inside a fundamental domain and thus $\int _{M_0} {\textrm {Div}}Z (x) d\Vol (x)  =0 .$ In particular,
\begin{eqnarray*}  0 &=& \int _{M_0} {\textrm {Div}}(x)  \nabla_x k_{\l_0}^2(x_0,x,\xi)d\Vol (x) = - \int _{M_0} \D_x k_{\l_0}^2(x_0,x,\xi)d\Vol (x)\\ &=& -2\l_0  \int _{M_0}\frac{d\mu _x^{\l_0} }{d\mu _{x_0}^{\l_0} }(\xi ) \,  d\mbox{Vol} (x) + 2 \int _{M_0}  \| \nabla_x k_{\l_0}(x_0,x,\xi) \|^2  d\mbox{Vol} (x).\end{eqnarray*}
It follows that \[ \int _{M_0} \left(\int_{\pp \M} \| \nabla_x k_{\l_0}(x_0,x,\xi) \|^2 d\mu _{x_0}^{\l_0}(\xi)  \right) d\mbox{Vol} (x) = \l_0 \int _{M_0} \left( \int _{\pp\M} d\mu _{x}^{\l_0}(\xi)  \right) d\mbox{Vol} (x)= \l_0.\]
\end{proof}

\section{Proof of Theorem~\ref{locallimit}} \label{sec:proof}

\subsection{Derivative of the Green function} \label{sec:7}

In this subsection, we establish
\begin{theo}\label{Theorem3.1}With the above notations, 
 for $x \not = y \in \M$, as $\l \to \l_0 $,
\[ \frac{\pp }{\pp \l} G_\l (x,y) \; \sim \; \frac{\sqrt{\Upsilon}}{2\sqrt {\l_0 - \l}}c(x,y)  .\]
where $c(x,y)$ is given by (\ref{Harish-Chandra}) and $\Upsilon=\Upsilon_{\l_0}$, given by (\ref{Upsilon}).\end{theo}
Theorem~\ref{Theorem3.1} follows from the following Proposition.
\begin{prop}\label{Proposition3.17} 
 For all $x,y \in \M$,
\begin{equation}\label{3.10bis} 
\lim_{\l \to \l_0} -P^3(\l)  \int _{\M} \int _{\M} G_\l (x, z) G_\l (z,w) G_\l (w, y) d{\Vol} (w)  d\mbox{Vol} (z) \; = \; \frac{\B^3 }{\Upsilon} c(x,y). \end{equation}
In particular, for $x \not = y \in \M$,  
\begin{equation}\label{3.10} 
\lim_{\l \to \l_0}  -P^3(\l) \frac{\pp ^2}{\pp \l ^2} G_\l (x,y) = 2\frac{\B^3 }{\Upsilon} c(x,y). \end{equation}
Moreover, for any compact neighborhood $K$ of $x$,  there is $\l'< \l_0$ such that $$y \mapsto \sup_{\l, \l' \leq \l \leq \l_0 } \left(-P^3(\l)  \int _{\M} \int _{\M} G_\l (x, z) G_\l (z,w) G_\l (w, y) d\mbox{Vol} (w)  d\mbox{Vol} (z) \right) $$ is integrable on $K$.
 \end{prop}
 We will estimate the integral \eqref{3.10bis} in two regions, $B(x,2)$ and the rest.
\begin{lem} There is a constant $C$ such that for all $\l, 0 \leq \l < \l_0,$ 
\[ \int _{B(x,2)}G_\l (x, z) \left(\int _{\M} G_\l (z,w) G_\l (w, y) d{\Vol} (w) \right) d{\Vol} (z) \leq C  \frac{\pp }{\pp \l} G_\l (x,y).\] \end{lem} 
\begin{proof}By Proposition~\ref{prop:derivative}, it suffices to show that \[ \int _{B(x,2)}G_\l (x, z) G_\l (z,w)  d{\Vol} (z) \leq C   G_\l (x,w).\]
For $d(x,w) \leq 3$, this follows from Corollary ~\ref{convolution}. For $d(x,w) \geq 3$, $G_\l (z,w) \leq C_0^2 G_\l (x,w) $ and $\int _{B(x,2)}G_\l (x, z)  d{\Vol} (z) \leq C$ by (\ref{eqn2.3}).\end{proof}
It follows that
$$
\lim_{\l \to \l_0} -P^3(\l)  \int _{B(x,2)} \int _{\M} G_\l (x, z) G_\l (z,w) G_\l (w, y) d{\Vol} (w)  d\mbox{Vol} (z)  = 0$$
 and the convergence is dominated on $K$ (see Proposition~\ref{limit2}).

 For the rest of the integral, we have
\begin{eqnarray*}
& & -P^3(\l)  \int _{\M\setminus B(x,2)} \int _{\M} G_\l (x, z) G_\l (z,w) G_\l (w, y) d\mbox{Vol} (w)  d\mbox{Vol} (z) \\ 
&=& -P^3(\l)  \int _{\M\setminus B(x,2)}  G_\l^2  (x, z) \frac{G_\l (y,z) }{G_\l (x, z)} \left( \int _{\M} \frac{G_\l (z,w) G_\l (w, y)}{G_\l(y,z)} d\mbox{Vol} (w)\right)  d\mbox{Vol} (z) \\
&=& P^2(\l) \int_2^\infty R e^{P(\l )R}\left( \int_{S(x,R)} e^{-P(\l ) R}   G_\l^2  (x, z) k_\l(x,y,z) \Psi_\l (x,y, z)  dz \right) dR,
\end{eqnarray*}
where $$ \Psi _\l (x, y, z) =   \frac{1}{d(x,z)}\left(-P(\l )  \int _{\M} \frac{G_\l (z,w) G_\l (w, y)}{G_\l(y, z)} d\mbox{Vol} (w)\right).$$ 
As in the proof of Proposition~\ref{limit2}, as $\l \to \l_0$, $P(\l ) \to 0 $ and the above integral converges towards \begin{equation} \label{limit3} \lim\limits_{R \to \infty , \l \to \l_0} \int_{S(x,R)} e^{-P(\l ) R}   G_\l^2  (x, z) k_\l(x,y, z)\Psi_\l (x, y, z)  dz \end{equation} if the limit exists uniformly in $\l$, which we will show for the rest of the proof.
First we study $\Psi _\l (x,y,z) $.
\begin{lem}\label{lem:6.4} There is a  H\"older continuous positive function  $u $ on $SM$  
such that for fixed $x, y, \e>0,$ there exist $R(d(x,y),\e)$ and $\delta = \delta(d(x,y), \e)$ so that
for any $z$ with $d(x,z)>R(d(x,y),\e)$ and $\l $ $\d$-close to $\l_0 $,
\[ \Psi_\l(x,y,z) \; \sim^{1+\e} \;  \B \frac{1}{d(x,z)}\int_0^{d(x,z)} u (\gg_s v_x^z) ds.\] 
\end{lem}
\begin{proof} For $w \in \M$, write $pr(w) $  for the projection of $w$ on the geodesic segment from $x$ to $z$.
For $R>0$, we denote 
$ N_R (x) \; := \; \{ w; w\in \M, d(x, pr(w)) \leq R\}, N_R (z) \; := \; \{ w; w\in \M, d(z, pr(w)) \leq R\}$
and define $$\M_1 := N_{\ab+1}(x)^c \cap N_{\ab+1}(z)^c = \{ w; w \in \M , \ab +1 \leq d(x,pr(w)) \leq d(x,z) -\ab-1 \}.$$
Let us first show that the integral on ${\widetilde{M_1}}^c$ is bounded.
As in Lemma~\ref{lem:4.1}, we decompose ${\widetilde{M_1}}^c$ into $\Ma \cup \Mb \cup \Mc \cup \Md , $ with  $\Ma :=  N_{\ab+1} (x) \setminus B(x, \ab'), $ $\Mb := N_{\ab+1} (x) \cap B(x,\ab'),$ $\Mc :=  N_{\ab+1}(z) \setminus B(z,\ab'),$ $\Md :=  N_{\ab+1}(z) \cap B(z,\ab') $ for $\ab' > \ab$ large enough so that there exists $a_6 =a_6(R, R')$ with the following properties:
\begin{enumerate}
\item for  $ z \in \Ma,$  there is a point $z_1 \in [z,x]$ with $d(z_1,x)>R_0$, $d(z_1,z)>R_0$ and $d(z_1,y) < a_6(R,R')$,
\item for  $ z \in \Mc,$  there is a point $z_1 \in [z,y]$ with $d(z_1,y)>R_0$, $d(z_1,z)>R_0$ and $d(z_1,x) < a_6(R,R').$
\end{enumerate}
 As in Lemma~\ref{lem:4.1}, the choice of $R'$ is uniform on $d(x,y)$. We use the Ancona-Gou\"ezel inequality (\ref{Alano}) to write for instance 
\begin{eqnarray*}  -P(\l)\int_ {\Mc}\frac{G_\l (w,z) G_\l (w, y)}{G_\l(z,y)} dw  \; &= &\; -P(\l) \int_{\Mc} G^2_\l (w,z) \frac{G_\l (w, y)}{G_\l(z,y)G_\l(w,z)} dw\\
 &\leq & \; -P(\l) C \int_{\Mc} G^2_\l (w,z)\frac{G_\l (w, y)}{G_\l(w',y)G_\l(w,w')} dw \\
&\leq &\; -P(\l) C \int_{\M\setminus B(z,\ab')} G^2_\l (w,z) dw
 \end{eqnarray*} 
which is bounded by Remark~\ref{rmk:5.4}. The argument is similar for $\Ma$.

For $w \in \Mb,$ $d(w,x )\leq \ab'$, $\displaystyle \frac{G_\l (w,z) G_\l (w, y)}{G_\l(z,y)}  \leq C(d(x,y)) G_\l (x,w)$ and the integral is finite by (\ref{finiteintegral}). The argument is similar for the integral over $\Md \subset B(z,\ab').$

We conclude that the contribution $-P(\l )  \int _{\M_i} \frac{G_\l (z,w) G_\l (w, y)}{G_\l(y, z)} d\mbox{Vol} (w)$ has an upper bound which depends on $d(x,y)$ and is independent on $d(x,z)$.
 
Now it remains to integrate on $\M_1$. We will find a $\G$-invariant positive H\"older continuous function $u$ on $S\M$ such that for $\l $ close to $\l_0$, independently on $d(x,z)$ but depending on  $d(x,y) $, \[-P(\l )  \int _{\Me} \frac{G_\l (z,w) G_\l (w, y)}{G_\l(y, z)} d\mbox{Vol} (w) \; \sim^{1+\e} \;  \B \int_0^{d(x,z)} u (\gg_s v_x^z) ds.\]
For  a vector $v = \dot{\g}_{v_x^z}(s), 0 \leq s \leq d(x,z)$ and $w \in \Me$, set $$\psi (v, w):= \psi (d(pr(w), \pi (v)))
=\max \{ 1-d(pr(w), \pi(v)), 0 \}$$ and $\ov  u_\l(v) :=   \int _{\M} \psi (v,w)  \frac{G_\l (w,z) G_\l (w, y)}{G_\l(z,y)} d\mbox{Vol} (w) .$ We have
\begin{eqnarray*} 
 \int_{\ab+1}^{d(x,z)-\ab-1} \ov u_\l (\gg_s v_x^z) ds
 &\leq& \int _{\Me} \frac{G_\l (z,w) G_\l (w, y)}{G_\l(y, z)} d\mbox{Vol} (w)  \\ 
&\leq& \int_{\ab}^{d(x,z)-\ab} \ov u_\l (\gg_s v_x^z) ds.
\end{eqnarray*}
We are reduced to find $u$ such that $-P(\l) \ov u_\l  (v) \to \B u(v)$ as $\l \to \l_0$, independently on $d(x,z)$ and depending on  $d(x,y) $. 
Rewrite   $ \ov  u_\l(v)$ as $ \int _0 ^\infty e^{P(\l )r} u_{\l,  r}(v) dr, $
where \[u_{\l,  r}(v) := e^{-P(\l )r} \int _{S(\pi (v), r)} G_\l^2 (\pi (v), w)  \psi (v,w)  \frac{G_\l (w,z) G_\l (w, y)}{G_\l(z,y) G_\l^2 (\pi (v) , w)} dw ,\] 

\begin{figure}[!htb]
\centering
\includegraphics[scale=1.2]{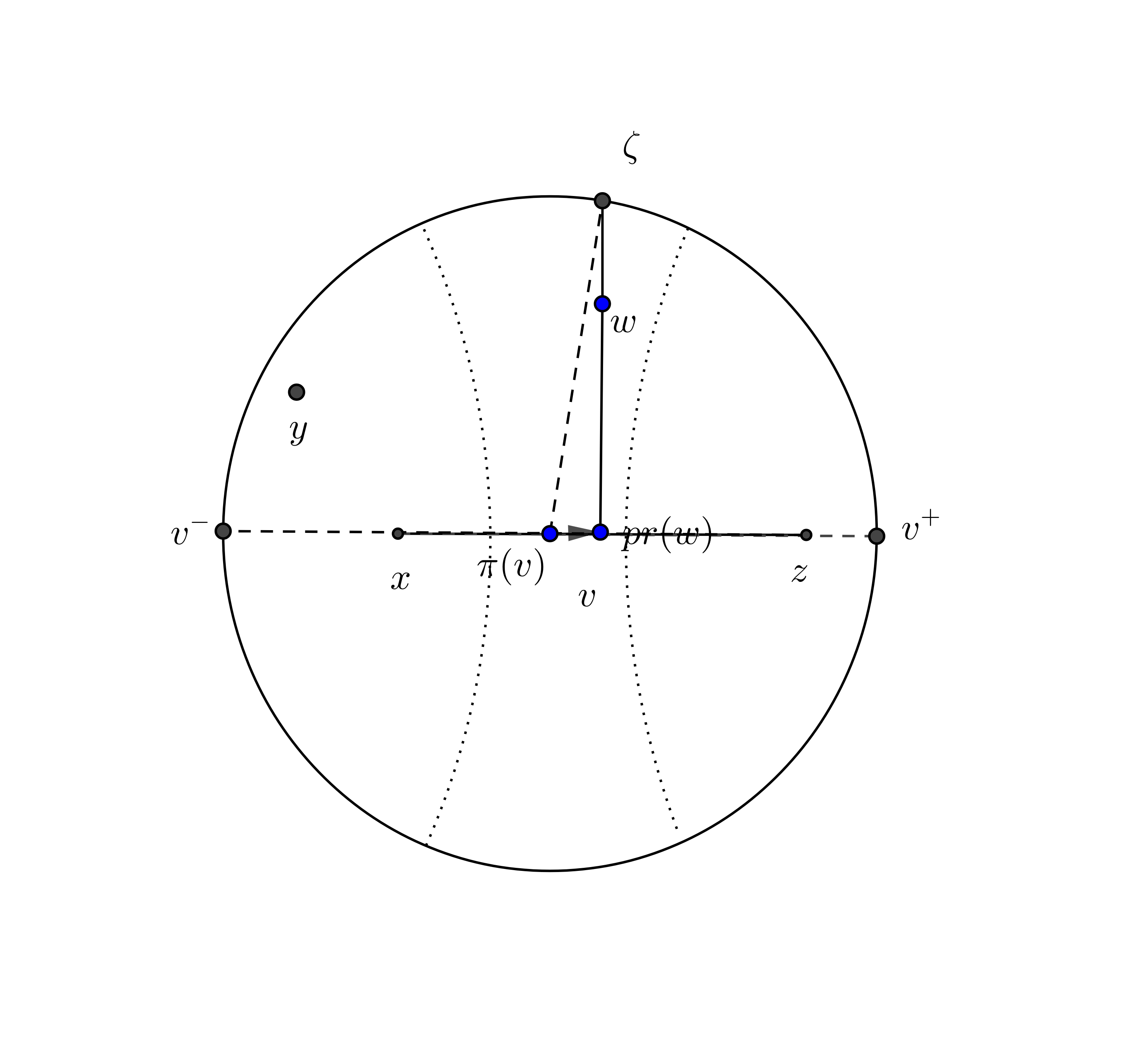}
\caption{Approximating by Naim kernels }
\label{fignew4}
\end{figure}

We choose  $\ab=\ab(x,y, \e)$ larger than 1 such that the angle between the vectors $v_{pr(w)}^x$ and $v_{pr(w)}^y$ is small enough if $d(x,pr(w)) \geq \ab$ and that Proposition~\ref{Naim2} holds for the triples $(x, \pi(v), w),  (y, \pi(v), w)$ and $(z, \pi(v), w):$ for $w\notin  N_\ab(x)\cup N_\ab(z)$  and $pr(w)$ is far from $w$, independently on $d(x,z)$, 
\[  \frac{G_\l (w,z) G_\l (w, y)}{G_\l(z,y) G_\l^2 (\pi (v) , w)} =
 \frac {\th ^\l _{\pi(v)} (w,z )\th ^\l _{\pi(v)} (w,y )}{\th ^\l _{\pi(v)} ( y,z)} \sim^{1+\e}
  \frac {\th ^\l _{\pi(v)} (\z , v^+ )\th ^\l _{\pi(v)} (\z , v^-)}{\th ^\l _{\pi(v)} ( v^-,v^+)},  \]
 where $\z$ is the end point of the geodesic going from $pr( w)$ to $w$ (see Figure~\ref{fignew4}). 

Extend the projection $pr$ to the boundary $\pp \M$. Then for $w \notin  N_1(x)\cup N_1(z)$, $\psi (v, w) = \psi (v, \z ) $. Also, the functions $ d_{\pi (v)} (\z , v^\pm ) $ are bounded away from 0 and the function $\th ^\l _{\pi(v)} (\z , v^+ )\th ^\l _{\pi(v)} (\z , v^-)$ is uniformly H\"older and bounded away from 0.
The denominator $\theta_{\l_0}(v)$ is also H\"older and the approximation is uniformly H\"older continuous.
Therefore, the map $$\zeta \mapsto \psi(v, \zeta) \frac {\th ^\l _{\pi(v)} (\z , v^+ )\th ^\l _{\pi(v)} (\z , v^-)}{\th ^\l _{\pi(v)} ( v^-,v^+)}$$
is H\"older  continuous uniformly on $v$.
By  Proposition~\ref{prop:5.5} centered at $\pi (v)$, there is $R(\e) $ and $\d (\e)  $ such that for $r \geq R(\e), \l \in [\l_0 - \d(\e) , \l _0], $  we have $ u_{\l,  r}(v) \; \sim^{1+\e} \B u(v) $, where 
\begin{equation}\label{u} u(v) \; = \; \int _{\pp \M} \psi (v, \z )  \frac {\th  _{\pi(v)} (\z , v^+ )\th  _{\pi(v)} (\z , v^-)}{\th (v)} d\mu^{\l_0}_{\pi (v)} (\z).\end{equation}
In the above equation, $v$ is a vector in the geodesic from $x$ to $z$.
Now consider $u$ above as a function on $S\M$ and observe that the right hand side of \eqref{u} is well-defined $\G$-invariant and positive on $S\M$. Let us denote the induced function on $SM$ by $u$ again. 

{We claim that the function $u$ is H\"older continuous on $SM$. Indeed,  consider two vectors $v_1,v_2 \in S\M$ at a small distance $d(v_1,v_2)$. For each $t \in [-1,1],$ we associate to $v'_1 = \gg_t v_1$ the vector $v'_2 = \gg_t v_2$. We have $d(v'_1, v'_2) \leq Cd(v_1,v_2).$ We can now pair each vector in $S_{p(v'_1)} \M $ orthogonal to $v'_1$ with a vector in $S_{p(v'_2)} \M $ orthogonal to $v'_2$, also within a distance  at most $ Cd(v_1,v_2).$ By considering their points at infinity, we have paired  each $\z_1 \in \pp \M$ such that  $\psi (v_1, \z_1) >0$  with a  point $\z_2 \in \pp\M$ such that  $\psi (v_1, \z_1 ) = \psi (v_2, \z_2)$ and $d_{p(v_1)} (\z_1, \z_2 ) \leq C (d(v_1, v_2))^\a.$ So, in formula (\ref{u}), the integrand  and the measure, which are H\"older continuous in $\zeta$ and smooth in $\pi(v)$ depend H\"older continuously on $v$.}
 
 It follows that for $\l$ close to $\l_0$, the function $\ov u_\l$ which is a function of $x,y,z$ satisfies  
 $$ -  P(\l) \ov u_\l (v) \sim^{1+\e} -P(\l) \Om u(v) \int_0^\infty e^{P(\l)r} dr = \Om u(v)$$
 independently on $d(x,z)$ and uniformly on $x$ and $y$ as long as $d(x,y)$ is bounded.
\end{proof}

\begin{proof}[Proof of Proposition~\ref{Proposition3.17}]  By \eqref{limit3} and Lemma~\ref{lem:6.4}, it remains to show that the limit 
\[  \lim_{R \to \infty, \l \to \l_0} \int_{S(x,R)} e^{-P(\l ) R}   G_\l^2  (x, z) k_\l(x,y,z) \left(\frac{1}{R} \int _0^R u (\gg_s v_x^z) ds \right) dz ,\] 
exists uniformly in $\l$ where the function $u$ is given by \eqref{u}.
As in the proof of Proposition~\ref{limit2}, we can replace $k_\l (x,y,z) $ by $k_{\l_0} (x, y, \s _x (v_x^z))$ for $R $ sufficiently large and $\l $ close to $\l _0$.
By Proposition~\ref{renewal3}, for $R$ large and $\l_0 - \l$ small,
\[  e^{-P(\l ) R}  \int_{S(x,R)}  G_\l^2  (x, z) k_{\l_0}(x, y, \s _x (v_x^z)) \left(\frac{1}{R} \int _0^R u (\gg_s v_x^z) ds \right) dz \sim  \B^2 c(x,y) \int _{SM} u \,dm_{\l_0} \]
by (\ref{Harish-Chandra}). Proposition~\ref{Proposition3.17} follows since
\begin{eqnarray*} 
&&\int _{SM} u dm_{\l_0} =  \quad \int_{SM_0} \int _{\pp \M} \psi (v, \z )  \frac {\th _{\pi(v)} (\z , v^+ )\th  _{\pi(v)} (\z , v^-)}{\th (v)} d\mu^{\l_0}_{\pi (v)} (\z)dm_{\l_0}(v) \\
&=& \; \int_{SM_0} \int _{\pp \M} \psi (v, \z )  \frac {\th_{\pi(v)} (\z , v^+ )\th_{\pi(v)} (\z , v^-)}{\th(v)} d\mu^{\l_0}_{\pi (v)} (\z)\Om \th^2(v) d\mu^{\l_0}_{\pi (v)}(v^-) d\mu^{\l_0}_{\pi (v)}(v^+) dt\\
&=&  \frac{\Om}{\Upsilon} \; \int _{\pp \M} \int_{(v^-, v^+, t) \in SM_0}  \psi (v, \z )  dt d\wt\tau^{\l_0} _{\pi(v)} (v^+ , v^-, \z) \\
& =&  \frac{\Om}{\Upsilon} \;  \wt\tau^{\l_0} _{\pi (v)} (S^2 M_0) =  \frac{\Om}{\Upsilon} .
\end{eqnarray*}
Recall that $\widetilde \tau^{\l_0} _{x}, \Upsilon$ are defined in \eqref{eq:tau} and \eqref{Upsilon}.
\footnote{The last equality is direct: take a point $(v^+, v^-, \z) $ well inside $S^2M_0$. Then, clearly, $\int_{(v^-, v^+, t) \in SM_0}  \psi (v, \z )  dt = 1.$ The boundary effects for the other points compensate exactly, so that the integral $ \int _{\pp \M} \int_{(v^-, v^+, t) \in SM_0}  \psi (v, \z )  dt d\wt\tau^{\l_0} _{\pi(v)} (v^+ , v^-, \z) $ is $ \wt\tau^{\l_0} _{\pi (v)} (S^2 M_0).$}
\end{proof}
\begin{proof}[Proof of Theorem~\ref{Theorem3.1}] Set $F(\l ) = \frac{\pp }{\pp \l} G_\l (x,y) $. By Proposition \ref{limit2} and Proposition~\ref{Proposition3.17}, 
$$\lim_{ \l \to \l_0}-P(\l ) F(\l ) =\B \; c(x,y) \;\ \textrm{and} \; \lim_{\l \to \l_0} -P^3(\l ) F'(\l ) =2 \frac{\Om^3}{\Upsilon} c(x, y).$$
 
 It follows that $\displaystyle \frac{2F'(\l) }{F(\l )^3} $ converges towards $\displaystyle \frac{4}{ \Upsilon}( c(x,y)) ^{-2} .$ Since $F(\l )$ goes to $\infty $ as $\l \to \l_0 $, we conclude that  $F(\l) \sim \frac {\sqrt{\Upsilon}} {2 }\frac{c(x,y)}{\sqrt{ \l_0 - \l }} .$  \end{proof}
By Proposition~\ref{limit2} and Theorem~\ref{Theorem3.1}, we obtain
\begin{cor}\label{equivalentP}
As $\l \to \l_0$, 
$$ - \frac{P(\l)}{\sqrt{\l_0 - \l}} \to \frac{2 \Om}{ \sqrt{\Upsilon}}.$$
\end{cor} 

Applying Proposition~\ref{Proposition3.17} and Corollary~\ref{equivalentP}, we get
 \begin{cor}\label{cor:6.5} For all $x, y \in \M,$
\[ \lim_{\l \to \l_0}(\l_0 -\l)^{3/2} \int  _{\M\x\M} G_\l (x,z) G_\l (z,w) G_\l(w,y) d\Vol (z) d\Vol (w) 
\; = \; \frac{\sqrt \Upsilon}{8} c(x,y).\] 
Moreover, for any compact neighborhood $K$ of $x$ in $\M$, there is $\l'< \l_0$ such that $$y \mapsto \sup_{\l, \l' \leq \l \leq \l_0 } (\l_0 -\l)^{3/2} \int  _{\M\x\M} G_\l (x,z) G_\l (z,w) G_\l(w,y) d\Vol (z) d\Vol (w) $$ is integrable on $K$.
\end{cor}

\subsection{Proof of Theorem~\ref{locallimit} and Theorem~\ref{Phi_zero}  }\label{sec:8}
The proof relies on the following Proposition, based on Hardy-Littlewood Tauberian Theorem:
 \begin{prop}\label{weaklocallimit} Fix $x_0 \in \M$. Let $F$ be a nonnegative $C^\infty$ function on $\M$, with compact support. Then,
 \begin{eqnarray*}&  & \lim\limits_{t\to \infty } t^{3/2}  \int _{\M \x \M} e^{\l_0t} \Pp(t,x, y)F(x)F(y) \, d\mbox{Vol}(x) d\mbox{Vol}(y) \\ &  &\quad = \;  \frac{\sqrt{\Upsilon}}{4 }  \int _{\M \x \M} c(x,y) F(x) F(y)  \, d\mbox{Vol}(x) d\mbox{Vol}(y),\end{eqnarray*}
 where $c(x,y) $ is given by (\ref{Harish-Chandra}). \end{prop}
 \begin{proof} Set $\mu_F$ for the spectral measure of $F$, i.e. the Borel finite measure on the spectrum $[0, +\infty )$ of $\D - \l_0 $ such that, for all $m \geq 0 $, { \[ \int _{{\M}} F(x) \D^m F(x)\,  d\mbox{Vol}(x) \;  =\; \int _{0}^{+\infty} (\varpi - \l_0 )^m \, d\mu _F (\varpi) .\] }
The function  \[ c_F(t) \; := \;  \int _{\M \x \M} e^{\l_0t} \Pp(t,x, y)F(x)F(y) \, d\mbox{Vol}(x) d\mbox{Vol}(y) = \int _{0}^{+\infty } e^{-\varpi  t} \, d\mu _F(\varpi) \] 
is nonincreasing  in $t$. It satisfies the following property
\begin{lem}\label{lem:6.7} {For all $s>0$,}
\[ \int _0^{+\infty } e^{-s t} t^2 c_F(t) \, dt =  2 \int  _{\M^4} G_{\l_0 -s} (x,z) G_{\l_0-s} (z,w) G_{\l_0 -s}(w,y) F(x) F(y)  \,d{\Vol }^4(z,w,x,y).\] \end{lem}
\begin{proof} On the one hand, we have 
\[ \int _0^{+\infty } e^{-s t} t^2 c_F(t) \, dt \; = \; \int _{\M\x \M} \int_0^\infty t^2  e^{(\l_0 - s )t} \Pp(t,x, y) dt F(x)F(y) \, d\mbox{Vol}(x) d\mbox{Vol}(y) .\]
On the other hand, we may write
\small{ \begin{eqnarray*}
&  & 2 \int  _{\M^4} G_{\l_0 -s} (x,z) G_{\l_0-s} (z,w) G_{\l_0 -s}(w,y) F(x) F(y)  \,d{\Vol }^4(z,w,x,y) \\
&=& 2 \int  _{\M^4 \x \R_+^3} e^{(\l_0 -s)(t+u+v)} \Pp(t,x,z) \Pp(u,z,w)  \Pp(v,w,y) \, dtdudv  F(x) F(y)  \,d{\Vol }^4(z,w,x,y) \end{eqnarray*}}
Introducing the variables $u+v =: r$ and $t+r =: \tau $ and using the semigroup property of the heat kernel, we obtain 
\[ \int _{\M^2} \left(\int_0^\infty \tau ^2  e^{(\l_0 - s )\tau } \Pp(\tau ,x, y) \, d\tau \right) F(x)F(y) \, d\mbox{Vol}^2(x,y). \]\end{proof}
By Corollary~\ref{cor:6.5} and Lemma~\ref{lem:6.7} we have, as $s \to 0,$ \footnote{Here we use the domination from Corollary~\ref{cor:6.5}, which follows from all the preceding domination results in Proposition~\ref{limit2} and Proposition~\ref{Proposition3.17}.}
\[ s^{3/2} \int _0^{+\infty } e^{-s t} t^2 c_F(t) \, dt \to  \frac{\sqrt{\Upsilon}}{4 }  \int _{\M \x \M} c(x,y) F(x) F(y)  \, d\mbox{Vol}(x) d\mbox{Vol}(y).
\]
By Hardy-Littlewood Tauberian Theorem (\cite{F} p. 445), as $T \to \infty$, we have
\begin{equation}\label{eqn:sim} \int_0^T t^2 c_F(t) dt \; \sim \; \frac{\sqrt \Upsilon}{4\G(5/2)} T^{3/2}\int _{\M \x \M} c(x,y) F(x)F(y) \, d\mbox{Vol}(x) d\mbox{Vol}(y).\end{equation}
Now we claim that 
\[ c_F(t)  \sim   \frac{\sqrt \Upsilon}{2\sqrt \pi t^{3/2}} \int _{\M \x \M} c(x,y) F(x)F(y) \, d\mbox{Vol}(x) d\mbox{Vol}(y).\] 
Indeed, by setting $\Xi T^{3/2}$ to be the right hand side of the equation (\ref{eqn:sim}), we have, for all $\e >0$, 
$$  \int_T^{T(1+\e)} t^2 c_F(t) dt \; = \; T^{3/2} \Xi (1+\e)^{3/2} - \Xi T^{3/2} + o(T^{3/2}) = \Xi T^{3/2} ((1+\e)^{3/2} -1 +o(1)).$$
On the other hand, since $c_F(t)$  is a non-increasing  function of $t$, for $\e >0 $ small,
$$ \int_T^{T(1+\e)} t^2 c_F(t) dt   \leq c_F(T) \int_T^{T(1+\e)} t^2 dt = c_F(T)  T^3 (\e +  \e^2 + \e^3/3) .$$
Comparing the two inequalities yields:
$$ \liminf _{T \to \infty } c_F(T) T^{3/2} \; \geq \;  \frac{3\Xi }{2} + o(\e).$$
 One shows  in the same way, using $\int_{T(1-\e)}^T$,  that $ \limsup _{T \to \infty } c_F(T) T^{3/2} \; \leq \;   \frac{3\Xi }{2} .$ This proves Proposition~\ref{weaklocallimit}.
 \end{proof}
 
 \begin{proof}[Proof of Theorem~\ref{locallimit} and Theorem~\ref{Phi_zero}.]

Since $c(x,y) = \int k_{\l_0}(x,y) d\mu_x$, and $k_{\l_0}(x,y)$ is smooth as a $\l_0$-harmonic function, the function $c(x,y)$ is smooth in $x$ and $y$. Moreover, by Proposition \ref{boundedgradient} below,  $ \log \Pp(t,x,y)$ has bounded gradient, uniformly in $t$ large.  We can therefore apply Proposition~\ref{weaklocallimit} to  functions $F$ with compact support such that the measures $F(x) d\Vol (x) $ converge to the Dirac measure $\d_{x_0} $ to get 
\[ \lim\limits _{t\to \infty }  t^{3/2} e^{\l_0 t} \Pp(t,x_0,x_0) \; = \; \frac{\sqrt \Upsilon}{2\sqrt \pi } c(x_0,x_0). \] 

We get the general case of $x_0 \not = x_1$  of Theorem~\ref{locallimit} and Theorem~\ref{Phi_zero} in the same way by applying Proposition~\ref{weaklocallimit} to functions that approximate $\d_{x_0} + \d_{x_1} .$  
\end{proof}

\section{Appendix I: Uniform mixing}\label{Appendix}

In this section, we establish a uniform power mixing of the geodesic flow for Gibbs measures, when the potential varies in a neighbourhood of the space $\K_\a$ of functions which will be defined shortly. The proof combines the ideas from  \cite{P1} and \cite{P2}, with a slightly different framework. For the comfort of the reader, we recall the different steps in our notations.

\subsection{Uniform mixing and three-mixing}\label{sec:7.1}
Let $ \X:= ( X, \A, m; \bg_t, t \in \R )$ be a system with a one parameter group $ \{\bg_t, t \in \R \} $ of measurable transformations of the space $(X, \A)$ preserving a probability measure $m$. For bounded measurable functions $f,g, h$ we define the correlations functions for $s,t \geq 0$:
\begin{eqnarray*} 
\rho _{f,g,m} (t) &= &  \int  f(x) g (\bg_t x ) dm (x) - \int f \, dm \int g \, dm  \\ 
\rho _{f,g,h,m } (s,t) &=&   \int  f(x) g(\bg_sx)  h (\bg_{s+t} x ) dm (x) - \int f \, dm \int g \, dm  \int h\, dm   \\ 
\ov {\rho }_{f,g,h,m } (t) &=&   \frac{1}{t} \int_0^t \left[ \int  f(x) g(\bg_sx)  h (\bg_{t} x ) dm (x) \right] ds- \int f \,dm \int g\, dm  \int h \,dm 
\end{eqnarray*} 

The system $\X$ is called {\it { mixing }} if $ \lim _{t \to \infty } \rho _{f,g,m }(t) = 0 $ for all bounded functions $f,g $, {\it { 3-mixing }} if $ \lim _{s,t \to \infty } \rho _{f,g,h ,m}(s,t) = 0 $ for all bounded functions $f,g,h $ and {\it { average 3-mixing }} if $ \lim _{t \to \infty } \ov {\rho }_{f,g, h,m }(t) = 0 $ for all bounded functions $f,g, h $. It is a well-known open problem whether mixing implies 3-mixing. It is easy to see that mixing implies average 3-mixing. 

Let us consider the rate of mixing.
A system $\X$ is called \emph{power mixing} for a class $\K$ of functions if for $f, g \in \K$,
$ \rho_{f,g, m}(t)$ decays polynomially (see Theorem~\ref{theo:uniform2mixing} for a precise statement). Below, we will show a uniform version of a power mixing of the geodesic flow for the class $\K=\K_\a$ which we define now.

Let $\a > 0 $. We denote $\K_\a $ the space of functions $f$ on $X$ such that $\| f\|_\a < \infty $, where
\[ \| f\|_\a := \sup _x |f(x)| + \sup _{x \not = y } \frac{|f(x) -f(y) |}{(d(x,y))^\a} . \]

From now on, let $\bg_t$ be an Anosov flow. For any potential function $\varphi \in \K_\a$, there is a unique invariant probability measure $m_\varphi $ attaining the supremum of the mesure theoretic pressure $h_m(\bg) + \int \varphi dm$ in the set $\Omega$ of all $\bg_t$-invariant Borel probability measures, i.e.:
\[P(\varphi ) : = \sup _{m \in \Omega } \left\{ h_m (\bg) + \int \varphi \, dm \right\} = h_{m_\varphi } (\bg) + \int \varphi \,  dm_\varphi, \]
where $h_m(\bg) $ denotes the measure theoretic entropy of $m$ (see e.g. \cite{PP}). The quantity $P(\varphi)$ is called the topological pressure of the potential function $\varphi$. The mapping $\varphi \mapsto m_\varphi $ is continuous from $\K_\a$ to the space of measures on $X$ endowed with the weak* topology.

The following property is important in Dolgopyat's approach to the speed of mixing.

\begin{defi} A system $\X$ is topologically power mixing if there exists $t_0, \d >0$ such that for any $r$, and $t > \max\{ \frac{1}{r^\d}, t_0 \}$, and any $x,y$, 
$$ \bg_t(B(x,r)) \cap B(y,r) \neq \emptyset.$$
\end{defi}

We now establish a local uniform power mixing for topologically power mixing Anosov flows, for Gibbs measures associated to potentials $\vf$, and for functions in $\K_\a$. The mixing rate is uniform as we vary the potential $\vf$ in a small neighbourhood in $\K_{\a_0}$, for $\a$ and $\a_0$ sufficiently small.

\begin{theo} \label{theo:uniform2mixing} Let $\X$ be a topologically power mixing Anosov flow. There exists $\a_0>0$ with the following property: let $\vf_0 \in \K_{\a_0}(X)$ be a potential. There exist $\e >0$, $\a>0$ and $C'_0, c'_0>0$ such that for all $\vf$
with $||\vf - \vf_0||_{\a_0} < \e$ and all $f,g,h \in \K _\a,$ we have, for all positive $s,t$:
\begin{equation}\label{uniform2mixing} \big| \rho_{f,g,h, m_{\varphi} }(s,t) \big|\; \leq \; C'_0 \|f\|_\a \|g\|_\a \|h\|_\a [(1 + s)^{-c'_0 } +(1 + t)^{-c'_0 } ]. \end{equation} \end{theo}

\begin{prop}
\footnote{In each of subsection~\ref{sec:7.2.2} and \ref{sec:7.2.3}, we prove Theorem~\ref{theo:uniform2mixing} for some class of functions $f,g,h$ with $\int f = \int g = \int h = 0$, prove Proposition~\ref{theo:uniformmixing}, and then use Proposition~\ref{theo:uniformmixing} to reduce the proof of Theorem~\ref{theo:uniform2mixing} to the case when $\int f = \int g = \int h = 0.$}
\label{theo:uniformmixing} Let $\X$  be a topologically power mixing Anosov flow. There exists $\a_0>0$ with the following property: let $\vf_0 \in \K_{\a_0}(X)$ be a potential. There exist $\e >0$, $\a>0$ and $C, c>0$ such that for all $\vf$ with $||\vf - \vf_0||_{\a_0} < \e$ and all $f, g \in \K _\a,$ we have, for all positive $t$:
\begin{equation}\label{uniformmixing} \big|\rho_{f,g, m_{\varphi }} (t)\big| \; \leq \; C \|f\|_\a \|g\|_\a (1 + t)^{-c } . \end{equation}

 \end{prop}
\begin{cor}\label{theo:uniformavmixing}Let $\X$  be a topologically power mixing Anosov flow. There exists $\a_0>0$ with the following property: let $\vf_0 \in \K_{\a_0}(X)$ be a potential. There exist $\e >0$, $\a>0$ and $C'_0, c'_0>0$ such that for all $\vf$ with $||\vf - \vf_0||_{\a_0} < \e$ and all $f,g,h \in \K _\a,$ we have, for all positive $t$:
\begin{equation}\label{uniformavmixing} \big| \ov {\rho}_{f,g , h, m_{\varphi} } (t) \big|\; \leq \; C_0' \|f\|_\a \| g \|_\a \|h\|_\a  (1 + t)^{-c_0' }  . \end{equation} \end{cor}

We assume now that the system $\X $ is the geodesic flow $\gg_t , t \in \R$ on the unit tangent bundle $X = SM$, where $M$ is a closed negatively curved manifold.  

Liverani proved exponential mixing for contact Anosov flows for the Liouville measure, which implies exponential mixing for the geodesic flow on manifolds of negative curvature for the Liouville measure \cite{Li}. It implies that the geodesic flow is topologically power mixing. Thus we can apply the above theorems to the geodesic flow and the Gibbs measure associated to $\vf_{\l_0}$ to obtain Propositions~\ref{4.1} {and \ref{4.2}}.

\subsection{Proof of Theorem \ref{theo:uniform2mixing} and Proposition~\ref{theo:uniformmixing}}

First, following Bowen and Ruelle \cite {B}, \cite{BR}, we can reduce the problem to the corresponding problem on suspended symbolic flows by introducing Poincar\'e sections for the flow with Markov property (see also \cite{PP} Chapter 9 and Appendix III), in such a way that H\"older continuous functions on $SM$ correspond to H\"older continuous functions on the symbolic system. (The H\"older constant might change, say from $\a_0$ to $2 \a$.)

We may thus assume that there is a subshift of finite type $(\Si, \s)$ and a positive $\a$-H\"older continuous function $\tau $  on $\Si $ such that the system $\X$ is the suspension flow $\sigma _t (x, r) = (x, r+t) $ on the set $ \Si^\tau := \{ (x, r): x \in \Sigma, 0 \leq r \leq \tau (x) \}/ [(x,\tau (x)) \sim (\s x, 0)].$ Let us denote by $[a_0, \cdots, a_k]$ the cylinder set $\{ x : x_i = a_i, i=0,\cdots, k\}$. Let us also define $d_\a$ on the space $\Si_+$ of one-sided sequences with the left-shift by $d_\a(x,y)=\a^k$, where $k$ is the first index for which $x_k, y_k$ are not equal.
Let us denote by $\K_\a(\Sigma_+)$ the space of $d_\a$-Lipschitz functions on {the space $\Si_+$ of one-sided sequences}. Let $\varphi \in \K_{2\a}(\Si^\tau)$  be a potential function on $\Si^\tau$. Then the function $ \int_0^{\tau (x)} \vf (x,r) \, dr $ is $d_{2\a}$-Lipschitz on $\Si.$

We may assume that the function $\tau$ is a function on $\Si_+$ in the sense that  $\tau (x) = \tau(y) $ if the points $x$ and $y$ in $\Si $ have the same nonnegative coordinates. Moreover, the function $\tau $ is a $d_\a$-Lipschitz function on $\Si _+$. The function $\phi_1$ on $\Si_+$ associated to $ \int_0^{\tau (x)} \vf (x,r) \, dr $ is a $d_\a$-Lipschitz function (\cite{Sin}, \cite{Bo}, see also Proposition 1.2 of \cite{PP} for example). Now normalize $\phi_1 $ to obtain a  $d_\a$-Lipschitz function $\phi$ with $\LL_{\phi}1=1$, where 
\begin{equation}\label{transfer1} \LL _{\phi } F (x) := \sum _{y; \s y = x } e^{\phi (y)} F(y)
\end{equation} is the transfer operator associated to $\phi$ (see e.g. \cite{PP} page 115 for these classical reductions).
We conclude that the map $\cT$ sending $\vf$ to $\phi$ is continuous from $\K_{\a_0}(SM)$ into $\K_\a (\Si _+) $. The equilibrium  measure $m_\vf $ for the function $\varphi$ is of the form \[ m_\vf = \frac {1}{\int \tau \, d\nu _\phi} (\ov \nu _\phi \otimes dr) \big|_{\Si^\tau }, \]
where $\ov \nu _\phi$ is the unique $\s $-invariant probability measure on $\Si$ such that its projection $\nu _\phi$ to $\Si _+ $ satisfies, for all functions $F \in C(\Si _+) $,
\begin{equation}\label{transfer2} \int \LL _{\phi } F \, d\nu_\phi = \int F \, d\nu _\phi. \end{equation}
Let us denote $\phi^{(k)}(x) = \phi(x) + \phi(\s (x)) + \cdots + \phi(\s^{k-1} (x))$.
For a given $\varphi_0$, we choose an $\e_1$-neighborhood of $\phi_0=\cT \vf_0$ so that there exists a constant $C_1 \geq 1$ with, for all {\it{normalized}} $\phi$ in the  $\e_1$-neighborhood of $\phi_0,$ all $k \in \N,$
\begin{equation}\label{eqn:8.6}
\left|\frac{e^{\phi^{(k)}(x)}}{e^{\phi^{(k)}(y)}} - 1 \right|  \leq C_1 \a^{-k} d_\a(x,y),\quad  \forall x,y \in \Si _+ \end{equation} \begin{equation}\label{eqn:8.6bis} \mathrm{and} \;\; C_1^{-1}  \leq \frac{\nu_\phi [a_0, \cdots, a_{k-1}] }{e^{\phi^{(k)}(x)}}\leq C_1, \quad  \forall x \in [a_0, \cdots, a_{k-1}]. 
\end{equation}
 \footnote{ Assume   the coordinates of $x$ and $y$  coincide  up to $k+n-1, n\geq 0$. Then,  for $j <k,|\phi(\s^j x)-\phi (\s ^jy)|\leq \a^{-j} d_\a(x,y) ||\phi|| .$ Therefore,  $|\phi^{(k)} (x) - \phi^{(k)} (y)| \leq \sum_{j=0}^{k-1}  \a^{-j} d_\a(x,y) ||\phi|| \leq  \a^{-k} d_\a(x,y) \frac{||\phi||}{1-\alpha}$.  If $x,y $ are not in the same $ [a_0, \cdots, a_{k-1}],$ then $\a^{-k} d_\a (x,y)$ is big. Note that the denominator of the second inequality does not have $e^{Pk}$ since $P=0$ for a normalized $\phi .$}

With those choices, for all $\phi$, 1 is an isolated eigenvalue of $\LL _{\phi } $ with eigenfunction the constant 1 (see \cite{PP}, Theorem 2.2 page 21). A ball of radius $r$ in $\Si ^\tau$ contains a cylinder of length $- C \log r $ in $\Si $ times an interval of length $cr$ in the flow direction. Its image on the manifold contains a ball of radius $r^{D}$, for some $D$. Therefore, the suspension  flow $\X$ is topologicallly power mixing for the symbolic distance. 

\begin{remark} The rest of the proof in this section follows the ideas of D. Dolgopyat (\cite{D2}). In order to check that all the arguments are uniform for equilibrium measures $m_\vf $ for $\vf$ in a neighborhood of $\vf_0$, we found it more convenient to follow \cite{Me}. In particular,  the constants $C_1, C_6, C_7, \g_3$ in this section  coincide with those in \cite{Me}.
\end{remark}

\subsubsection{Properties of the complex transfer operator}

In this subsection, we will denote the space of complex $d_\a$-Lipschitz continuous functions on $\Si_+$ by $ \K_\a (\Si _+ ) $ again. Let $\phi \in \K_\a (\Si_+)$ with $\LL_\phi 1= 1$.
We define the complex transfer operator $\LL _{\phi + s \tau}, s \in \C$ on $\K_\a (\Si _+) $ by 
$$\LL _{\phi +s\tau } F (x) := \sum _{y; \s y = x } e^{\phi(y) + s \tau (y)} F(y)  .$$
Following \cite{Me}, set $s=a+ib$.

We recall that, by mixing of the geodesic flow, $ \| \LL _{\phi + i b \tau } \|_\a < 1 $ for $b  \neq 0$  (see \cite{PP} Proposition 6.2). In particular, for $b \neq 0,$ the series $\sum _n \LL ^n_{\phi + ib \tau} $ converges as a series of operators in $\K_\a (\Si _+) .$ The sum $\sum _n \LL ^n_{\phi + s \tau} = (I - \LL_{\phi + s \tau} )^{-1} $ depends analytically on $s=a +ib$ for $a<0$ and has a continuous extension to $a=0, b\neq 0$. Dolgopyat's method allows to extend analytically that sum beyond the imaginary axis (Propositions \ref{spectralgap} and \ref{3.5}).
\begin{prop}\label{spectralgap} There is $\d = \d_{\phi_0}>0, \e>0 $ such that, for all normalized $\phi$ with $||\phi-\phi_0||_\a<\e$,  the mapping $s \mapsto \sum _n \LL ^n_{\phi+ s\tau  }$ is meromorphic on $V_\d,$ where    $$V_\d := \{ s=a + ib : |b | < 2 , \;  | a | <  \d \} $$ with a simple pole at $s = 0 $. Moreover, for a function $K \in \K_\a (\Si _+) ,$ the residue at $s = 0 $ of the meromorphic function $s \mapsto \sum _n \LL ^n_{\phi + s \tau} K $  (with values in $\K_\a $) is a  constant function with value $\nu _\phi (K).$ \end{prop}
\begin{proof} For a fixed $\phi$ , this follows from  \cite {PP}, Proposition 6.2 and Theorem 10.2, with a fixed $\d=\d _\phi $. {By \cite{Ka} Theorem IV.3.1 and compactness of  the closure $\overline {V_\d}$,  there is a neighborhood $\mathcal {U}_0 $ of $\phi _0 $ such that for normalized $\phi \in \mathcal{U}_0$, the rest of the spectrum of $ \LL _{\phi + s \tau}, s \in \overline {V_\d}, $ is separated from 1 by $\d=\d_{\phi_0}$.} \end{proof}

\begin{prop}\label{3.5}(Compare with Lemma 3.5 of \cite{Me}) Let $\X$ be a topologically power mixing Anosov flow. Let $\phi_0$ be a $\a$-H\"older continuous function. 
There exist constants  $\e, \d, \b, D_0$ such that, for all normalized $\phi, \| \phi -\phi _0 \|_\a < \e,$   the series of operators $\sum _n \LL ^n_{\phi +s\tau } $ has an analytic extension on the region $U= U_{\d,\b}$, where  $$U_{\d,\b} \;  := \; \{ s, s= a+ib; |b| > 1 , \;  |a | <  \frac{2\d}{|b|^{\b/2}} \} $$ and, for $s \in U,$ 
  \begin{equation} \|\sum _n \LL ^n_{\phi +s \tau } \|_\a  \; \leq \; D_0 |b |^{D_0} . \end{equation} 
\end{prop}
\begin{proof} As in \cite{Me}, we carry the calculations  for $0 \leq a \leq 1$ and $b >1$. They are analogous for $b< -1$ and for $-1 \leq a \leq 0.$
More precisely, we find a neighborhood $\mathcal U$ of $\phi _0 $  and $\theta >0, C>0$ such that  the conclusion holds for all $s=a+ib$ with $|b|>1,$ $ |a|<C^{-1}|b|^{-\theta}$,
and for all normalized $\phi \in \mathcal U$.
We first have the preliminary estimate of \cite{Me} in a uniform way.
\begin{lem}\label{fundamental} (Lemma 3.7 of \cite{Me}) There exist $C_6, C_7, \gamma_3, \e_2>0$ such that for all normalized $\phi$ with $||\phi-\phi_0||_\a<\e_2$,
\begin{enumerate}
\item $|\LL_{\phi+ ib\tau} |_\infty \leq 1,$
\item $|| \LL_{\phi+ib\tau}^n F||_\a \leq C_6 \{ b |F|_\infty + \a^n||F||_\a \}$ for all $n \geq 1$ and $F \in \K_\a(\Si_+)$,
\item $|| \LL_{\phi}^n F - \int_{\Si_+} F d\nu_\phi ||_\a \leq C_7 \gamma_3^n ||F||_\a$ for all $n \geq 1$ and $F \in \K_\a(\Si_+)$.
\end{enumerate}
\end{lem}
\begin{proof}
Part (2) comes from the basic inequality (\cite{PP}, Proposition 2.1) thus $C_6$ is uniform in $\phi$. Part (3) comes from the spectral gap of $\LL_\phi$ thus $C_7$ and $\gamma_3$ can be chosen uniformly in a neighbourhood of $\phi_0$ (see e.g. Kato \cite{Ka} Theorem IV.3.1).
\end{proof}
As in \cite{Me}, define
\[ \| f\|_b := \max \left\{ |f|_\infty,  \frac{1}{2C_6 b} \sup _{x \not = y } \frac{|f(x) -f(y) |}{(d(x,y))^\a} \right \}. \]
Since one may assume that $2C_6 b>1$, we have
$$ || F||_b \leq ||F||_\a \leq (2C_6b+1) ||F||_b,  $$
which implies that $||\LL||_\a/||\LL||_b$ lies between $2C_6b+1$ and $(2C_6b+1)^{-1}$.

\

Let $M_b F = e^{-ib\tau} F \circ \sigma$. 
\begin{defi} The operator $M_b$ has no approximate eigenfunction if there exists $N \in \N$ such that for every triple $( \theta \geq N, \b>0, C \geq 1)$, there exists $k=k(\theta, \b, C)$ such that for all $ (b, \rho, F)$ with $|F|=1, \rho \in \R$ and $|b|>k$, 
$$ |M_b^{\b \log |b|} F(y)-e^{i \rho} F(y) |  \geq C |b|^{-\theta}, $$ for some $y$.\end{defi}

\begin{lem}[Uniform version of Section 3.2 of \cite{Me}]\label{uniformDolgo} Consider the following conditions.
\begin{enumerate}
\item $M_b$ has no approximate eigenfunction.
\item There exist constants  $\e,D$ such that, for all normalized $\phi$ with $||\phi-\phi_0||_\a<\e$, and $b>1,$   the series of operators $\sum _n \LL ^n_{\phi +ib \tau } $ satisfies
  \begin{equation*}  \|\sum _n \LL ^n_{\phi +ib \tau } \|_b  \; \leq \; D |b |^{D} . \end{equation*} 
\item There exist constants  $\e, \d, \b, D_0$ such that, for all normalized $\phi$ with $||\phi-\phi_0||_\a<\e$,  the function $s \mapsto \sum _n \LL ^n_{\phi +s\tau } $ has an analytic extension to the region $U_{\d,\b}$  and  for $s \in U_{\d,\b},$ 
  \begin{equation*}  \|\sum _n \LL ^n_{\phi +s \tau } \|_b  \; \leq \; D_0 |b |^{D_0} . \end{equation*}  \end{enumerate}
With the above notations, (1) implies (2) and (2) implies (3).

\end{lem}

\begin{proof} See Section 3.2  of \cite{Me}. Let $\e_1$ be a constant such that $C_1$ in equation ~(\ref{eqn:8.6}) and $\a_1, \a_2$ in \cite{Me} are uniform in $\phi$ in $\e_1$-neighborhood of $\phi_0$. Now let $\e = \min \{ \e_1, \e_2\}$, where $\e_2$ is chosen as in Lemma~\ref{fundamental}. 
\end{proof}
{We now achieve the proof of Proposition~\ref{3.5}}:
topologically power mixing of $\X$ implies that $M_b$ has no approximate eigenfunction by Sections 3 and 5 of \cite{D2}, thus Proposition~\ref{3.5} follows.
\end{proof}

\subsubsection{One-sided smooth functions}\label{sec:7.2.2}
We start by proving  Theorem \ref{theo:uniform2mixing} for a  particular space of functions. For $\a >0 $ and $M \in \N$, let $\K_{\a, M}^+  $ be the set of functions $f$ on $\Si ^\tau $ with the following properties:
\begin{itemize}
\item for all $x \in \Si ,$ $f(x,r) = 0 $ for $r$ outside the interval $[ \frac{\inf \tau }{3},  \frac{2\inf \tau }{3}]$, 
\item for all $x \in \Si ,$ $r \mapsto f(x,r)$ is of class $C^M$,
\item for all $r \in [ \frac{\inf \tau }{3},  \frac{2\inf \tau }{3}]$, $x \mapsto f(x,r) $ depends only on the nonnegative coordinates of $x$ and
\item the functions $ \frac{\partial ^k f}{\partial r^k }  (x,r)$, for  $0\leq k \leq M $ are $\a$-H\"older continuous in $x \in \Si $ and continuous in $r$.
\end{itemize}
For $f \in \K_{\a, M}^+ $, we denote $||f||_{\a, M} := \sup _{r, k \leq M} ||  \frac{\partial ^k f}{\partial r^k }  (.,r)||_\a .$ 
The heart of the proof uses the arguments of \cite {D2} to establish:
\begin{prop}\label{prop:uniform2mixing} Let $\phi_0 \in \K_\a (\Si _+)$as above. There exist $\e, C, c>0$ and $M$  such that for all $\phi, \| \phi - \phi _0 \|_{\a} < \e$, all $f,g,h \in \K _{\a, M}^+,$ we have, for all positive $t_1,t_2$:
\begin{equation}\label{uniform2mixing2} \big| \rho_{f,g,h, m_{\vf }}(t_1,t_2) \big| \; \leq \; C \|f\|_{\a,M} \|g\|_{\a,M} \|h\|_{\a,M} [(1 + t_1)^{-c }+ (1 + t_2)^{-c } ]. \end{equation} \end{prop}

\begin{proof} Choose $\e$ so that Proposition~\ref{3.5} and Proposition~\ref{spectralgap} holds for all $\phi$ with $ \| \phi - \phi _0 \| < \e $. Fix $f,g,h, \phi $ and write $\rho (t_1,t_2) $ for $\rho_{f,g,h, m_{\varphi} }(t_1,t_2) .$  Assume first that {$\int f \, dm_\varphi =  \int h \, dm_\varphi = 0.$}
We consider the Laplace   transform
$$ \wh \rho (s_1, s_2 ) = \int _{\R_+ \x \R_+} \rho (t_1, t_2) e^{-s_1 t_1} e^{-s_2t_2} \, dt_1 \,dt_2  $$
which makes sense a priori for $a_j>0,$ where $s_j = a_j + i b_j, j = 1,2.$ The following computation is valid for $a_j >0 $ and will allow us to extend $ \wh \rho (s_1,s_2) $ analytically  to a larger domain and deduce the decay of $\rho (t_1,t_2) $ as $t_1,t_2$ go to infinity. 
\begin{lem} Consider the Laplace transforms $F, G $ and $H$ of the functions $f,g$ and $h$ given by:
$$ F(x, s) = \int _\R e^{-sr} f(x,r) dr , \; G(x,s) = \int _\R e^{-sr } g(x,r) dr , \; H(x,s) = \int _\R e^{-sr} h(x,r) dr .$$ Then, we have, for $a_1,a_2 >0$:
$$ \wh \rho (s_1,s_2) 
 =  \sum _{n,m}   \int _\Si H(x, s_2) \LL ^m_{\phi -s_2\tau  } \big[ G(., s_1-s_2 ) \LL ^n_{\phi -s_1\tau  } F (.,-s_1) (.) \big](x) \, d\nu _\phi (x) .$$ \end{lem}
\begin{proof}
We develop:
\begin{eqnarray*}
\wh \rho (s_1,s_2) & = &  \int _{\R_+ \x \R_+ } \int _{\Si ^\tau}  f(x,r) g(\s_{t_1}(x,r))  h (\s_{t_1 +t_2} (x,r )) e^{-s_1 t_1} e^{-s_2t_2} \, dm_\vf(x,r) \, dt_1 \,dt_2\\
  & =& \sum _{n,m}  \int _{\R_+ \x \R_+ \x \R_+} \int _\Si  f(x,r) g(\s^n x, r+t_1 -\tau ^n (x))  h (\s^{n+m}x, r+t_2 +t_1-\tau ^{n+m}(x) )\\
  & & \quad \quad \quad \quad       \quad \quad \quad \quad  \quad \quad \quad \quad   \quad \quad \quad \quad   e^{-s_1 t_1} e^{-s_2t_2}  \, dr \, d\nu _\phi (x)\, dt_1 \,dt_2,  \quad \quad \quad (*)
\end{eqnarray*}
where $\tau ^n (x) :=\sum _{k = 0}^{n-1} \tau (\s ^k (x))$. Observe that for all fixed positive $n,m$ the integral in $t_1, t_2, r$ is also an integral over $ \R \x \R \x \R .$   Then using the variables $w = r+ t_1 -\tau ^n(x)$ and $z = w +t_2 - \tau ^m (\s ^n x) $, the integral (*) can be written as 
{$$(*) =  \int _\Si  H(\s ^{n+m} x, s_2) e^{-s_2 \tau ^m (\s ^n x) }G(\s ^n x, s_1-s_2) e^{-s_1 \tau ^n (x) }  F(x, -s_1 )\, d\nu_\phi (x) .$$}
Using now  the invariance of $\nu _\phi $ under $\LL _{\phi } $ (\ref{transfer2}) and the fact that $ \LL^n (H K\circ \s^n )(x) = K(x) \LL (H) (x),$ we obtain:
 {\begin{eqnarray*}
(* )
  &=&   \int _\Si H(\s ^m x, s_2) e^{-s_2 \tau ^m ( x) } G(x, s_1-s_2 ) \LL ^n_{\phi -s_1\tau } F (\cdot,-s_1) (x) \, d\nu _\phi (x) \\
 &=&   \int _\Si H(x, s_2) \LL ^m_{\phi -s_2\tau  } \big[ G(\cdot, s_1-s_2 ) \LL ^n_{\phi -s_1\tau  } F (\cdot,-s_1) (\cdot) \big](x) \, d\nu _\phi (x) .
 \end{eqnarray*}}
 The Lemma follows for $a_j = \Re \, s_j >0.$ \end{proof}
 
  By Proposition~\ref{3.5} and our choice of $\e$, we conclude that there exist constants  $\d, \b, D_0$ such that, for all normalized $\phi$ with $||\phi-\phi_0||_\a<\e$,   { the mapping  $s \mapsto \sum _n \LL ^n_{\phi +s\tau } $ extends analytically} on the region {$U_{\d,\b}$}   and,  for {$s \in U_{\d,\b},$}
  \begin{equation}  \|\sum _n \LL ^n_{\phi +s \tau } \|_\a  \; \leq \; D_0 |b |^{D_0} . \end{equation} Moreover, by  Proposition~\ref{spectralgap}, there is $\d >0 $ such that  the series of operators $\sum _n \LL ^n_{\phi +s \tau } $ converges and is meromorphic on  the region $V_\d$, has 
 a simple pole at $0$ and has residue at 0 the projection on the constant function $\nu _\phi (.)$.

 On the other hand, since $f,g $ and $h$ belong to $\K_{\a, M}^+$ , the functions $s \mapsto F(\cdot,s), s\mapsto G (\cdot,s ) $ and $s\mapsto H(., s  ) $ are {holomorphic from $\C$} into $\K _\a (\Si _+) .$ Moreover, { for $s = a +ib$ and $|a| $ bounded,} the functions $\| F(\cdot, s )\|_\a, \| G(\cdot,s )\|_\a $ and $ \| H(\cdot, s)\|_\a $ decay at infinity as $(|b|)^{-M} $  
and 
  $$ \nu _\phi  (F(., 0)) \;=\; \int _{\Si _+} \left( \int _\R f(x,r) \,dr \right) \, d\nu_\phi (x) \;=\; \int _{\Si ^\tau } f \, dm_\vf \; = \; 0 .$$
  It follows that the function $$J(x,s) := \sum _n \LL ^n_{\phi + s\tau  } F(\cdot, s) (x)$$ is analytic from  {$U_{\d,\b} \cup V_\d $} into $\K_\a$ and that its $\K_\a $-norm is bounded by   $C \| f\|_{\a,M} (1 + |b |)^{D_0-M } $ as $
  |b | \to \infty $.
 Summarizing, {for each $b_2 \neq 0$,  the function $s_1 \mapsto  \wh \rho (s_1,ib_2) $ admits an analytic extension to $\{ (s_1, ib_2) ; s_1 \in U_{\d,\b} \cup V_\d \} $} and this extension satisfies: 
 {$$  \wh \rho (s_1,ib_2 ) \; = \; \sum _m \int _{\Si _+ } H(x, ib_2 ) \LL ^m_{\phi -ib_2 \tau   } [ G(\cdot, s_1 -ib_2 ) J(\cdot,-s_1 )] (x) \, d\nu _\phi (x)  .$$}
 As before, for each fixed $s_1  \in U_{\d,\b} \cup V_\d $,   the mapping $s_2\mapsto \sum _m \LL ^m_{\phi +s_2 \tau  } [ G(\cdot, s_2-s_1 ) J(\cdot,s_1  )] (x) $ is meromorphic  from $U_{\d,\b} \cup V_\d$ with a unique simple pole at $s_2 = 0 $ and a residue a constant function on $\Si _+ $ with value $C_0(s_1).$
 Therefore, for all $s_1 \in U_{\d,\b} \cup V_\d,$ $s_2 \mapsto \wh \rho (s_1,s_2 )$ admits a meromorphic extension to $U_{\d,\b} \cup V_\d$ of the form 
  $$  \wh \rho (s_1, s_2 ) \; = \; \frac {C_0(s_1) \int _{\Si _+ } H(x,0) \, d\nu _\phi (x) }{2\pi i s_2} + \ov \rho (s_1, s_2 ) ,$$ 
  where $\ov \rho (\xi , \eta ) $ is an analytic function on $(U_{\d,\b} \cup V_\d) \x  (U_{\d,\b} \cup V_\d) $  such that 
  {$$|\ov \rho (s_1,s_2 )| \; \leq \; C\| h\|_{\a,M}\| g\|_{\a,M} \| f\|_{\a,M} (1+ |b_2 |)^{-M }(1+ |b_1-b_2|)^{D_0-M }(1+ |b_1|)^{D_0-M }.$$}
  We again have $\int _{\Si _+ } H(x,0) \, d\nu _\phi  (x) = 0 $ by our condition that   $\int h \, d\mu_\phi = 0 $ and finally, the function $\wh \rho (s_1, s_2) $  admits  an analytic extension  to $(U_{\d,\b} \cup V_\d) \x  (U_{\d,\b} \cup V_\d) $ and satisfies:
  {$$|\wh \rho (s_1, s_2 )| \; \leq \; C \| h\|_{\a,M}\| g\|_{\a,M} \| f\|_{\a,M} (1+ |b_2 |)^{-M }(1+ |b_1-b_2|)^{D_0-M }(1+ |b_1|)^{D_0-M }.$$}
    
  We now compute $\rho (t_1,t_2 ) $ as the Laplace inverse of $\wh \rho (s_1,s_2)$  by integrating on the imaginary axis in $s_2$ and in $s_1$. 
    For a fixed $s_1 \in U_{\d,\b} \cup V_\d$, we can move the curve of integration in $s_2  $  to the curve $$ \G := \{ -\d \min \{ 1, \frac { 1} {|b |^\b }\} +i b  ; b \in \R \} .$$ We obtain that the function $ \wt \rho (s_1 , t_2 ) $ 
  \begin{eqnarray*} &  & \wt \rho (s_1, t_2) := \frac{-1}{4\pi ^2} \int _\R \wh \rho( s_1, ib_2) e^{ib_2 t _2} \, db_2 \\ &  = &\frac{-1}{4\pi ^2} \left(  \int _{-1}^{+1} \wh \rho (s_1, -\d +i b_2 ) e^{ib_2t_2} e^{-\d t_2 }\, db_2 + \int _{ \R \setminus [-1, 1] }\wh \rho ( s_1,  -  \d  \frac { 1} {|b_2 |^\b } +ib_2)e^{ib_2t_2} e^{-\d t_2/ |b_2|^\b }\, db_2 \right)\end{eqnarray*}
  is, as a function of $s_1$, an analytic function on $U_{\d,\b} \cup V_\d$ and satisfies 
    \begin{eqnarray*} | \wt \rho (s_1 , t_2 )| & \leq & C\frac{ \| h\|_{\a,M}\| g\|_{\a,M} \| f\|_{\a,M}}{(1+ |b_1|)^{M- D_0 }} \left( 2 e^{-\d t_2} + \int _{ \R \setminus [-1, 1] } \frac{  e^{-\d t_2 / |b|^\b} }{(1 + |b|)^{M- D_0}}\, db \right) \\
     &\leq & C \frac { \| h\|_{\a,M}\| g\|_{\a,M} \| f\|_{\a,M}}{(1+ |b_1|)^{M-D_0 }} (1 + t_2) ^{- \b }, \end{eqnarray*} 
    as soon as $M > D_0 + 2.$ 
We are interested in  $\rho (t_1,t_2) =  \int _\R \wt \rho (s_1 , t_2) e^{i b_1 t_1 } \, db_1$. In the same way, by moving the curve of integration in $s_1 $ to $\G $, we obtain (recall that we have assumed that {$\int f \, dm_\l = \int h \, dm_\l = 0$}):
\[ \rho (t_1, t_2 ) \; \leq \;  C  \| h\|_{\a,M}\| g\|_{\a,M} \| f\|_{\a,M} [ (1 + t_1)^{-\b } + (1 + t_2)^{-\b }]. \]

\

Observe that the above proof also yields, setting $g = 1$:
\begin{prop}\label{prop:uniformmixing} Let $\phi _0 \in \K_\a (\Si _+)$as above. For $\e, C, c>0$ and $M$  as above,  for all normalized $\phi$ with $||\phi-\phi_0||_\a<\e$, , all $f,h \in \K _{\a, M}^+,$  we have, for all positive $t$,
\begin{equation}\label{uniformmixing2} \big| \rho_{f,h, m_{\vf} }(t) \big| \; \leq \; C \|f\|_{\a,M}  \|h\|_{\a,M} [(1 + t)^{-c}]. \end{equation} \end{prop}

Indeed, if we assume $\int f \,dm_\vf  = 0 $, this is exactly the same computation, with only one variable $s$. But (\ref{uniformmixing2}) holds for $f$ as soon as it holds for $f -\int f \,dm_\vf.$ By the same token, using Proposition~\ref{prop:uniformmixing}, we can replace in (\ref{uniform2mixing2}) $f $ and $h$ by $f -\int f \,dm_\vf$
and $h -\int h \,dm_\vf.$ This achieves the proof of Proposition~\ref{prop:uniform2mixing}.
   \end{proof}

\subsubsection{From one-sided to two-sided smooth functions}\label{sec:7.2.3}
This part goes back to Ruelle (\cite {R}), we present it here for completeness. We consider a new space of functions:
for $\a >0 $ and $M \in \N$, let $\K'_{\a, M} $ be the set of functions $f$ on $\Si ^\tau $ with the following properties:
\begin{itemize}
\item for all $x \in \Si ,$ $f(x,r) = 0 $ for $r$ outside the interval $[ \frac{\inf \tau }{3},  \frac{2\inf \tau }{3}]$, 
\item for all $x \in \Si ,$ $r \mapsto f(x,r)$ is of class $C^M$ and
\item the functions $ \frac{\partial ^k f}{\partial r^k }  (x,r)$, for  $0\leq k \leq M $ are $\a$-H\"older continuous on $\Si $ and continuous in $r$.
\end{itemize}
For $f \in \K'_{\a, M} $, we still denote $\| f\|_{\a, M} := \sup _{r, k \leq M } \|  \frac{\partial ^k f}{\partial r^k }  (.,r) \| _\a .$ We show in this subsection \begin{prop}\label{twosided} There exist $\e', C', c'>0$ and $M$  such that for all normalized $\phi$ with $||\phi-\phi_0||_\a<\e'$,  all $f,g,h \in \K' _{\a, M},$ we have, for all positive $t_1,t_2$:
\begin{equation*}\big|\rho_{f,g,h, m_{\vf} }(t_1,t_2) \big|  \; \leq \; C' \|f\|_{\a,M} \|g\|_{\a,M} \|h\|_{\a,M} [(1 + t_1)^{-c' }+ (1 + t_2)^{-c' } ]. \end{equation*} \end{prop}
\begin{proof}

Assume first  that $\int f \, dm_\vf = \int g \, dm_\vf = \int h \, dm_\vf = 0$.

The following construction reduces the proof of Proposition \ref{twosided} to a direct extension of the  proof of Proposition \ref{prop:uniform2mixing}. Let $A(x) $ be a function in {$\K_{\a} ( \Si ) $ }; then (see e.g. \cite{P1}), there exists a decomposition $ A = \sum _{j=0 }^\infty  A_j $, where
\begin{enumerate} 
\item $x \mapsto A_j (x) $ depends only on the coordinates $( x_{-j}, x_{-j+1}, \cdots )$ of $x$,
\item $\sup _x |A_j (x) | \leq \a^j \| A\|_\a $ and
\item $ \| A_j \|_{\a} \leq \| A\|_\a $. 
\end{enumerate}
{Now assume that $s \mapsto A(x,s) $ is {holomorphic from $\C$} into $\K _\a (\Si _+)$ and that { for $s = a +ib$ and $|a| $ bounded,} the function $\| A(\cdot, s )\|_\a $ decays at infinity as $(|b|)^{-M} .$  The same construction yields a holomorphic family $s \mapsto A_j(x,s)$ with properties (1),(2) and (3) true for all $s$.}\footnote{The mapping $A \mapsto A_j$ can be chosen  linear from $\K_\a$ to $\K_\a$ and therefore $s \mapsto A_j(x,s) $ is holomorphic from $\C$ into $\K _\a (\Si _+)$. See  \cite {R}, page 110.}
{We define the functions $\wt A_j  (x,s) := e^{- s\tau ^j ( x) } A_j (\s^{j} x,s) $.} Then, by \cite {R} (see also \cite{D1} and \cite{P1}), there is $\a', 0 <\a' <\a , $ and $ \theta , 0< \theta < 1,$ such that, for all $s$ with {$s = a+ ib, |b|>1$ \begin{enumerate} 
\item $x \mapsto \wt A_j (x,s)) $ depends only on the coordinates $( x_{0}, x_{1}, \cdots )$ of $x$,
\item $\sup _x |\wt A_j (x,s) | \leq e^{Cj|a|} \a^j \| A(.,s)\|_\a $ and
\item $ \| \wt A_j(.,s) \|_{\a'} \leq   C e^{Cj|a|}    | b| \theta ^j  \| A(.,s)\|_\a  $.
\end{enumerate}} 
Finally, we set {$\wt A(x,s)  := \sum _j \wt A_j (x,s) ;$} we have, if $|a|$ is small enough, 
 \begin{enumerate} 
\item $x \mapsto \wt A (x,s) $ depends only on the coordinates $( x_{0}, x_{1}, \cdots )$ of $x$,
\item $\sup _x |\wt A (x,s) | \leq C \| A\|_\a $,
\item $ \| \wt A(.,s)  \|_{\a'} \leq   C    | b |  \| A(.,s)\|_\a  $   {for $|b|>1$} and
\item $\int \wt A (x,0) \, d\ov \nu (x) = \int A(x,0) \, d\ov \nu (x) $ for any shift invariant measure $\ov \nu $ on $\Si $.
\end{enumerate}
{In particular, by property (3), for $|a|$ small enough,  the function $\| \wt A_j(\cdot, s )\|_{\a'} $ decays at infinity like $(|b|)^{-M+1} .$}
Property (4) is clear since {$ \wt A (x,0) = \sum _j \wt A_j (x,0) = \sum _j  A_j (\s ^j x,0)$}, whereas $A(x,0) = \sum _j  A_j (x,0)$ and both series of functions converge uniformly. %

Choose $\e'$ so that for all normalized $\phi$ with $||\phi-\phi_0||_\a<\e'$, Proposition~\ref{spectralgap} and Proposition~\ref{3.5} apply on $\K_{\a'}$. Fix $f,g,h   \in \K'_{\a, M}$ and write $\rho (t_1,t_2) $ for $\rho_{f,g,h, m_{\vf} }(t_1,t_2)$. We  now write as before the Laplace transform $ \wh \rho (s_1,s_2)$ of $\rho (t_1,t_2)$ as:
{$$ \wh \rho (s_1,s_2) 
 =  \sum _{n,m}   \int _\Si  H(\s ^{n+m} x, s_2) e^{-s_2 \tau ^{m+n} (x) }G(\s ^n x, s_1-s_2) e^{(s_2-s_1) \tau ^n (x) }  F(x, -s_1 )\, d\ov \nu_\phi (x)  ,$$ 
 where, as before, the functions $H(x,s), G(x,s) $ and $F(x,s)$ are the Laplace transforms of the functions $f,g$ and $h.$ The functions $H(x,s), G(x,s) $ and $F(x,s)$ satisfy all the above assumptions and we can associate the functions $\wt H(x,s), \wt G(x,s) $ and $\wt F(x,s)$ such that their $\| \|_{\a'} $ norms in $x$  decay at infinity as $(|b|)^{-M+1} .$}

We consider this sum  as a series in the sense of tempered distributions: for any $B (s,t) $ in the Schwartz space of $\R^2$, $\int \wh B(ib_1, ib_2) \wh \rho (ib_1, ib_2 )\, db_1  db_2 $ makes sense and is equal to $-4\pi^2 \int B(t_1,t_2) \rho (t_1, t_2 ) \, dt_1 dt_2.$ The series of integrals $\int B(t_1,t_2) \rho _{n,m} (t_1,t_2) \, dt_1dt_2 $ converges absolutely. It still does if one considers the sum over $n,m $ in $ \Z $ instead of $\Z_+$.
For each $(n,m) \in \Z\x\Z$, we write, using the decompositions $H (x, s) =  \sum _j H_j (x, s),$ $  G (x, s ) = \sum _k G_k (x, s) , F(x, s )= \sum _\ell F_\ell (x, s) $ and the above $\wt A _j$ notation: 
{\begin{eqnarray*}
& & \wh \rho_{n,m}   (s_1,s_2) =\\
&:= & \int _\Si H(\s ^{n+m} x, s_2) e^{-s_2 \tau ^{m+n} (x) }G(\s ^n x, s_1-s_2) e^{(s_2-s_1) \tau ^n (x) }  F(x, -s_1 )\, d\ov \nu_\phi (x)  \\
 & = &  \sum _{j,k,\ell}   \int _\Si H_j(\s ^{n+m} x, s_2) e^{-s_2 \tau ^{m+n} (x) }G_k(\s ^n x, s_1-s_2) e^{(s_2-s_1) \tau ^n (x) }  F_\ell (x,- s_1 )\, d\ov \nu_\phi (x)  \\
&  =  & \sum _{j,k,\ell}   \int _\Si  \wt H_j^{s_2}(\s ^{n+m-j} x, s_2) e^{-s_2 \tau ^{m+n-j} (x) } \wt G_k^{(s_1-s_2)}(\s ^{n-k} x, s_1-s_2) e^{(s_2-s_1) \tau ^{n-k} (x) }  \\
&    &\quad \quad \quad    \quad \quad \quad   \quad \quad \quad   \quad \quad \quad   \quad \quad \quad   \quad \quad  \quad   \quad \quad   \wt F^{-s_1}_\ell(\s ^{-\ell }x, -s_1 )e^{-s_1 \tau^\ell (\s ^{-\ell} x)}\, d\ov \nu_\phi (x) \\
& =   & \sum _{j,k,\ell}   \int _\Si  \wt H_j(\s ^{n+m-j} x, s_2) e^{-s_2 \tau ^{m+k-j} (\s^{n-k}x) } \wt G_k(\s ^{n-k} x, s_1-s_2) e^{-s_1 \tau ^{n-k+\ell} (\s ^{-\ell }x) }  \\
&   & \quad \quad \quad    \quad \quad \quad   \quad \quad \quad   \quad \quad \quad   \quad \quad \quad   \quad \quad \quad   \quad \quad   \quad \quad \quad   \quad \quad   \wt F_\ell(\s^{-\ell}x, -s_1 ) \, d\ov \nu_\phi (x),
\end{eqnarray*}}
where we used the cocycle relation $ \tau ^{n+m } (x) = \tau ^n (x) + \tau ^m (\s ^n x) $ valid for all $m,n \in \Z $.

We now replace the summation in $(n,m) $ by a summation in $(p,q) $, where $p : =n-k +\ell , q := m+k -j $. 
Assume for example $p \geq 0,  q \geq 0 $ (and then $ p+q = n+m -j +l  \geq 0 $). We write, using the invariance of $ \ov \nu _\phi$, the integral
 { \begin{equation}\label{nmjkl}    \int _\Si  \wt H_j(\s ^{n+m-j} x, s_2) e^{-s_2 \tau ^{m+k-j} (\s^{n-k}x) } \wt G_k(\s ^{n-k} x, s_1-s_2) e^{-s_1 \tau ^{n-k+\ell} (\s ^{-\ell }x) } 
  \wt F_\ell(\s^{-\ell}x, -s_1 ) \, d\ov \nu_\phi (x),\end{equation}
 as: 
  \[  \int _\Si   \wt H_j(\s ^{n+m-j +\ell } x, s_2) e^{-s_2 \tau ^{m+k-j} (\s^{n-k+\ell}x) } \wt G_k(\s ^{n-k+\ell} x, s_1-s_2) e^{-s_1 \tau ^{n-k+\ell} (\s ^{-\ell }x) }   \wt F_\ell(x, -s_1 ) \, d\nu_\phi (x),\]}
  where we replaced $\ov \nu _\phi $ by $\nu _\phi $ since the integrand now depends only on the non-negative  coordinates of $x$.
  As before, we can write these integrals using the transfer operators as
{ \begin{eqnarray*} & \int _\Si  \wt H_j(\s ^{m+k -j} x, s_2) e^{-s_2 \tau ^{m+k -j  } ( x) } \wt G_k ( x, s_1 -s_2) \LL _{\phi -s_1\tau}^{n-k+\ell }
 ( \wt F_\ell ( ., -s_1)) (x) \, d\nu_\phi (x) \\
 = & \int_\Si \wt H_j (x, s_2)  \LL _{\phi -s_2 \tau}^{q } [\wt G_k ( ., s_1 -s_2) \LL _{\phi -s_1 \tau}^{p}
 ( \wt F_\ell ( ., -s_1)) (.)] (x) \, d\nu _\phi (x).
 \end{eqnarray*}}
If $|a_1|, |a_2 | , $ and $ |a_1 -a_2| $ are  small enough, one can  sum  in $j,k,\ell \in \Z_+^3$ the integral (\ref{nmjkl}) for the same value of $(p,q)$;  we obtain, when $p,q \geq 0 $, 
 {\[ \int_\Si \wt H(x, s_2)  \LL _{\phi -s_2 \tau}^{q } [\wt G ( ., s_1-s_2) \LL _{\phi -s_1 \tau }^{p}
 ( \wt F ( ., -s_1)) (.)] (x) \, d\nu _\phi (x). \]}
The other possible signs of $p , q$ and $p+q$ are treated in the same way. 

 By applying   Proposition \ref{3.5} to $\K_{\a'} $, we conclude that there are positive numbers $ \d', \b' , D'_0$ such that, for all normalized $\phi$ with $|| \phi - \phi_0||_\a < \e',$  the series of operators $\sum _n \LL ^n_{\phi + s\tau } $ {has an analytic extension to  the region $U' = U_{\d',\b'}$  and  for $s \in U',$ }
  \begin{equation}  \|\sum _n \LL ^n_{\phi + s \tau} \| \; \leq \; D'_0 |b |^{D'_0} . \end{equation}
 Moreover,  there is $\d' >0 $ such that on the series of operators $\sum _n \LL ^n_{\phi + s \tau  } $ converges and is meromorphic on  the region $V'= V_{\d'}$, with a simple pole at $0$ and residue the projection on the constant function $\nu _\phi (.)$.   
We conclude as above (but with a different argument for each one of the six sums over $(p,q), (-q, p+q), (-p, p+q), (-p-q,q), (p, -p-q), (-p,-q) $ in $(\Z_+ \x \Z_+)$) that $\wh \rho (s_1, s_2)$ is given by an analytic function defined on the region where $s_1, s_2$ and $s_1 - s_2$ all belong to $U' \cup V'$ (and have a real  part smaller than $\d_0$) and satisfying 
$$|\wh \rho (s_1 ,s_2)| \; \leq \; C \| h\|_{\a,M}\| g\|_{\a,M} \| f\|_{\a,M} (1+ |b_1 |)^{D''_0-M }(1+ |b_1 - b_2 |)^{D''_0-M }(1 + |b_2 |)^{D''_0-M },$$
where $D''_0 = D'_0 + 1$.

If $M$ has been chosen  greater than $D''_0 +2$, we obtain  Proposition \ref{twosided} (for functions with integral 0)  by the same argument as  before, provided one chooses in each of the six cases contours $\G$ of integration with the right sign.

The extension of Proposition~\ref{prop:uniformmixing} to functions $f,h \in \K'_{\a, M}$ with $\int f dm_\phi = 0 , \int h dm_\phi = 0 $ goes again by the same computation, without the function $g$. Again, (\ref{uniformmixing2}) holds for $f$ as soon as it holds for $f -\int f \,dm_\vf.$ This justifies the reduction to functions with integral 0 in the proof of proposition~\ref{twosided}. \end{proof}

\subsubsection{H\"older continuous functions} 

We conclude the proof of Theorem~\ref{theo:uniform2mixing} and of Proposition~\ref{theo:uniformmixing} by approximating any H\"older continuous function by regular functions. We have proven  (\ref{uniform2mixing}) for functions in $\K'_{\a, M} $ with some constants $C', c'$;  (\ref{uniform2mixing}) holds also if $f,g, h$ are such that $f \circ \s_{t_1}, g \circ \s _{t_2}, h \circ \s_{t_3}  \in \K' _{\a,M} $ for  bounded $t_i, i = 1,2, 3$. There is $C_9 = 10  + 6\frac{\sup _x\tau (x)}{\inf _x \tau(x)} $ such that any function which is of class $C^M$ along the trajectories of the special flow $(\Si ^\tau, \s_t, t \in \R )$ and such that  the first $M$ derivatives along the flow are $\a$-H\"older continuous functions can be written as a sum of less than $C_9$ functions in $\K'_{\a,M}$. Using the projection from the manifold to $\Si ^\tau $, we conclude that there exist $\e, C'', c'>0,  \a, \a_0, M $ such that for all $\vf, \| \vf - \vf _0 \|_{\a_0} < \e $, all $f,g,h $ that are  of class $C^M$ along the trajectories of the  flow and  such that all  the derivatives along the flow up to order $M$ belongs to $\K_\a(SM) $,  we have, for all $t_1,t_2\geq 0 $:
\begin{equation*}\ \big| \rho_{f,g,h, m_{\varphi} }(t_1,t_2) \big|\; \leq \; C'' \|f\|_{\a ,M} \|g\|_{\a , M} \|h\|_{\a , M}[(1 + t_1)^{-c' } +(1 + t_2)^{-c'} ], \end{equation*}
where $\| . \|_{\a, M} $ is the maximum of the $\| \|_\a $ norms of the first $M $ derivatives along the flow.

We conclude by smoothing all  functions in $\K_\a$. Let $\ov \psi $ be a $C^M $ nonnegative  function on $\R$, with support in $[-1, +1] $ and integral 1. For $\e>0$ and a function $f \in \K_\a$, set $$\ov \psi _\e (t): = \frac{1}{\e} \ov \psi (\frac {t}{\e }) \; \textrm {and } f_\e (x) := \int _\R \ov \psi _\e (t) f (\vf _t x) \,dt .$$ 
We have $ \sup _x |f(x)  - f_\e(x)  | \leq \e ^\a \| f \|_\a $ and $\| f_\e \|_{\a, M } \leq \e^{ -M- 1} \| f \|_\a $.

Fix $t_1,t_2>0$, choose $\displaystyle \e =  [ 1/3 (1 + t_1)^{-c' } + 1/3(1 + _2)^{-c'} ]^{\frac {1}{\a +3M +3}} $ and replace $f,g,h$ by $f_\e , g_\e, h_\e .$ One obtains  (\ref{uniform2mixing}) for $f,g,h$ with some  constant $C'_0 $ and $c'_0 = \frac {c'\a}{\a +3M +3}.$

\section{Appendix II: Potential theory on $\M$}\label{Appendix2} In this section, we recall the potential  theory that we used. Some justifications are more transparent when using the probabilistic approach.
\subsection{General theory}

 Let $\M$ be a simply connected nonpositively  curved Hadamard
 manifold with Ricci curvature bounded from below. Then the manifold is {\it{stochastically complete}} (\cite{Pi}, \cite{Y2}) and the heat kernel $\Pp (t,x, y)$ satisfies, for all $x,z \in \M, s,t >0$
 \begin{equation}\label{complete} \int _{\M} \Pp (t,x, y) \, d\Vol (y) \; = \; 1, \quad {\textrm {and}} \quad \Pp (t+s, x, z) = \int _{\M} \Pp (t,x, y)\Pp (s, y,z)  \, d\Vol (y) . \end{equation}
 
 The following  results of Sullivan \cite{Su} hold more generally for open connected Riemannian manifold without boundary.
\begin{defi}\label{lambda0}
The bottom of the spectrum $\l_0$ is defined to be
$$ \l_0 =  \inf \frac{\int_{\M} |\nabla \phi|^2}{\int_{\M} |\phi|^2},$$ {where the infimum is taken} 
over smooth functions $\phi$ on $\M$ with compact support.
\end{defi}
Indeed, the $L^2$ spectrum of the operator {$\D$} is a subset of $[\l_0, +\infty)$ that contains $\l_0$ (\cite{Su}). Moreover,  the same $\l_0$ is related to 
smooth positive eigenfunctions of $\D$. 
\begin{lem}\label{lem2.2} {With $\l_0$ as in the definition \ref{lambda0}, }
\begin{enumerate}
\item For each $\l \leq \l_0$, there is a smooth positive $\l$-harmonic function $\phi$.
 For each $\l > \l_0$, there are no smooth positive $\l$-harmonic functions.
\item {If for some $x\neq y$,} $\int_0 ^\infty e^{\l_0 t} \Pp(t,x,y) dt =\infty,$ then there is a unique positive $\l_0$-harmonic function $\phi_0$ up to multiplicative constants.
\item {If for some $x \neq y$,} $\int_0 ^\infty e^{\l_0 t} \Pp(t,x,y) dt =\infty,$ the Markov process on $\M$ associated with the semi-group of probability densities 
\begin{equation}\label{eqn2.1} q(t,x, y) :=  \Pp(t,x,y) \frac{\phi_0(y)}{\phi_0(x)} e^{\l_0 t}\end{equation} 
is recurrent, i.e. almost every path starting from any point in $\M$ enters every set of positive measure infinitely often.
\end{enumerate}
\end{lem}
\begin{proof}
Part (1) is Theorem 2.1 of \cite{Su}
Part (2) and (3) are Theorem 2.7 and Theorem 2.10 of \cite{Su}, respectively.
\end{proof}
We recall the  Harnack inequality and its consequence.

\begin{prop}[Harnack inequality \cite{L}, Theorem 6.1]\label{Harnack}  
 There is a $C_0 >1 $ such that for all $\l \in [0, \l_0]$, for any positive {$\l $-harmonic} function $f$  on an open  domain $\DD $,  we have $\| \nabla \log f\|(x)  \leq \log C_0$ if $d(x, \pp \DD ) > 1 $.
\end{prop}

{We also recall a consequence of the parabolic Harnack inequality in the case when the Ricci curvature is bounded from below by some constant $-a_7^2.$
\begin{prop}\label{boundedgradient} There are $C, T_1$ such that, for all  $x,y$ in a compact set $A\subset \M$, $t\geq T_1$, \[  \| \nabla \log \Pp (t,x,y) \| \leq C .\] \end{prop}
\begin{proof} Choose $R$ large enough that $A\subset B(x,R/2).$
The function $\Pp (t, x, y)$ is a solution of the heat equation on $\M$ with Ricci curvature bounded below, by $-a_7^2$,  then by a sharp gradient estimate by Souplet and Zhang \cite{SZ}, on $\{ (y,t) : y \in B(x, R/2), s\in [t_0 -T/2, t_0] \}$,
$$ \frac{|\nabla_y \Pp (t, x, y)|}{\Pp (t, x, y)} \leq C \left( \frac{1}{R} + \frac{1}{\sqrt{T}} + a_7 \right) \left( 1 + \log \frac{\max \Pp (t, x, y)}{\min \Pp (t, x, y) }\right),$$
where the maximum and minimum are taken on the set $\{ (y,t) : y \in B(x,R), t\in [t_0 - T, t_0]\}$.\\
We need to show that $\displaystyle  \frac{\max \Pp (t, x, y)}{\min \Pp (t, x, y)}$ is bounded uniformly for $t$ large. 
Assume not. \\ Then there exist $y_n, y'_n \in B(x,R), t_n \to \infty, T_n, T'_n \in [0,T]$ such that $\displaystyle  \frac{\Pp (t_n - T_n, x, y_n)}{\Pp (t_n - T'_n, x, y'_n)} \to \infty.$ We can assume, by taking a subsequence, that $y_n \to y, y'_n \to y', T_n \to T_\infty , T'_n \to T' _\infty $ and that there exist $\l_0 $ harmonic functions $\psi, \psi'$ on $B(x,R)$ such that 
\[  \frac{\Pp (t_n - T_n, x, y_n)}{ \Pp (t_n - 2T, x, x)} \to e^{\l_0 ( T_\infty  -2T) }\psi (x,y)  \;\; {\textrm{and}} \;\;  \frac{\Pp (t_n - T'_n, x, y'_n)}{ \Pp (t_n - 2T, x, x)} \to e^{\l_0 ( T'_\infty  -2T)} \psi '(x,y') .\]
(See e.g. \cite{ABJ}, Theorem 2.2). The function $\psi '$ is a $\l_0$-harmonic function that is not identically $0$. Indeed, by \cite{ABJ}, Lemma 2.1, \[ \psi'(x,x)= e^{-\l_0 ( T'_\infty  -2T)}\lim\limits_{t \to \infty }  \frac{\Pp (t - T'_\infty , x, x)}{ \Pp (t - 2T, x, x)}= 1.\] 
  So it does not vanish, and the above limit cannot be $+\infty .$
\end{proof}}

We assume in the rest of this section that the Green function $G_{\l _0} (x,y)= \int_0^\infty e^{\l_0 t} \Pp (t,x,y)\, dt $ is finite.

\subsection{Relative Green function}

 A {\it {path }} in $\M$ is a continuous mapping $\om = \om _t, t\geq 0,$ from $[0, + \infty ) $ to $\M$. The space $\Om $ of paths is endowed with the compact open topology and the corresponding Borel $\s $-algebra. It follows from (\ref{complete}) that for each $x \in \M$, there is a probability measure $\P_x$ on $\Om $ such that $\om_0 = x \; \P_x$-a.e., $\{ \om _t, t \geq 0 \} ,$ is a Markov process and for all  Borel subsets $A$ of $\M$, all $t > 0,$ 
 $$ \P_x (\{ \om , \om _t \in A \}) \; = \; \int _{A} \Pp (t,x, y) \, d\Vol (y) .$$
 The probability $\P_x$ is called the Wiener measure starting from $x$ and the corresponding expectation integral is denoted by $\E_x.$  
 
   Let $A$ be a closed subset of $\M $ and assume $x \not \in A$. For $\om \in \Om_x$, let $T_A(\om)   \in ]0, + \infty ]$  be the first time the trajectory $\om $ hits $A$.  For $\l \leq \l_0, $ the {\it {relative Green}} function $G_\l (x,y : \M \setminus A )$ is the positive function such that,
for every nonnegative measurable function $F$,
\begin{equation}\label{stoppingtime}   \int_{\M \setminus A} F(y) G_\l (x,y :\M \setminus A ) dy \; = \;  \E _x \left[\int _0^{T_A(\om)} e^{\l t} F(\om_t) \,  dt \right] . \end{equation}
{For all open sets $\DD$ and $\CC \subset \DD$  all $0 \leq \l \leq \l_0, $ and all $x \neq y  \in  \CC,$ we have $$ G_\l (x,y :\CC ) \leq G_\l(x,y :\DD) \leq G_{\l_0} (x,y :\DD ) \leq G_{\l_0}(x,y) < + \infty .$$
 \begin{cor}\label{coro:Harnack}  There is a constant $C_0$ such that for any open set $\DD$, 
 any $0 \leq \l \leq \l_0 $ and any $x, y,z \in \DD$ such that $d(x,z), d(x,y), d(x,\pp\DD), d(y,\pp\DD), d(z,\pp\DD) $ are all at least 1, we have
 \[ G_\l (x,z:\DD) G_\l (x,y:\DD) \leq C_0 \max \{ G_\l (x,y):\DD);  d(x,y )\geq 1\}   G_\l (z,y:\DD) .\]
 \end{cor}
(See {\it{Remarque}} on page 94 of \cite{An2} for a proof of Corollary~\ref{coro:Harnack}.)} 

Consider $A$ a closed $(n-1)$-dimensional submanifold in 
 {$\DD$} and assume $x,z \in \DD .$  Write $T(\om )$ for  $T_{A\cup \pp \DD} (\om )  .$ Observe that  if  {$T (\om ) < T_{\pp \DD}(\om ), \; \om _{T (\om)} \in A\subset \DD .$} In particular, in that case, $ G_\l (\om _{T (\om ) }, z: \DD)$ makes sense.  
\begin{prop}\label{Markov} {With the above notations, we have,   for all $\l \leq \l_0$, all $x,z \in \DD \setminus A,$} $$ G_\l (x,z: \DD ) \; = \; \E_x \left[1_{T (\om ) < T_{\pp \DD}(\om ) }  e^{\l T (\om)} G_\l (\om _{T(\om)}, z: \DD)\right] + G_\l (x, z: \DD \setminus A) . $$ \end{prop} 
\begin{proof} We may assume that $x \neq z$. Then we may write for $\d < d(z, A\cup \pp \DD)/2$, and $d < d(z,x)/2$, 
\begin{eqnarray*} 
& &\int _{B(z, \d)} G_\l (x, w: \DD ) \, dw \\ &=&  \E _x \left[\int _0^{T_{\pp \DD}(\om)} e^{\l t} 1_{B(z,\d)} (\om _t) \, dt \right]\\
&=&  \E_x \left[ 1_{T(\om)< T_{\pp \DD} (\om) } \int _{T (\om) }^{T_{\pp \DD} (\om) } e^{\l t} 1_{B(z, \d)} (\om_t) \, dt \right]   +   \E_x  \left[ 1_{T(\om)< T_{\pp \DD} (\om)}  \int _0 ^{T (w) } e^{\l t} 1_{B(z, \d)}(\om _t) \, dt \right]\\
&&+  \E_x  \left[ 1_{T(\om) \geq T_{\pp \DD} (\om) } \int _0 ^{T_{\pp \DD}(\om) } e^{\l t} 1_{B(z, \d)} (\om_t) \, dt \right]\\
&=& \E_x \left[1_{T(\om )  < T_{\pp \DD}(\om ) }e^{\l T (\om ) }  \int _{B(z, \d)} G_\l (\om _{T(\om)}, w: \DD) \, dw \right] + \int _{B(z, \d)} G_\l (x, y: \DD \setminus A) d\mathrm{Vol}y . 
\end{eqnarray*} 
We used the Strong Markov Property of the stopping time $T(\om)$ to write the last line.\footnote{To justify the convergence as $\d \to 0$, we have to use $G_\l (\om _{T(\om)}, w: \DD) \leq C G_\l (\om _{T(\om)}, z: \DD) $ and $G_\l (x, y: \DD \setminus A) \leq C G_\l (x, z: \DD \setminus A) $ as soon as $\d < 1/2 d(z, A\cup \pp \DD)$ and $\d <d(z,x)/2$, which follows from Proposition \ref{Harnack} applied to a constant multiple of the metric.}
The proposition follows by letting $\d  \to 0.$
\end{proof} 
Let $\varpi _x^\l $ be the distribution on $A\cap \DD $ such that the proposition writes, for all $ \l \leq \l_0,$
\begin{equation*}
G_\l (x,z : \DD) \; = \; \int _{A\cap \DD} G_\l (y,z:\DD) \, d\varpi _x^\l (y) + G_\l (x, z: \DD \setminus A) 
\end{equation*}  
The measure $\varpi_x^0 $ is the distribution of the hitting point $\om _{T(\om ) } $ on $A \cap \DD$ and, for $F$ positive measurable function on $A$, {\begin{equation}\label{hitting}\int _A F(y) \, d\varpi_x^\l (y) = \E_x [ 1_{T(\om)< T_{\pp\DD} (\om) } e^{\l T (\om )} F(\om _{T(\om ) })] .\end{equation} }
\begin{cor}\label{lem:3.2}
Let $A $ be a closed $(m-1)$-dimensional submanifold of the open $\DD$, and $x \in \DD \setminus A.$  {For all $\l \leq \l_0$, all $x,z \in \DD \setminus A,$}, there is a measure $\varpi _x^\l $  on $A  $ such that:
\begin{equation}\label{Poisson} 
G_\l (x,z : \DD) \; = \; \int _{A} G_\l (y,z:\DD) \, d\varpi  _x^\l (y)  + G_\l (x,z: \DD \setminus A ) .
\end{equation} 
\end{cor} 
 {\begin{defi} A barrier $A$ is a closed $(m-1)$-dimensional manifold  that separates $\DD $ into two disjoint connected components. \end{defi}}
 Clearly, if $A$ is a barrier, and $x,z$ are in distinct connected components of $\DD \setminus A$, then all paths going from $x$ to $z$ hit the barrier $A$. Relation (\ref{Poisson}) becomes
 \begin{equation}\label{Poissonbarrier} 
G_\l (x,z : \DD) \; = \; \int _{A} G_\l (y,z:\DD) \, d\varpi  _x^\l (y)  .
\end{equation} 

Assume now that we have disjoint barriers $A_1, A_2 $ in $\DD$. Denote $\CC_i, i = 1,2,3$ the connected components of $\DD \setminus (A_1\cup A_2)$ in such a way that $A_1 $ separates $\CC _1 $ from $\CC _2 $ and that $A_2 $ separates $\CC _2 $ from $\CC _3 $. 
\begin{prop}\label{barrier}  With the above notations, for all  $x \in \CC _1, 0 \leq \l \leq \l_0 , $ the measures  $\varpi _{x,A_1}^\l $, $\varpi _{x,A_2}^\l $ satisfy, for any positive measurable function $F$ on $A_2$, 
\[ \int _{A_2} F(a_2) \, d\varpi _{x,A_2}^\l (a_2) \; = \int _{A_1} \left(  \int _{A_2} F(a_2) \, d\varpi _{a_1,A_2}^\l (a_2)  \right)\, d\varpi _{x,A_1}^\l (a_1)  .\] \end{prop}

\begin{proof} Any path $\om $ starting from $x \in \CC_1$ hits $A_1$ before hitting  $A_2$. Set $T_i (\om ): = T_{A_i}(\om ), i = 1,2.$ Unless $T_{1} (\om ) = T_{2} (\om ) = +\infty ,$ we have $T_{1} (\om ) < T_{2} (\om ).$  Then, we may write:
\begin{eqnarray*} \int _{A_2} F(a_2) \, d\varpi _{x,A_2}^\l (a_2) \, da_2  &=& \E_x \left[ 1_{T_2 (\om ) <\infty } e^{\l T_2 (\om )} F(\om_{T_2(\om)})\right] \\
&=&  \E_x \left[ 1_{T_1 (\om ) <\infty }1_{T_2 (\om ) <\infty } e^{\l T_1 (\om )}e ^{\l (T_2-T_1) (\om )} F(\om_{T_2(\om)})\right] \\
&=& \E_x \left[ 1_{T_1 (\om ) <\infty }e^{\l T_1 (\om )}\E_{\om _{T_1 (\om )}} [ 1_{T_2 (\om ') <\infty }e^{\l T_2 (\om' )} F(\om_{T_2(\om')})] \right] , \end{eqnarray*}
where we used the strong Markov property and $\om'$ is the path $\om'_t = \om _{t + T_1(\om )}$. We obtain
\[ \int _{A_2}  F(a_2) \, d\varpi _{x,A_2}^\l (a_2) \, da_2  \; = \; \E_x \left[ 1_{T_1 (\om ) <\infty } e^{\l T_1 (\om )} \int _{A_2} F(a_2) \, d\varpi _{\om_{T_1(\om)} ,A_2}^\l (a_2)  \right]. \]
The relation follows. \end{proof}

Assume furthermore  that a barrier $A$ is the boundary $\pp \CC$ of a  bounded domain $\CC \subset \DD$. For $x \in \CC,$ write $\Pp (t,x, y : \CC)$ for the fundamental solution of the heat equation vanishing at $\pp \CC.$ For all positive $F$ with compact support inside $\CC$, we have
$$ \int _{\CC} F(y) \Pp (t,x,y : \CC) \, d\Vol (y) \; = \; \E_x \left[ 1_{t< T_A (\om) } F(\om _t) \right].$$
In particular, for $0\leq \l \leq \l_0, x,y \in \CC, $ $$ G_\l (x,y: \CC ) \; = \; \int _0^\infty  e^{\l t} \Pp (t,x, y : \CC) \, dt.$$
\begin{prop}\label{prop:8.8}[See e.g. \cite{GSC}, Section 2.2] The hitting measure $\varpi _x^\l$ has a density $\rho _x^\l $ with respect to the Lebesgue measure $dy$ on $\pp \CC$ given, for $ y \in \pp \CC , $ by
\[ \rho _x^\l (y) \; = \; \frac{\pp}{\pp n} G_\l (x,z: \CC) |_{z=y} ,\]
where $\displaystyle \frac{\pp}{\pp n} $ denotes the derivative in the  direction of the normal to $\pp \CC.$\footnote{Note that we are looking at the hitting measure of a ball, so we have bounded geometry and \cite{GSC} applies. Note that the relation (\ref{Harnackdensity}) is used in the proof of  Lemma~\ref{derivativeHolder}.}\end{prop}
In particular, the densities   $\rho _x^\l $ are $\l$-harmonic functions of $x \in \CC$ and, by Proposition~\ref{Harnack}, satisfy, if $d(x,\pp\CC) >1$, for all $y\in \pp \CC,$
\begin{equation}\label{Harnackdensity}  \| \nabla_{x'} \log \rho _z^\l (y)|_{x'=x} \|\; \leq \; \log C_0 . \end{equation}

\subsection{Regularity of the hitting distributions} 
In the following propositions, we estimate some regularity of  the hitting distribution  with some geometric hypotheses. Since ``bounded geometry" is used in many different ways, let us define it.
\begin{defi}
We say that a $(m-1)$-dimensional submanifold  $A$ has bounded geometry if, for all $x \in A$,  the set $A \cap B(x,2)$ can be  given in local geodesic coordinates by equations with uniformly bounded $C^2$-coefficients.
\end{defi}
\begin{prop}\label{regularity}Let $A$ be a $(n-1)$ dimensional submanifold   of $\DD$ with bounded geometry. Set $A_1$ for the set of points of $A$ at distance at least $1$  from $\DD^c$. There exists a constant $C_3$ such that for  $\l \in [0, \l _0],$ for any positive   function $F$ on $A_1$, any $x \in \DD$ with $d(x, \DD^c) > 1,$ 
  \[\int _{A_1} F(y)  d\varpi _x^\l (y)  \leq \; C_3 L(F)^2 \int _{A }  G_\l (x,y) F(y)  \, dy ,\] where $L(F) : = e ^{\sup_{A } || \nabla \log F  || }$ is the (multiplicative) Lipschitz constant of $ F$ and $dy$ is the Lebesgue measure on $A$.\end{prop}
\begin{proof} Fix  $\d, 0 < \d \leq 1/2. $ We choose a cover of $A_1 $ by open balls $B (y_p, \d) , y_p \in A_1 $ such that the balls $B (y_p, \d/3) , y_p \in A_1 $ are disjoint and a  partition of unity $\vf _p $ on $A_1 $ subordinate to the cover $B(y_p, \d) \cap A_1 $  of $A_1$. We have to estimate:
\[  \int_{A_1} F(y) d\varpi _x^\l (y)  \leq \; \sum _k \sum _p e^{(k+1 )  \l } \E_x \left[1_{ T(\om) \in [k, k +1)}1_{T(\om) < T_\DD (\om ) } \vf _p ( \om _{T(\om ) }) F(\om _{T(\om )})\right]. \]
Firstly, we estimate from above $F$ on $B(y_p, \d) $  by $L(F) F(y_p)$. 
Then, we write  for all $s , k +2 \leq s < k+3,$ 
\begin{eqnarray*}   &  &\P_x \left[ \om _s \in B(y_p , \d)\right]\; \geq \; \P_x \left[ \om _s \in B(y_p , \d) , s<T_\DD (\om) \right] \\ &\geq & \P_x \left[ \om _s \in B(y_p , \d)  ,  k \leq T(\om ) < k+1 , s < T_\DD (\om) , \om _{T(\om)}\in B (y_p ,\d) \cap A_1\right] \\& \geq & \; \E_x\left[  1_{[k, k+1)} (T(\om)) 1_{B(y_p , \d)\cap A_1} (\om _{T(\om)}) U(y_p,\om_{T(\om) }, s-T(\om ))\right] ,
\end{eqnarray*} where \[ U(y,z,t) \; := \; \P_z\left[  \om _{t } \in B(y , \d), 1\leq t \leq  T_\DD (\om) \right].\]  Here, we used the Strong Markov property to write the second inequality. Set \[ C_{10}^{-1} \; := \; \inf \{U(y,z,t);  y,z \in \DD, d(y,z) \leq \d, d(y, \DD ) >1 , 1\leq t\leq  3\} .\]
 The constant $C_{10} $ is finite by bounded geometry and we have
\[\E_x \left[ 1_{B(y_p , \d)} (\om_s) \right] 
\;  \geq \; C_{10}^{-1} \E_x\left[  1_{[k, k+1)} (T(\om)) 1_{B(y_p , \d)\cap A_1} (\om _{T(\om)}) 1_{T(\om) < T_\DD (\om)}\right] .\]
It follows that \[  e^{(k+1 )  \l } \E_x \left[1_{ T(\om) \in [k, k +1)}1_{ T(\om) < T_\DD (\om) } \vf _p ( \om _{T(\om)})\right] \; \leq C_{10} \int _{k+2}^{k+3} \E_x\left[ e^{\l s }  1_{B(y_p , \d)  }(\om _s) \right]\, ds .\] 
We thus have, by summing over $ k \in \N$, \begin{eqnarray*} \int _{A_1  } F(y)  d\varpi _x^\l (y) 
& \leq &  C_{10} L(F) \sum _p F(y_p) \E_x\left[\int _0 ^\infty  e^{\l s} 1_{ B(y_p , \d)} (\om _s)  \, ds \right]\\
&\leq & C_{10} L(F)   \sum _p F(y_p) \int _{B (y_p, \d)} G_\l (x, w) \, dw \\
&\leq & C_0 C_{10} L(F)   \sum _p F(y_p) G_\l (x, y_p) \Vol (B (y_p, \d)).
\end{eqnarray*}
By bounded geometry and our condition on  the $y_p$s, we can  choose $\d$ small enough and a constant $C_{11} $ such that $\Vol (B (y_p, \d)) \leq C_{11}\int_A \vf _p (y) \, dy   .$   By Proposition~\ref{Harnack} and the Lipschitz regularity of $F$,  we have:
\begin{eqnarray*} \int _{A_1 } F(y)  d\varpi _x^\l (y) & \leq & C_{10}C_{11} C_0^2 L(F)^2  \sum _p  \int _A F( y)  G_\l (x,  y) \vf_p (y) \, dy \\ & =&  C_{10}C_{11}C_0^2  L(F)^2    \int _A F(y)  G_\l (x, y) \, dy .\end{eqnarray*}
The inequality  follows.
\end{proof}

{\begin{prop}\label{regularity2}
Let $\CC$ be an open domain, $\CC \subset \DD, d(\CC , \pp \DD )>1$. Let $x \in \CC,$ and assume that $A := \pp \CC$ has bounded geometry. Let $\varpi _x ^\l $ be the distribution   in (\ref{Poisson}) on $A$. There exists a constant $C_3$ such that if $x \in \CC $ and $ d(x, A ) >1,$ then  for  $\l \in [0, \l _0],$ for any positive   function $F$ on $A$, 
\[ C_3^{-1} (L(F))^{-2} \int _{A }  G_\l (x,y:\DD) F(y)  \, dy  \; \leq \; \int _{A } F(y)  d\varpi  _x^\l (y),\]
where $L(F) : = e ^{\sup_{A } || \nabla \log F  || }$ is the (multiplicative) Lipschitz constant of $ F$ and $dy$ is the Lebesgue measure on $A$.
\end{prop}}
\begin{proof}
The proof  is similar to the proof of Proposition~\ref{regularity}. Fix  $\d, 0 < \d \leq 1/2.$ We  choose a cover of $\pp \CC $ by open balls $B (y_p, \d) , y_p \in \pp \CC$ such that the balls $B (y_p, \d/3)\cap \pp \CC , y_p \in \pp \CC $ are disjoint and we choose a  partition of unity $\vf _p $ on $\pp \CC  $ subordinate to the cover $B(y_p, \d) \cap\pp \CC .$ We write, setting $T(\om) = T_{\pp \CC}(\om)$ and using (\ref{hitting}),
\[  \int_{\pp \CC} F(y) d\varpi _x^\l (y)   = \E_x [  e^{\l T (\om )} F(\om _{T(\om ) })] \geq  \sum _{k\geq 3}  \sum _p e^{k\l } \E_x \left[1_{ T(\om) \in [k, k +1)}\vf _p ( \om _{T(\om ) }) F(\om _{T(\om )})\right]. \]
By bounded geometry, there is $\theta , 0< \theta <1,$ such that  one can choose for each $y_p$ a point $z_p \in \CC$  such that $d(z_p, y_p ) = \d$ and $d(z_p, \pp\CC) > \theta \d .$ Let $B_p \subset \CC$ be the ball of center $z_p $ and radius $\theta \d /2.$  Then we write for all $s, k-3 < s\leq k-2,$
\begin{eqnarray*}  \E_x \left[ 1_{ B_{p} }(\om _s)    1_{ T(\om) \in [k, k +1)}\vf _p ( \om _{T(\om ) })\right]& = &  \E_x \left[ 1_{ B_{p} }(\om _s)   \E_{\om _s}  1_{ T(\om') \in [k-s, k +1-s)}\vf _p ( \om' _{T(\om' ) })\right] \\ &\geq & c_{10} \E_x \left[ 1_{ B_{p} }(\om _s)   \right], \end{eqnarray*} where \[ c_{10} \; := \;   \inf _p \inf _{z \in  B_p, 1 \leq \kappa \leq 4} \E_z [ \vf_p (\om' _{T(\om ')}) 1_{ T(\om') \in (\kappa , \kappa  +1) }] \]
is positive  by bounded geometry and our choice of $\vf_p, B_p$.
It follows that 
\begin{eqnarray*}
&  & \int _{\pp\CC  } F(y)  d\varpi _x^\l (y) \;  \geq \; (L(F))^{-1}\sum _p F(y_p) \sum _{k\geq 3} e^{k\l } \E_x \left[1_{ T(\om) \in [k, k +1)}\vf _p ( \om _{T(\om ) }) \right]\\
&\geq &   (L(F))^{-1}\sum _p F(y_p) \sum _{k\geq 3} e^{k\l } \int_{k-3}^{k-2}\E_x \left[1_{ B_p} (\om _s) 1_{ T(\om) \in [k, k +1)}\vf _p ( \om _{T(\om ) }) \right]\, ds\\
&\geq & c_{10} (L(F))^{-1}\sum _p F(y_p) \sum _{k\geq 3} e^{k\l } \int_{k-3}^{k-2 }\E_x \left[1_{ B_p} (\om _s) \right]\, ds\\
&\geq & c_{10} (L(F))^{-1}\sum _p F(y_p)\int _{ B_p} G_\l (x,z) \, d\Vol (z)\\
& \geq &  c_{10} C_0^{-1} L(F)^{-1} \sum _p F(y_p)  G_\l (x,y_p)  \Vol (B_p) \\
& \geq &  c_{10}C_0^{-2} L(F)^{-2}  c_{13}  \int _{\pp\CC} F(y) G_\l (x, y) \, dy,
\end{eqnarray*}
where  $c_{13}$ is another geometric constant  such that $\Vol (B_p) \geq c_{13}\int_{\pp\CC} \vf _p (y) \, dy  $ for all $p$.
\end{proof}

A priori, the constant $C_3$ depends on the geometries of $A$, and of the manifold, only through the choice of $\d$ and of $C_{10}, C_{11}, c_{10} $ and  $c_{13} $. In particular, the estimates  of Propositions \ref{regularity}  and \ref{regularity2} are  uniform for all the closed sets in the text and we use the same constant $C_3$ when we apply them.

\end{document}